\documentclass[10pt,reqno]{amsart}

\usepackage{amsmath,amssymb,amsfonts,latexsym,amsthm}

\usepackage{mathrsfs}

\numberwithin{equation}{section}
\DeclareMathOperator{\esssup}{ess\,sup}

\usepackage{graphicx}

\usepackage{ifthen} 

\provideboolean{shownotes} 
\setboolean{shownotes}{false} 
\newcommand{\margnote}[1]{
\ifthenelse{\boolean{shownotes}}%
{\marginpar{\raggedright\tiny\texttt{#1}}}%
{}%
}

\newcommand{\hole}[1]{
\ifthenelse{\boolean{shownotes}}%
{\begin{center} \fbox{ \rule {.25cm}{0cm}
\rule[-.1cm]{0cm}{.4cm} \parbox{.85\textwidth}{\begin{center}
\texttt{#1}\end{center}} \rule {.25cm}{0cm}}\end{center}}
{}
}

\theoremstyle{plain}
\newtheorem{defi}{Definition}[section]
\newtheorem{theo}{Theorem}[section]
\newtheorem{lema}{Lemma}

\newtheorem{prop}{Proposition}

\newtheorem{coro}{Corollary}
\newtheorem{rem}{Remark}
\theoremstyle{remark}

\theoremstyle{remark}

\begin{document}

\title[Coupling in the Linearization]{The Cauchy problem for a quasilinear system of equations with coupling in the linearization}
\author[F. Angeles]{Felipe Angeles Garc\'ia}
\date{}
\address{Instituto de Matem\'aticas, Universidad Nacional Aut\'onoma de M\'exico, Circuito Exterior s/n, Ciudad Universitaria, Ciudad de M\'exico, C.P. 04510, M\'exico}
\email{teojkd@ciencias.unam.mx}

\begin{abstract}
The Cauchy problem for a quasilinear system of hyperbolic-parabolic equations is addressed with the method of linearization and fixed point. Coupling between the hyperbolic and parabolic variables is allowed in the linearization and we do not assume the Friedrich's symmetrizability of the system. This coupling results in linear energy estimates that prevent the application of Banach's contraction principle. A metric fixed point theorem is developed in order to conclude the local existence and uniqueness of solutions. We show that the boundedness in the high norm and contraction in the low norm can be incorporated into the formulation of the fixed point by introducing the notion of a closed extension of the solution map. We apply our results to the Cattaneo-Christov system for viscous compressible fluid flow, a system of equations whose inviscid part is not hyperbolic.
\end{abstract}
\keywords{Quasilinear systems; hyperbolic-parabolic coupling; local well-posedness; Cattaneo-Christov systems}
\maketitle

\setcounter{tocdepth}{1}

\section{Introduction}
The present work revisits the local in time well-posedness for the Cauchy problem of a system of hyperbolic-parabolic partial differential equations. This problem has been fairly studied before, for example, in the purely hyperbolic symmetric case one can revise \cite{garding}, \cite{katosym},  \cite{lax}, \cite{maj}, \cite{novo}, \cite{rack} and for the composite hyperbolic-parabolic case we have the standard literature \cite{kawa}, \cite{matsu}, \cite{serre2}, \cite{hudja}. Consider the equation, 
\begin{equation}
A^{0}(U)U_{t}+A^{i}(U)\partial_{i}U-B^{ij}(U)\partial_{i}\partial_{j}U+D(U)U=F(U;D_{x}U),
\label{eq:11}
\end{equation}
where the summation convention has been used. Here $U$ stands as the variable to be determined and is such that $U=U(x,t)\in\mathbb{R}^{N}$ for all $(x,t)\in\mathbb{R}^{d}\times[0,T]$ for some $T>0$ given. The coefficients are matrices of order $N\times N$ dependent on $U$. Equations of the form \eqref{eq:11} often appears in compressible fluid dynamics (see \cite{daf}, \cite{gonzalez}, \cite{lionscom}, \cite{novo} and \cite{serre}, for example).\\
We use the \emph{linearization and fixed point} method. The essence of this method is first to obtain sufficiently strong energy estimates for the Cauchy problem of a linear version of \eqref{eq:11} (see equation \eqref{eq:12}) in order to show its well-posedness. Then, such estimates are used to define a Banach space $Y$, a subset $X\subset Y$ and an operator $\mathcal{T}$ such that $\mathcal{T}:X\rightarrow X$ is well-defined as $\mathcal{T}(U)=V$ where, 
\begin{equation}
A^{0}(U)V_{t}+A^{i}(U)\partial_{i}V-B^{ij}(U)\partial_{i}\partial_{j}V+D(U)V=F(U;D_{x}U)
\label{eq:12}
\end{equation}
is the associated linear system, often refer as the \emph{linearization}. One shows that this operator has a unique fixed point, $U_{\infty}\in X$ that corresponds to a unique local solution of the system \eqref{eq:11} satisfying some given initial condition.\\
There are different ways to approach the existence of the fixed point of $\mathcal{T}$. Altough one can use topological fixed point arguments (see \cite{itaya1} and \cite{itaya2} and \cite{hudja} for example), we take focus in the application of metric fixed point theorems. In this case, it is common to define an iteration that approximates the solution of the initial value problem for \eqref{eq:11}, say $\mathcal{T}(V^{k})=V^{k+1}$. Then, the sequence of real numbers $a_{k}:=\|\mathcal{T}(V^{k+1})-\mathcal{T}(V^{k})\|_{y}$ is considered and it has been reported to satisfy two types of inequalities:
\begin{itemize}
	\item [(1)] There is a constant $0<\alpha<1$ such that $a_{k}\leq\alpha a_{k-1}$ (cf. \cite{katosym}, \cite{kawa} and \cite{lax}).
	\item [(2)] There is a constant $0<\alpha_{1}<1$ such that $a_{k}\leq\alpha_{1} a_{k-1}+\beta_{k}$, where $\left\lbrace\beta_{k}\right\rbrace$ is a sequence chosen with the property that $\sum_{k}\beta_{k}<\infty$ (cf. \cite{maj}, \cite{novo}, \cite{rack} and \cite{serre2}).
\end{itemize}
In the first case, this means that the operator $\mathcal{T}$ is a contraction, so Banach's fixed point theorem yields the existence and uniqueness of the fixed point of $\mathcal{T}$. Meanwhile, in the second case, although $\mathcal{T}$ is not a contraction (but almost), the inequality implies that $\sum_{k}a_{k}<\infty$. Hence, $\left\lbrace\mathcal{T}(V^{k})\right\rbrace$ is a Cauchy sequence in $Y$. Then, a fixed point of $\mathcal{T}$ can be shown to exist.\\
Consider Kawashima's results \cite{kawa}. Here, a split is done in the unknown variable of the linear equation \eqref{eq:12}, $V=(\hat{u},\hat{v})^{\top}$, and it is assumed that the \emph{hyperbolic variable} $\hat{u}$ is \emph{decoupled} from the \emph{parabolic variable} $\hat{v}$. This means that a particular block structure on the matrix coefficients of \eqref{eq:12} is assumed, namely,
\begin{equation}
\label{eq:13}
\begin{aligned}
A^{0}_{1}(u,v)\partial_{t}\hat{u}&+A^{i}_{11}(u,v)\partial_{i}\hat{u}=f_{1}(U,D_{x}v),\\
A^{0}_{2}(u,v)\partial_{t}\hat{v}&-B^{ij}_{22}(u,v)\partial_{i}\partial_{j}\hat{v}=f_{2}(U,D_{x}U),
\end{aligned}
\end{equation}
where $U=(u,v)^{\top}\in X$ is given and the matrix coefficients are such that the equation for $\hat{u}$ is symmetric hyperbolic and the equation for $\hat{v}$ is symmetric strongly parabolic. By making this assumption, we can obtain strong enough energy estimates to assure that the mapping $\mathcal{T}$ can be classified in case $(1)$. Furthermore, as equation \eqref{eq:13} suggests, when dealing with the following quasilinear equation,
\small
\begin{equation}
\label{eq:14}
\begin{aligned}
A^{0}_{1}(u,v)\partial_{t}u+A^{i}_{11}(u,v)\partial_{i}u+A^{i}_{12}(u,v)\partial_{i}v&=g_{1}(U,D_{x}v),\\
A^{0}_{2}(u,v)\partial_{t}v+A^{i}_{21}(u,v)\partial_{i}u+A^{i}_{22}(u,v)\partial_{i}v-B^{ij}(u,v)\partial_{i}\partial_{j}v&=g_{2}(U,D_{x}U),
\end{aligned}
\end{equation}
\normalsize
is enough to consider its linear version as \eqref{eq:13} as long as we take 
\begin{equation}
\begin{aligned}
\label{eq:15}
f_{1}(U,D_{x}v)&=g_{1}(U,D_{x}v)-A^{i}_{12}(u,v)\partial_{i}v\\
f_{2}(U,D_{x}U)&=g_{2}(U,D_{x}U)-A^{i}_{21}(u,v)\partial_{i}u-A^{i}_{22}(u,v)\partial_{i}v.
\end{aligned}
\end{equation}
See \cite{itaya2}, \cite{ka}, \cite{matsu} and \cite{matsuni}, for example. In particular, we say that a linear system of the form \eqref{eq:12} has \emph{coupling} between the hyperbolic and parabolic variables if the spatial derivatives of the parabolic variable appear in the evolution equation for the hyperbolic variable.\\
In \cite{serre2}, Serre treats the local existence of the Cauchy problem for a quasilinear system of the form \eqref{eq:11} (with $D(U)=0$ and $F(U,D_{x}U)=0$) but assumes that the system can be derived from a viscous system of conservation laws that are \emph{entropy-dissipative}. Due to this structure, Serre provides his own \emph{normal form}, which is reminiscent to the symmetrizability of the so called \emph{reduced hyperbolic system} (\cite{serrevis}). Moreover, he manages to enlarge the class of initial data to $H^{s}(\mathbb{R}^{d})$ where $s>d/2+1$ (instead of $s\geq d/2+2$ established in \cite{kawa}). In particular, in Serre's treatment of the linear problem, the dissipative term is written in conservative form, $A^{0}$ and $A^{i}$ are symmetric for all $i=1,..,d$ while $A^{0}$ is positive definite. No decoupling assumption is needed to obtain the linear energy estimates and these are strong enough to classify Serre's fixed point argument in case $(2)$.\\
In this work we assume coupling in the linearization and present a third case, one in which $\left\lbrace a_{k}\right\rbrace$ satisfies the inequality
\begin{itemize}
	\item [(3)] $a_{k}\leq\alpha_{0}\left(a_{k-1}+a_{k-2}\right)$ for some $0<\alpha_{0}<\frac{1}{2}$, for all $k\geq 2$.
\end{itemize}
As we will show in Lemma \ref{optimal}, this is enough to prove that $\sum_{k}a_{k}<\infty$ and ultimately conclude the existence of a unique fixed point of $\mathcal{T}$.\\
It is well known that in this type of problems, $X$ is not necessarily a closed subset of $Y$ (see \cite[Lemmas 2.1 and 2.2]{maj}). Thus, in principle, is not clear that $\{\mathcal{T}(V^{K})\}$ converges to an element of $X$. For this reason, in Theorem \ref{fixedpo}, we only provide the existence and uniqueness of a fixed point for an extension of $\mathcal{T}$, say $\widehat{\mathcal{T}}$. Then, in section \ref{Extensionsection} we show that the range of $\widehat{\mathcal{T}}$ has a small gain in regularity implying that its fixed point actually belongs to $X$. Both this property and the existence of $\widehat{\mathcal{T}}$ are consequence of the existence of energy estimates as it is shown in Theorems \ref{surjective} and \ref{extension} respectively.\\
In particular, the method presented highlights the fact that equation \eqref{eq:12} doesn't require a full Friedrichs' symmetrizer in order to obtain the energy estimates. Moreover, we don't even require the hyperbolicity of the non-diffusive part of equations \eqref{eq:11} and \eqref{eq:12} (that is, formally setting $B^{ij}=0$). We provide an example with such property in section \ref{CattaneoC}.
\section{Motivation}
This work has been motivated by the study of the Cattaneo-Christov system for compressible fluid flow (\cite{amp}, \cite{christov}, \cite{zhan}, \cite{jor}, \cite{stra}, \cite{stra2}, \cite{straughan2}, \cite{stra4}). This system comprises the basic equations in compressible fluid dynamics i.e. the conservation of mass, the balance of momentum and the balance of energy, written in quasilinear form, coupled with Christov's evolution equation for the heat flux, namely 
\begin{equation}
\tau\left[q_{t}+v\cdot\nabla q-q\cdot\nabla v+(\nabla\cdot v)q\right]+q=-\kappa\nabla\theta,\label{eq:17}
\end{equation}
where $q$ is the heat flux, $v$ is the velocity field, $\theta$ is the temperature, $\tau$ is the thermal relaxation time and $\kappa> 0$ is the heat conductivity coefficient. Similar couplings with Cattaneo type models for the heat flux have been studied before in many different settings (see \cite{amp}, \cite{chandrasek}, \cite{chjo}, \cite{kawa2}, \cite{qin}, \cite{stra} and \cite{tara} for example).\\
 Christov's equation \eqref{eq:17} is a frame indifferent formulation of the classical model of Cattaneo (\cite{cat}, \cite{straughan2}) that was developed to correct the infinite speed of propagation of heat predicted by Fourier's heat flux model as well as to agree with the objectivity principle in continumm dynamics (\cite{christov}, \cite{amorro}). This system of equations can be written in the quasilinear form \eqref{eq:11} with $U=(\rho, v, q, \theta)\in\mathbb{R}^{2d+2}$ where $d$ is the space dimension. By introducing $q$ as a dynamical variable, the diffusion term is taken out of the equation for the energy, contrary to the Navier-Fourier-Stokes system. Instead, \eqref{eq:17} introduces a zeroth order term in its quasilinear form. This feature suggested the study of its \emph{strict dissipativity}. In lay terms, this property states that solutions to the linearized problem around equilibrium states show some decay structure (see \cite[Definition 1.1]{ka1}). If the linearization of \eqref{eq:11} has a \emph{Friedrichs symmetrizer}, Kawashima showed that its strict dissipativity is equivalent to require that the dissipation terms do not allow solutions of travelling wave type to be, simultaneously, solutions to the associated hyperbolic system without dissipation (see \cite[Theorem 1.1]{ka1}). In consequence, the strict dissipativity depends on the matrix coefficients of the linear system \eqref{eq:12}. Therefore, when verifying the strict dissipativity of the linearization of \eqref{eq:11} the source term in \eqref{eq:12} only has to contain the linearization of the fully non-linear terms (see \cite[equation 2.11]{amp} and \cite[equation 4.8]{ka1}, for example). Any other type of linearization may not represent the original structure of the coefficients of the quasilinear system \eqref{eq:11}.\\
In one space dimension, the linearization of the Cattaneo-Christov system around constant equilibrium states, was shown to be \emph{strictly dissipative} \cite{amp}. In this case, the existence of a symmetrizer, allows for the equivalence between the Kawashima-Shizuta condition and the strict dissipativity of the linear system. In turn, this implies the existence of global linear decay rates (see \cite[Theorem 5.2 and Corollary 5.3]{amp}). To establish the well-posedness of the Cauchy problem for this linear system, we can obtain the energy estimates by standard fashion. However, we found that for certain space $Y$ they are not adequate to yield an operator $\mathcal{T}$ with either properties, $(1)$ or $(2)$. It will be shown that, as consequence of the \emph{coupling terms} in the linearization, the sequence $\left\lbrace a_{k}\right\rbrace$ associated with $\mathcal{T}$ satisfies that $a_{k}\leq\alpha_{0}a_{k-1}+\gamma_{k}$ for all $k\in\mathbb{N}$ for certain $0<\alpha_{0}<1$ to be chosen. But in this case, contrary to the \emph{contractive} cases $(1)$ and $(2)$, the sequence $\gamma_{k}$ cannot be chosen a priori so that $\sum_{k}\gamma_{k}<\infty$ thus leading to a different metric fixed point argument.\\
In several space dimensions, the inviscid Cattaneo-Christov system is not hyperbolic \cite{angeles2021nonhyperbolicity}, hence, it is not \emph{Friedrichs symmetrizable} (\cite{benzoni}, \cite{serre}). In particular, this means that, even if there is a conservative structure for the Cattaneo-Christov system, there will not exist a convex entropy. As consequence, Serre's results \cite{serre2} are not applicable to this case. For this reasons, the hyperbolicity of the non-viscous version of \eqref{eq:11} is not part of the assumptions treated in this work.
\section{Notation and basic results}
By $L^{2}:=L^{2}(\mathbb{R}^{d})$ we denote the standard Hilbert space of measurable functions such that $|f|^{2}$ is integrable over $\mathbb{R}^{d}$ with the usual norm and inner product denoted as $\|\cdot\|$ and $\langle\cdot,\cdot\rangle$, respectively.\\
$L^{\infty}:=L^{\infty}(\mathbb{R}^{d})$ stands as the space of bounded measurable functions over $\mathbb{R}^{d}$, with the norm $\|f\|_{\infty}=\operatorname{esssup}_{x\in\mathbb{R}^{d}}|f(x)|$.\\
We denote $\mathbb{N}$ as the set of natural numbers and $\mathbb{N}_{0}:=\mathbb{N}\cup\{0\}$. If $\alpha\in\mathbb{N}^{d}_{0}$ then $\alpha=(\alpha_{1},....,\alpha_{d})$ with each $\alpha_{i}\in\mathbb{N}_{0}$, which means that $\alpha$ is a multi-index. In particular, for $\alpha\in\mathbb{N}_{0}^{d}$ we write
\[\partial_{x}^{\alpha}\cdot=\frac{\partial^{|\alpha|}\cdot}{\partial^{\alpha_{1}}x_{1}\cdots\partial^{\alpha_{d}}x_{d}},\]
and for $k\in\mathbb{N}$, $D^{k}_{x}f$ is the set of all partial derivatives $\partial_{x}^{\alpha}f$ for $|\alpha|=k$. We agree that $D^{1}\cdot=\nabla\cdot$ and that $D^{0}f=f$.\\
Let $m\in\mathbb{N}_{0}$. We denote $H^{m}:=H^{m}(\mathbb{R}^{d})$ as the standard Sobolev space
\[H^{m}(\mathbb{R}^{d})=\left\lbrace u\in L^{1}_{loc}(\mathbb{R}^{d}):\partial_{x}^{\alpha}u\in L^{2}(\mathbb{R}^{d}),~\forall~\alpha\in\mathbb{N}_{0}^{d}~\mbox{such that}~|\alpha|\leq m\right\rbrace\]
with the norm $\|f\|_{m}=\left(\sum_{|\alpha|\leq m}\|\partial_{x}^{\alpha}f\|^{2}\right)^{1/2}$. This norm is generated by the inner product in $H^{m}$ defined as $\langle f,g\rangle_{m}=\sum_{|\alpha|\leq m}\langle \partial_{x}^{\alpha}f,\partial_{x}^{\alpha}g\rangle$. Note that $H^{0}=L^{2}$ and $\|f\|_{0}=\|f\|$. We also set $\widehat{H}^{m}:=\widehat{H}^{m}(\mathbb{R}^{d})$ as the Banach space
\[\widehat{H}^{m}(\mathbb{R}^{d})=\left\lbrace u\in L^{\infty}(\mathbb{R}^{d}):\nabla u\in H^{m-1}(\mathbb{R}^{d})\right\rbrace\]
with the norm $\|f\|_{\bar{m}}=\|f\|_{L^{\infty}}+\|\nabla f\|_{m-1}$. When $m=0$ we define $\widehat{H}^{0}=L^{\infty}$ and $\|f\|_{\bar{0}}=\|f\|_{\infty}$. The following results can be found in \cite{cherrier}.
\begin{theo}{\cite{cherrier}}
If $s>\frac{d}{2}$ the space $\widehat{H}^{s}$ is an algebra under point by point multiplications. That is, if $f,g\in\widehat{H}^{s}$ their product $fg$ belongs to $\widehat{H}^{s}$ and 
\begin{equation}
\|fg\|_{\bar{s}}\leq C\|f\|_{\bar{s}}\|g\|_{\bar{s}}\label{eq:31}
\end{equation}
where $C$ is a constant independent of $f$ and $g$. More generally, if $f\in\widehat{H}^{s}$ and $g\in H^{r}$, $0\leq r\leq s$, then $fg\in H^{r}$, and 
\begin{equation}
\|fg\|_{r}\leq C\|f\|_{\bar{s}}\|g\|_{r}\label{eq:32}
\end{equation}
where $C$ is a constant independent of $f$ and $g$.
\end{theo}
\begin{theo}{\cite{cherrier}}
Let $m,n,\mbox{and},k\in\mathbb{N}_{0}$ such that $m\geq k$, $n\geq k$, and $m+n-k>\frac{d}{2}$. Let $f\in H^{m}$ and $g\in H^{n}$. Then, the product $fg\in H^{k}$, and 
\begin{equation}
\|fg\|_{k}\leq C\|f\|_{m}\|g\|_{n},
\label{eq:33}
\end{equation}
with $C$ independent of $f$ and $g$.
\end{theo}
\begin{coro}{\cite{cherrier}}
Let $s,r\in\mathbb{N}_{0}$ such that $s>\frac{d}{2}$ and $0\leq r\leq s$. Then $H^{s}\times H^{r}\hookrightarrow H^{r}$, and for all $f\in H^{s}$ and $g\in H^{r}$
\begin{equation}
\|fg\|_{r}\leq C\|f\|_{s}\|g\|_{r}\label{eq:34}
\end{equation}
where $C$ is a constant independent of $f$ and $g$. In particular, $H^{s}$ is an algebra under point wise multiplication and for all $f,g\in H^{s}$ 
\begin{equation}
\|fg\|_{s}\leq C\|f\|_{s}\|g\|_{s}\label{eq:35}
\end{equation}
in accord with \eqref{eq:33} for $r=s$.
\end{coro}
We define the commutator of two functions $\xi$ and $w$ as
\begin{equation}
G_{\alpha}(\xi,w):=\partial_{x}^{\alpha}(\xi w)-\xi\partial_{x}^{\alpha}w,\label{eq:36}
\end{equation}
and $G_{0}(\xi,w)=0$ for $\alpha=0$. It satisfies the following estimates (for the proof, see \cite{cherrier})
\begin{theo}{\cite{cherrier}}
Let $m,s\in\mathbb{N}$ such that $s>\frac{d}{2}+1$ and $1\leq m\leq s$. Let $\xi\in\widehat{H}^{s}$ and $w\in H^{m}$. Then, for all $\alpha\in\mathbb{N}_{0}^{d}$ with $|\alpha|\leq m$, the commutator $G_{\alpha}(\xi,w)$ belongs to $H^{m-|\alpha|}$ and 
\begin{equation}
\|G_{\alpha}(\xi,w)\|_{m-|\alpha|}\leq C\|\nabla\xi\|_{s-1}\|w\|_{m-1}.\label{eq:37}
\end{equation}
\end{theo}
The following result provides chain rule estimates, (see, \cite{cherrier} and \cite{hudja}, for example)
\begin{theo}
Let $s\geq 1$ be an integer and assume that  $v=(v_{1},..,v_{N})\in\widehat{H}^{s}$ Let $F=F(v)$ be a $C^{\infty}$-function of $v\in\mathbb{R}^{N}$. Then for $1\leq j\leq s$, we have $D_{x}F(v)\in H^{j-1}$ and 
\begin{equation}
\|D_{x}F(v)\|_{j-1}\leq CM(1+\|v\|_{L^{\infty}})^{j-1}\|D_{x}v\|_{j-1},
\label{eq:19}
\end{equation}
where $C$ is a positive constant and $M=\sum_{k=1}^{j}\sup\left\lbrace|D_{v}^{k}F(v)| :v\leq\lambda:=\|v\|_{L^{\infty}}\right\rbrace$.
\end{theo}
Let $(0,T)\subset\mathbb{R}$ be an interval, and $X$ a Banach space with norm $\|\cdot\|_{X}$. The space $L^{p}(0,T;X)$ consists of all strongly measurable function $u:[0,T]\rightarrow X$ such that
\[\|u\|_{L^{p}(0,T;X)}:=\left(\int_{0}^{T}\|u(t)\|_{X}^{p}dt\right)^{1/p}<\infty\]
for $1\leq p<\infty$ and $\|u\|_{L^{\infty}(0,T;X)}:=ess\sup_{0\leq t\leq T}\|u(t)\|_{X}<\infty$. Furthermore, the space $\mathcal{C}([0,T];X)$ comprises all continuous functions $u:[0,T]\rightarrow X$ with $\|u\|_{\mathcal{C}([0,T];X)}:=\max_{0\leq t\leq T}\|u(t)\|_{X}<\infty$.\\
Let $u\in L^{1}(0,T;X)$. We say that $v\in L^{1}(0,T;X)$ is the weak derivative of $u$, written $u_{t}=v$, if it satisfies that 
\[\int_{0}^{T}\phi^{\prime}(t)u(t)dt=-\int_{0}^{T}\phi(t)v(t)dt\]
for all scalar test functions $\phi\in\mathcal{C}_{0}^{\infty}(0,T)$ and we set 
\begin{align*}
\mathcal{C}^{1}([0,T];X):=\left\lbrace u\in\mathcal{C}([0,T];X):u_{t}\in\mathcal{C}([0,T];X)\right\rbrace.
\end{align*}
For the main properties of Banach spaces involving time we refer the reader  to \cite{evans}, \cite{cherrier}, \cite{tuomas}, \cite{yosida}. 
\section{Coupling vs. Uncoupling}
\label{couvsuncou}
In this section we compare the energy estimates obtained from an uncoupled linear system of equations like \eqref{eq:13} with those obtained in the presence of coupling. Through a simple example we readily show that the energy estimates of a linear equation with coupling between hyperbolic and parabolic variables are not suited for the application of Banach's contraction principle.\\
Let $1\leq m\leq s$. Assume that $w$ and $z$ are given scalar functions such that 
\begin{eqnarray}
w&\in&\mathcal{C}([0,T];H^{m})\cap\mathcal{C}^{1}([0,T]; H^{m-1}),\nonumber\\
z&\in&\mathcal{C}([0,T];H^{m})\cap L^{2}(0,T;H^{m+1}),\nonumber
\end{eqnarray}
and consider the one dimensional system of equations 
\begin{eqnarray}
u_{t}+u_{x}&=&f_{1}(w,z,z_{x}),\label{eq:61}\\
v_{t}-v_{xx}&=&f_{2}(w,z,w_{x},z_{x}),\label{eq:62}
\end{eqnarray}
with initial conditions $u_{0}$ and $v_{0}$ respectively. Assume that $f_{1}$, $f_{2}$, $u_{0}$ and $v_{0}$ are smooth enough. In this system the \emph{hyperbolic} variable $u$ is decoupled from the \emph{parabolic} variable $v$. For all $|\alpha|\leq m$, apply the operator $\partial_{x}^{\alpha}\cdot$ to \eqref{eq:61}, multiply by $2\partial_{x}^{\alpha}u$ and add up all the resulting equations, with respect to $\alpha$, to obtain, 
\begin{equation}
\frac{d}{dt}\|u\|_{m}^{2}\leq C\left(\|f_{1}\|_{m}\|u\|_{m}+\|u\|_{m}^{2}\right).
\label{eq:63}
\end{equation}
Using the chain rule and redefining the constant $C$ we find 
\begin{equation}
\|u\|_{m}\frac{d}{dt}\|u\|_{m}\leq C\left(\|f_{1}\|_{m}\|u\|_{m}+\|u\|_{m}^{2}\right),
\label{eq:64}
\end{equation}
which implies that 
\begin{equation}
\frac{d}{dt}\|u\|_{m}\leq C\left(\|f_{1}\|_{m}+\|u\|_{m}\right).
\label{eq:65}
\end{equation}
Then, Gronwall's lemma and H\"older's inequality leads us to
\begin{equation}
\|u\|^{2}_{m}\leq C e^{Ct}\left\lbrace\|u_{0}\|^{2}_{m}+t\left(\int_{0}^{t}\|f_{1}\|^{2}_{m}\right)\right\rbrace.
\label{eq:66}
\end{equation}
Notice the appearance of the factor $t$ in front of $\int_{0}^{t}\|f_{1}\|^{2}_{m}$. Similarly, we have
\begin{align}
\frac{d}{dt}\|v\|_{m}^{2}+\|v_{x}\|_{m}^{2}\leq\|f_{2}\|_{m-1}\|v_{x}\|_{m}\leq\frac{1}{2}\left(\|f_{2}\|_{m-1}^{2}+\|v_{x}\|_{m}^{2}\right),\label{eq:difest}
\end{align}
which leads to the estimate (see \eqref{eq:441}),
\begin{equation}
\|v\|_{m}^{2}+\int_{0}^{t}\|v\|_{m+1}^{2}\leq C e^{Ct}\left\lbrace\|v_{0}\|_{m}^{2}+\int_{0}^{t}\|f_{2}\|_{m-1}^{2}\right\rbrace.
\label{eq:67}
\end{equation}
Assume the existence of positive constants $c_{1}$ and $c_{2}$, such that the following estimates hold,
\begin{eqnarray}
\|f_{1}\|_{s-1}^{2}&\leq&c_{1}\left(\|w\|_{s-1}^{2}+\|z\|_{s-1}^{2}+\|z\|_{s}^{2}\right),\label{eq:69}\\
\|f_{2}\|_{s-2}^{2}&\leq&c_{2}\left(\|w\|_{s-1}^{2}+\|z\|_{s-1}^{2}\right).\label{eq:610}
\end{eqnarray}
Apply \eqref{eq:69}-\eqref{eq:610} into \eqref{eq:66} and \eqref{eq:67} (with $m=s-1$) respectively and combine the estimates to get
\begin{align}
&\sup_{0\leq\tau\leq t}\|(u,v)(\tau)\|_{s-1}^{2}+\int_{0}^{t}\|v(\tau)\|_{s}^{2}d\tau\leq\nonumber\\
&\leq c_{0}e^{Ct}\left\lbrace t\int_{0}^{t}\|(w,z)(\tau)\|_{s-1}^{2}d\tau+t\int_{0}^{t}\|z(\tau)\|_{s}^{2}d\tau+\int_{0}^{t}\|(w,z)(\tau)\|_{s-1}^{2}d\tau\right\rbrace\nonumber\\
&\leq c_{0}e^{Ct}\left\lbrace t^{2}\sup_{0\leq\tau\leq t}\|(w,z)(\tau)\|_{s-1}^{2}+t\int_{0}^{t}\|z(\tau)\|_{s}^{2}d\tau+t\sup_{0\leq\tau\leq t}\|(w,z)(\tau)\|_{s-1}^{2}\right\rbrace\nonumber\\
&\leq c_{0}t(1+t)e^{Ct}\left\lbrace\sup_{0\leq\tau\leq t}\|(w,z)(\tau)\|_{s-1}^{2}+\int_{0}^{t}\|z(\tau)\|_{s}^{2}d\tau\right\rbrace,
\label{eq:611}
\end{align}
where the constant $c_{0}$ is independent of $t$. Then, by taking $t=T_{\ast}\in(0,T)$ sufficiently small we can assure that
\begin{equation}
0<c_{1}T_{\ast}(1+T_{\ast})e^{CT_{\ast}}<1.
\label{eq:612}
\end{equation}
Define $Y$ as the Banach space of vectors $(w,z)$ such that $w\in\mathcal{C}([0,T_{\ast}];H^{s-1})$, $z\in\mathcal{C}([0,T_{\ast}];H^{s-1})\cap L^{2}(0,T;H^{s})$. For the sake of the argument, assume the existence of a closed set $X\subset Y$ such that the mapping 
\begin{align}
\mathcal{T}&:X\rightarrow X\nonumber\\
&(w,z)\mapsto(u,v)\nonumber
\end{align}
is well defined. Then, \eqref{eq:611} and \eqref{eq:612} imply that $\mathcal{T}$ is a contraction. Therefore, by Banach's contraction principle, there is $(u,v)\in X$ such that $\mathcal{T}(u,v)=(u,v)$.\\
Now, instead of \eqref{eq:61}, assume that we are provided with the equation
\begin{equation}
u_{t}+u_{x}+v_{x}=f_{1}(w,z,z_{x}),
\label{eq:613}
\end{equation}
and consider the coupled system of equations \eqref{eq:613}-\eqref{eq:62}. Instead of \eqref{eq:63} we obtain the estimate
\begin{equation}
\frac{d}{dt}\|u\|_{m}^{2}\leq C\left(\|f_{1}\|_{m}\|u\|_{m}+\|u\|_{m}^{2}+\|u\|_{m}\|v_{x}\|_{m}\right).
\label{eq:614}
\end{equation}
If we follow the same steps that led us from \eqref{eq:63} to \eqref{eq:65}, we obtain the inequality
\begin{equation}
\frac{d}{dt}\|u\|_{m}\leq C\left(\|f\|_{m}+\|u\|_{m}+\|v_{x}\|_{m}\right),
\label{eq:615}
\end{equation}
which is undesirable in the sense that, even when integrated and combined with \eqref{eq:67}, the resultant inequality will be unfit to the application of Gronwall's lemma. The reason, is the appearance of the term $\|v_{x}\|_{m}$ in the right hand side of \eqref{eq:615}. Observe that, due to the coercivity of the second order term in \eqref{eq:62}, the term involving $\|v_{x}\|_{m}$ appears in the left hand side of \eqref{eq:67}. Thus, when estimating \eqref{eq:613}, instead of using the chain rule in \eqref{eq:614} (to eliminate one factor $\|u\|_{m}$), we have to apply Cauchy's weighted inequality, to obtain
\begin{align*}
\frac{d}{dt}\|u\|_{m}^{2}\leq C\left(\|f_{1}\|_{m}^{2}+\|u\|_{m}^{2}\right)+\frac{C\|u\|_{m}^{2}}{2\epsilon}+\frac{C\epsilon\|v_{x}\|_{m}^{2}}{2},
\end{align*}
where $\epsilon>0$ is taken so that $1-C\epsilon>0$. Once we integrate and combine this estimate with \eqref{eq:difest}, the resulting inequality is ready for the application of Gronwall's lemma. Although, this is standard procedure, the price to pay from this estrategy is the lack of the factor $t$ in front of $\int_{0}^{t}\|f_{1}\|_{m}^{2}$. In this case, the analog to \eqref{eq:611} turns out to be,
\begin{align}
&\sup_{0\leq\tau\leq t}\|(u,v)(\tau)\|_{s-1}^{2}+\int_{0}^{t}\|v(\tau)\|_{s}^{2}d\tau\leq\nonumber\\
&\leq c_{0}e^{Ct}\left\lbrace \int_{0}^{t}\|(w,z)(\tau)\|_{s-1}^{2}d\tau+\int_{0}^{t}\|z(\tau)\|_{s}^{2}d\tau+\int_{0}^{t}\|(w,z)(\tau)\|_{s-1}^{2}d\tau\right\rbrace\nonumber\\
&\leq c_{1}e^{Ct}\left\lbrace t \sup_{0\leq\tau\leq t}\|(w,z)(\tau)\|_{s-1}^{2}+\int_{0}^{t}\|z(\tau)\|_{s}^{2}d\tau\right\rbrace.
\label{eq:616}
\end{align}
Therefore, no matter how small $t\in[0,T]$ is taken, the mapping $\mathcal{T}$ associated with equations \eqref{eq:613} and \eqref{eq:62}, will not be a contraction. Moreover, as this example shows, this feature is a consequence of the presence of coupling in the linear equations \eqref{eq:613}-\eqref{eq:62}. For this reason, in the following section we develop a fixed point argument that doesn't require the contraction property for $\mathcal{T}$ or even $\mathcal{T}^{2}$.
\section{The fixed point of a Fibonacci contraction}
\begin{lema}
\label{optimal}
Let $\left\lbrace a_{l}\right\rbrace_{l\in\mathbb{N}_{0}}$ be a sequence of non-negative real numbers with the following property: there is $0<\alpha_{0}<\tfrac{1}{2}$ such that 
\begin{equation}
a_{l}\leq\alpha_{0}\left(a_{l-1}+a_{l-2}\right)~\forall~l\geq 2.
\label{eq:21}
\end{equation}
Then, $\sum_{l}a_{l}<\infty$.
\end{lema}
\begin{proof}
Let $m\geq 2$ be any natural number. Adding all the cases in \eqref{eq:21} from $l=2$ until $l=m$ and rearranging terms we get
\begin{equation*}
\sum_{l=2}^{m}a_{l}-\alpha_{0}\sum_{l=2}^{m-1}a_{l}-\alpha_{0}\sum_{l=2}^{m-2}a_{l}\leq\alpha_{0}(2a_{1}+a_{0})\quad\forall~m\geq 2.
\end{equation*}
Since $a_{l}\geq 0$ for all $l\in\mathbb{N}_{0}$ and $0<\alpha_{0}<\tfrac{1}{2}$, the last inequality implies that 
\begin{equation*}
\sum_{l=2}^{m}a_{l}\leq\frac{\alpha_{0}}{1-2\alpha_{0}}(2a_{1}+a_{0})\quad\forall~m\geq 2
\end{equation*}
and the claim follows.
\end{proof}
\begin{theo}[Fixed Point]
\label{fixedpo}
Let $Y$ be a Banach space with norm $\|\cdot\|_{y}$, $X\subset Y$ a non-empty set and $\mathcal{T}:X\rightarrow X$, a given operator. Define $X_{\infty}\subset Y$ as all the vectors $U\in Y$ for which there is a sequence $\left\lbrace U^{k}\right\rbrace\subset X$ and $V\in Y$ such that 
	\begin{equation}
	U^{k}\rightarrow U~\mbox{and}~\mathcal{T}(U^{k})\rightarrow V~\mbox{in}~ Y.
	\label{eq:24}
	\end{equation}
	Assume the following properties for $\mathcal{T}$ are true:
\begin{itemize}
	\item [(i)] $\mathcal{T}$ admits a well defined extension $\widehat{\mathcal{T}}:X_{\infty}\rightarrow Y$ as $\widehat{\mathcal{T}}(U)=V$ for every $U\in X_{\infty}$ and $V$ as in \eqref{eq:24}. That is, we assume that $V$ is independent of the Cauchy sequence in $X$ that converges to $U$.
	\item [(ii)] There is a constant, $0<\alpha_{0}<\frac{1}{2}$, such that, for all $U_{1},U_{2}\in X$ 
	\begin{equation}
	\|\mathcal{T}^{2}(U_{2})-\mathcal{T}^{2}(U_{1})\|_{y}\leq\alpha_{0}\left\lbrace\|\mathcal{T}(U_{2})-\mathcal{T}(U_{1})\|_{y}+\|U_{2}-U_{1}\|_{y}\right\rbrace.
	\label{eq:25}
	\end{equation}
\end{itemize}
Then, $\widehat{\mathcal{T}}(X_{\infty})\subseteq X_{\infty}$, $\widehat{\mathcal{T}}$ satisfies \eqref{eq:25} for all $U_{1},U_{2}\in X_{\infty}$ and there is a unique fixed point $U_{\infty}\in X_{\infty}$ of $\widehat{\mathcal{T}}$, that is, $\widehat{\mathcal{T}}(U_{\infty})=U_{\infty}$.
\end{theo}
\begin{proof}
Step 1. We begin by showing that $\widehat{\mathcal{T}}(X_{\infty})\subseteq X_{\infty}$. If $V\in\widehat{\mathcal{T}}(X_{\infty})$ there is $U\in X_{\infty}$ such that $V=\widehat{\mathcal{T}}(U)$. This means that there is a sequence $\left\lbrace U^{k}\right\rbrace\subset X$ such that $U^{k}\rightarrow U$ and $\mathcal{T}(U^{k}):=V^{k}\rightarrow V$ in $Y$. Since $V^{k}\in X$ for all $k\in\mathbb{N}$ we only have to show that there is a $V^{\ast}\in Y$ such that $\mathcal{T}(V^{k})\rightarrow V^{\ast}$ in Y. According with \eqref{eq:25}, $V^{k}=\mathcal{T}(U^{k})$ and $V^{l}=\mathcal{T}(U^{l})$ satisfy that 
\begin{align*}
\|\mathcal{T}(V^{k})-\mathcal{T}(V^{l})\|_{y}\leq\alpha_{0}\left(\|\mathcal{T}(U^{k})-\mathcal{T}(U^{l})\|_{y}+\|U^{k}-U^{l}\|_{y}\right).
\end{align*}
Therefore, $\left\lbrace\mathcal{T}(V^{k})\right\rbrace$ is a Cauchy sequence in $Y$ and there is $V^{\ast}\in Y$ such that $\mathcal{T}(V^{k})\rightarrow V^{\ast}$ in $Y$. Hence, $V\in X_{\infty}$ and the result follows.\\
In particular, we have shown that if $U\in X_{\infty}$ then $\widehat{\mathcal{T}}^{2}(U)$ is well-defined since $V=\widehat{\mathcal{T}}(U)$ and $\widehat{\mathcal{T}}(V)=V^{\ast}$ imply that $\widehat{\mathcal{T}}^{2}(U)=V^{\ast}$.\\
Step 2. We proceed to show that \eqref{eq:25} is satisfied by $\widehat{\mathcal{T}}$ in $X_{\infty}$. Let $U_{1},U_{2}\in X_{\infty}$. Then, there are sequences $\left\lbrace U_{1}^{k}\right\rbrace,\left\lbrace U_{2}^{k}\right\rbrace\in X$ and vectors $V_{1},V_{2}\in Y$ such that $U_{1}^{k}\rightarrow U_{1}$, $U_{2}^{k}\rightarrow U_{2}$, $V_{1}^{k}:=\mathcal{T}(U_{1}^{k})\rightarrow V_{1}$ and $V_{2}^{k}:=\mathcal{T}(U_{2}^{k})\rightarrow V_{2}$ in $Y$. Moreover, as we previously showed, there are vectors $V_{1}^{\ast},V_{2}^{\ast}\in Y$ such that $\mathcal{T}(V_{1}^{k})\rightarrow V_{1}^{\ast}$ and $\mathcal{T}(V_{2}^{k})\rightarrow V_{2}^{\ast}$ in $Y$. As consequence of \eqref{eq:25}, for each $k\in\mathbb{N}$, we have that
\begin{align*}
\|\mathcal{T}^{2}(U_{1}^{k})-\mathcal{T}^{2}(U_{2}^{k})\|_{y}\leq\alpha_{0}\left\lbrace\|\mathcal{T}(U_{1}^{k})-\mathcal{T}(U_{2}^{k})\|_{y}+\|U_{1}^{k}-U_{2}^{k}\|_{y}\right\rbrace.
\end{align*}
By taking the limit in the last inequality and recalling that $\widehat{\mathcal{T}}(U_{i})=V_{i}$ and $\widehat{\mathcal{T}}^{2}(U_{i})=\widehat{\mathcal{T}}(V_{i})=V_{i}^{\ast}$ for $i=1,2$ we obtain that
\begin{equation}
	\|\widehat{\mathcal{T}}^{2}(U_{2})-\widehat{\mathcal{T}}^{2}(U_{1})\|_{y}\leq\alpha_{0}\left\lbrace\|\widehat{\mathcal{T}}(U_{2})-\widehat{\mathcal{T}}(U_{1})\|_{y}+\|U_{2}-U_{1}\|_{y}\right\rbrace
	\label{eq:26}
	\end{equation}
for all $U_{1},U_{2}\in X_{\infty}$.\\
Step 3. Let $U^{0}\in X$. Define $\mathcal{T}(U^{0})=:V^{0}$, $\mathcal{T}^{2}(U^{0})=\mathcal{T}(V^{0})=:V^{1}$ and in general, for $k\in\mathbb{N}$ such that $k\geq 2$, $\mathcal{T}(V^{k})=:V^{k+1}$. Then, by $(ii)$, for all $k\geq 2$,
\[\|\mathcal{T}(V^{k+1})-\mathcal{T}(V^{k})\|_{y}\leq\alpha_{0}\left\lbrace\|\mathcal{T}(V^{k})-\mathcal{T}(V^{k-1})\|_{y}+\|\mathcal{T}(V^{k-1})-\mathcal{T}(V^{k-2})\|_{y}\right\rbrace.\]
If we set $a_{k}:=\|\mathcal{T}(V^{k+1})-\mathcal{T}(V^{k})\|_{y}$, the previous inequality reads
\[a_{k}\leq\alpha_{0}\left(a_{k-1}+a_{k-2}\right),\]
and by Lemma \ref{optimal}, $\sum_{k}a_{k}<\infty$. This implies that $\left\lbrace \mathcal{T}(V^{k})\right\rbrace$ is a Cauchy sequence in $Y$, and so, there is a $U_{\infty}\in Y$ such that $\mathcal{T}(V^{k})\rightarrow U_{\infty}$. Moreover, since $\mathcal{T}(V^{k})=V^{k+1}$ we also have that $V^{k}\rightarrow U_{\infty}$. Then, by $(i)$, $U_{\infty}\in X_{\infty}$ and $\widehat{\mathcal{T}}(U_{\infty})=U_{\infty}$. This proves the existence of the fixed point. For the uniqueness assume that $U_{\infty},V_{\infty}\in X_{\infty}$ are fixed points of $\widehat{\mathcal{T}}$. Then, \eqref{eq:26} yields 
\[\|U_{\infty}-V_{\infty}\|_{Y}\leq\alpha_{0}\left(\|U_{\infty}-V_{\infty}\|_{y}+\|U_{\infty}-V_{\infty}\|_{y}\right),\]
and since $0<2\alpha_{0}< 1$, it follows that
\[\|U_{\infty}-V_{\infty}\|_{Y}<\|U_{\infty}-V_{\infty}\|_{y}.\]
Hence, it would be impossible that $U_{\infty}\neq V_{\infty}$. This concludes the proof.
\end{proof}
\begin{rem}
Observe that $X\subseteq X_{\infty}\subseteq\overline{X}$. Thus, if $X$ is a closed set in $Y$, $X=X_{\infty}=\overline{X}$, hence $\mathcal{T}=\widehat{\mathcal{T}}$. Then, $\mathcal{T}$ will be a well-defined continuous operator. In such case, if $\mathcal{T}$ satisfies \eqref{eq:25}, there will be a unique fixed point $U\in X$ for $\mathcal{T}$. Nonetheless, when dealing with the local existence and uniqueness of the Cauchy problem of \eqref{eq:11}, the resulting definitions of $Y$ and $X$ usually prevent the latter from being considered a closed subset of $Y$ \textup{(}see \cite{maj}, \cite{katosym} and \cite{lax}, for example\textup{)}. Although, the existence of a unique fixed point for $\mathcal{T}$ in $X$ can be proven, this property is not a consequence of the closedness of $X$ as it is shown in the following sections.
\end{rem}
It is interesting to note that, if $0<\alpha_{0}\leq\frac{1}{6}$, we can provide an alternative way to show that the sequence $\left\lbrace\mathcal{T}(V^{k})\right\rbrace$, defined in step 3 of the proof of Theorem \ref{fixedpo}, is a Cauchy sequence in $Y$. Indeed, first consider the following propositions.
\begin{prop}
\label{fibonacci}
Let $\left\lbrace a_{l}\right\rbrace_{l\in\mathbb{N}_{0}}$ be a sequence of non-negative real numbers satisfying the iteration rule \eqref{eq:21} for some $0<\alpha_{0}<1$. Then
\begin{equation}
a_{2k}\leq\alpha_{0}^{k}\phi_{2k}\beta_{0}\quad\mbox{and}\quad a_{2k+1}\leq\alpha_{0}^{k}\phi_{2k+1}\beta_{0},
\label{eq:22}
\end{equation}
where $\beta_{0}:=a_{0}+a_{1}$ and $\left\lbrace\phi_{k}\right\rbrace_{k=1}^{\infty}$ is Fibonacci's sequence, i.e. $\phi_{1}=1$, $\phi_{2}=1$, $\phi_{3}=2$, $\phi_{4}=3$,.. and in general $\phi_{k}=\phi_{k-1}+\phi_{k-2}$.
\end{prop}
\begin{proof}
Observe that, for $l=2$, the recursion relation in \eqref{eq:21}, yields that $a_{2}\leq\alpha_{0}\beta_{0}$, and since $0<\alpha_{0}<1$, we have that $a_{3}\leq\alpha_{0}\left(\alpha_{0}\beta_{0}+a_{1}\right)\leq\alpha_{0}(2\beta_{0})$. Now assume that the statement is true for the first $l$ natural numbers and consider $a_{l+1}$. Then, \eqref{eq:21} states that $a_{l+1}\leq\alpha_{0}\left(a_{l}+a_{l-1}\right)$, and we have two cases:
\begin{itemize}
	\item [(i)] If $l+1=2k_{0}$ for some $k_{0}\geq 2$, then $l=2(k_{0}-1)+1$ and $l-1=2(k_{0}-1)$ and so, the induction hypothesis yields 
	\begin{eqnarray*}
	a_{2k_{0}}&\leq&\alpha_{0}\left(a_{2(k_{0}-1)+1}+a_{2(k_{0}-1)}\right)\\
	&\leq&\alpha_{0}\left(\alpha_{0}^{k_{0}-1}\phi_{2(k_{0}-1)+1}\beta_{0}+\alpha_{0}^{k_{0}-1}\phi_{2(k_{0}-1)}\beta_{0}\right)\\
	&=&\alpha_{0}^{k_{0}}\beta_{0}\left(\phi_{2(k_{0}-1)+1}+\phi_{2(k_{0}-1)}\right)=\alpha_{0}^{k_{0}}\beta_{0}\phi_{2k_{0}}.
	\end{eqnarray*}
	\item [(ii)] If $l+1=2k_{0}+1$ for some $k_{0}\geq 2$, then $l=2k_{0}$ and $l-1=2(k_{0}-1)+1$, which implies that 
	\begin{eqnarray*}
	a_{2k_{0}+1}&\leq&\alpha_{0}\left(a_{2k_{0}}+a_{2(k_{0}-1)+1}\right)\\
	&\leq&\alpha_{0}\left(\alpha_{0}^{k_{0}}\phi_{2k_{0}}\beta_{0}+\alpha_{0}^{k_{0}-1}\phi_{2k_{0}-1}\beta_{0}\right)\\
	&=&\alpha_{0}^{k_{0}}\left(\alpha_{0}\phi_{2k_{0}}\beta_{0}+\phi_{2k_{0}-1}\beta_{0}\right)\leq\alpha_{0}^{k_{0}}\phi_{2k_{0}+1}\beta_{0}.
	\end{eqnarray*}
\end{itemize}
\end{proof}
\begin{prop}
Let $\left\lbrace a_{l}\right\rbrace_{l\in\mathbb{N}_{0}}$ be a sequence of non-negative real numbers satisfying the iteration rule \eqref{eq:21}, for some $0<\alpha_{0}\leq\frac{1}{6}$. Then, 
\begin{equation}
a_{2k},~a_{2k+1}\leq\frac{\beta_{0}}{2^{k}}\quad\mbox{for all}\quad k\in\mathbb{N}.
\label{eq:fibonaccidecay}
\end{equation}
\end{prop}
\begin{proof}
By induction it is easy to show that, $\phi_{2k},\phi_{2k+1}\leq 3^{k}$ and as consequence, 
\begin{eqnarray*}
\alpha_{0}^{k}\phi_{2k}\beta_{0},~\alpha_{0}^{k}\phi_{2k+1}\beta_{0}\leq\frac{1}{2^{k}}\beta_{0}.
\end{eqnarray*}
Then, by \eqref{eq:22} the result follows.
\end{proof}
Now assume that we want to show that $\left\lbrace \mathcal{T}(V^{k})\right\rbrace$ is a Cauchy sequence in $Y$ and take $m,n\in\mathbb{N}$ such that $m>n$. By the triangle inequality and \eqref{eq:fibonaccidecay}, it follows that 
\begin{align*}
\|\mathcal{T}(V^{m})-\mathcal{T}(V^{n})\|_{y}&\leq a_{m-1}+a_{m-2}+\cdots+a_{n+1}+a_{n}\\
&\leq\beta_{0}\left(\frac{1}{2^{[\tfrac{m-1}{2}]}}+\frac{1}{2^{[\tfrac{m-2}{2}]}}+\cdots+\frac{1}{2^{[\tfrac{n+1}{2}]}}+\frac{1}{2^{[\tfrac{n}{2}]}}\right),
\end{align*}
where $[x]$ denotes the integer part of $x$, that is, the only integer $l=[x]$ such that $l\leq x<l+1$. Then, 
\begin{align*}
\|\mathcal{T}(V^{m})-\mathcal{T}(V^{n})\|_{y}&\leq\frac{\beta_{0}}{2^{[\tfrac{n}{2}]}}\left(\frac{1}{2^{[\tfrac{m-1}{2}]-[\tfrac{n}{2}]}}+\frac{1}{2^{[\tfrac{m-2}{2}]-[\tfrac{n}{2}]}}+\cdots+\frac{1}{2^{[\tfrac{n+1}{2}]-[\tfrac{n}{2}]}}+1\right)\\
&\leq\frac{\beta_{0}}{2^{[\tfrac{n}{2}]}}\left(1+2\sum_{k=0}^{\infty}\frac{1}{2^{k}}\right).
\end{align*}
Given $\epsilon>0$ there is $n_{0}\in\mathbb{N}$ such that, $\beta_{0}\left(1+2\sum_{k=0}^{\infty}\frac{1}{2^{k}}\right)<\epsilon 2^{[\tfrac{n}{2}]}$, for all $n\geq n_{0}$. Therefore, $\left\lbrace\mathcal{T}(V^{k})\right\rbrace$ is a Cauchy sequence in $Y$.
\section{Well-posedness for linear equations}
\subsection{Linear parabolic system} Let us consider the following linear Cauchy problem
\begin{eqnarray}
A^{0}(x,t)u_{t}-B^{ij}(x,t)\partial_{i}\partial_{j}u&=&f(x,t)-A^{i}(x,t)\partial_{i}u-D(x,t)u\label{eq:41}\\
u(x,0)&=&u_{0}(x)\label{eq:42}
\end{eqnarray}
where $x\in\mathbb{R}^{d}$, $t\in[0,T]$ with $T>0$ be given, $u=u(x,t)\in\mathbb{R}^{N}$, $A^{0},B^{ij},A^{i},D$ are square matrices of order $N\times N$ for each $(x,t)\in\mathbb{R}^{d}\times[0,T]$, and $f=f(x,t)\in\mathbb{R}^{N}$ is a given function.\\
In the following, each capital letter $L$ representing the matrix coefficients is understood as the following map $L=L(t)=L(\cdot,t)$. The dependence on $t$ will be made explicit only when is needed. Furthermore, $s\in\mathbb{N}$ is such that  $s\geq s_{0}+1$, $s_{0}=\left[\frac{d}{2}\right]+1$, where $l=\left[\frac{d}{2}\right]$ is the only integer that satisfies that $l\leq\frac{d}{2}<l+1$.
Let us state the following assumptions on the matrix coefficients:
\begin{itemize}
	\item [\textbf{H1}] The matrix $A^{0}$ is symmetric and $B^{ij}=B^{ji}$ for all $i,j=1,...d$. For each $(x,t)\in Q_{T}$, the symbol $\sum_{i,j=1}^{d}B^{ij}(x,t)\omega_{i}\omega_{j}$ is symmetric and we assume the existence of a positive constant $\eta>0$ such that the Legendre-Hadamard ellipticity condition is satisfied, namely
	\[\left(B^{ij}(x,t)\xi_{i}\xi_{j}v,v\right)_{\mathbb{R}^{N}}\geq\eta|\xi|^{2}|v|^{2}\]
	for all $\xi=(\xi_{1},..,\xi_{d})\in\mathbb{R}^{d}$, $v\in\mathbb{R}^{N}$ and $(x,t)\in Q_{T}$. Where $(\cdot,\cdot)_{\mathbb{R}^{N}}$ denotes the inner product in $\mathbb{R}^{N}$ and the index summation convention has been used. Also, there are two positive constants $a_{0}$ and $a_{1}$ such that for all $v\in\mathbb{R}^{N}$,
	\[a_{0}|v|^{2}\leq\left(A^{0}(x,t)v,v\right)_{\mathbb{R}^{N}}\leq a_{1}|v|^{2}\quad\forall~(x,t)\in\mathbb{R}^{d}\times[0,T].\]
	\item [\textbf{H2}] $A^{0},(A^{0})^{-1}\in L^{\infty}(0,T;\widehat{H}^{s})$. In particular, there is a positive constant $g$ such that 
	\[\|A^{0}\|_{L^{\infty}(0,T;\widehat{H}^{s})},\|(A^{0})^{-1}\|_{L^{\infty}(0,T;\widehat{H}^{s})}\leq g.\]
	\item [\textbf{H3}] $\partial_{t}A^{0}\in L^{2}(0,T; H^{s-1})$.
	\item [\textbf{H4}] $A^{i},D\in L^{2}(0,T;\widehat{H}^{s})$ for all $i=1,..,d$.
	\item [\textbf{H5}] $D_{x} B^{ij}\in L^{\infty}(0,T;H^{s-1})$ and $\partial_{t}B^{ij}\in L^{2}(0,T;H^{s-1})$ for all $i,j=1,...,d$.
	\item [\textbf{H6}] $B^{ij}(\cdot,t)$ is uniformly continuous for each $t\in[0,T]$ and $B^{ij}\in L^{\infty}(Q_{T})$ where $Q_{T}:=[0,T]\times\mathbb{R}^{d}$ for all $i,j=1,...,d$. Therefore, in combination with \textbf{H5}, we obtain that $B^{ij}\in L^{\infty}(0,T;\widehat{H}^{s})$. In particular, there is a positive constant $g$ large enough so that 
	\[\sum_{i,j=1}^{d}\|B^{ij}\|_{L^{\infty}(0,T;\widehat{H}^{s})}\leq g.\]
	\end{itemize} 
Let $T>0$ and $1\leq m\leq s$. We define the set
\small
\begin{equation}
\mathcal{P}_{m}(T):=\left\lbrace u\in L^{\infty}(0,T;H^{m}): u_{t}\in L^{2}(0,T;H^{m-1}),~\nabla u\in L^{2}(0,T;H^{m})\label{eq:43}\right\rbrace,
\end{equation}
\normalsize
which is a Banach space with the norm
\begin{equation}
\|u\|_{\mathcal{P}_{m}(T)}^{2}:=\esssup_{0\leq t\leq T}\|u(t)\|^{2}_{m}+\int_{0}^{T}\|u_{t}(t)\|_{m-1}^{2}+\|\nabla u(t)\|_{m}^{2}dt.\label{eq:44}
\end{equation}
By the energy method, we can show the following result (see Appendix A).
\begin{theo}
\label{parabolicestimate}
Let $s\in\mathbb{N}$ such that $s\geq s_{0}+1$, $s_{0}=\left[\frac{d}{2}\right]+1$ and assume \textup{\textbf{H1}}-\textup{\textbf{H6}}. Set
\begin{equation}
\mu_{0}(t):=\sum_{i=1}^{d}\|A^{i}\|_{\bar{s}}^{2}+\sum_{i,j=1}^{d}\|B^{ij}\|_{\bar{s}}^{2}+\sum_{i=1}^{d}\|A^{i}\|_{\bar{s}}+\|D\|_{\bar{s}}^{2}+\|D\|_{\bar{s}}+1\label{eq:413}
\end{equation}
and
\begin{equation}
\mu_{1}(t):=\|\partial_{t}A^{0}\|_{s-1}+\sum_{i,j=1}^{d}\|\partial_{t}B^{ij}\|_{s-1}.
\label{eq:414}
\end{equation}
Let $1\leq m\leq s$, $u_{0}\in H^{m}$ and $f\in L^{2}(0,T;H^{m-1})$. If $u\in\mathcal{P}_{m}(T)$ is a solution of the Cauchy problem \eqref{eq:41}-\eqref{eq:42} then,
\begin{equation}
\|u(t)\|_{m}^{2}+\int_{0}^{t}\left(\|u(\tau)\|_{m+1}^{2}+\|u_{t}(\tau)\|_{m-1}^{2}\right)d\tau\leq J_{0}^{2}(t)\Psi_{0}^{2}(t),\label{eq:441}
\end{equation}
for all $t\in[0,T]$, where 
\begin{equation}
\Psi_{0}^{2}(t):=C_{1}e^{C_{1}\int_{0}^{t}(\mu_{0}(\tau)+\mu_{1}(\tau))d\tau,}
\label{eq:442}
\end{equation}
\begin{equation}
J_{0}^{2}(t):=\|u_{0}\|_{m}^{2}+\int_{0}^{t}\|f(\tau)\|_{m-1}^{2}d\tau,
\label{eq:443}
\end{equation}
and $C_{1}=C_{1}(g,\eta, a_{0})$ is a positive constant.
\end{theo}
The existence and uniqueness of solutions of \eqref{eq:41}-\eqref{eq:42} can be addressed by means of an evolution semigroup approach. By making the new assumptions
\begin{itemize}
	\item [$\mathbf{H2^{\prime}}$] $A^{0},(A^{0})^{-1}\in \mathcal{C}([0,T];\widehat{H}^{s})$. In particular, 
	\[\|A^{0}\|_{\mathcal{C}([0,T];\widehat{H}^{s})},\|(A^{0})^{-1}\|_{\mathcal{C}([0,T];\widehat{H}^{s})}\leq g\]
	\item [$\mathbf{H4^{\prime}}$] $A^{i},D\in \mathcal{C}([0,T];\widehat{H}^{s})$ for all $i=1,..,d$,
	\end{itemize}
	instead of \textbf{H2} and \textbf{H4} respectively, $f\in\mathcal{C}([0,T];H^{m-1})$ and introducing,
	\begin{itemize}
		\item [\textbf{H7}] $\partial_{t}A^{i},\partial_{t}D\in L^{2}(0,T;H^{s-1})$ for all $i=1,..,d$,
	\end{itemize}
is easy to show that, theorem 1 in \cite{kato1} can be satisfied by setting
\begin{align}
X=&L^{2},~U=u,~U_{0}=u_{0},\nonumber\\
A(t)=&(A^{0})^{-1}\left\lbrace B^{ij}(t)\partial_{i}\partial_{j}\cdot-A^{i}(t)\partial_{i}\cdot-D(t)\cdot\right\rbrace,\nonumber\\
D(A(t))=&Y=H^{2},S(t)=\lambda_{0}I-A(t),~B(t)=0,\nonumber
\end{align}
for some $\lambda_{0}$ sufficiently large (see \cite{kawa} and \cite{matsu}). Thus, the existence of a solution, 
\begin{equation}
u\in\mathcal{C}([0,T];H^{2})\cap\mathcal{C}^{1}([0,T];L^{2}),
\label{eq:445}
\end{equation}
for the problem \eqref{eq:41}-\eqref{eq:42} follows. 
\begin{theo}
\label{parabolicwp}
Let $s\in\mathbb{N}$ such that $s\geq s_{0}+1$, $s_{0}=\left[\frac{d}{2}\right]+1$. Assume that \textup{\textbf{H1}}, $\mathbf{H2^{\prime}}$, \textup{\textbf{H3}}, $\mathbf{H4^{\prime}}$, \textup{\textbf{H5}}, \textup{\textbf{H6}}, and \textup{\textbf{H7}} are satisfied. If $u_{0}\in H^{m}$ and $f\in\mathcal{C}(0,T;H^{m-1})$, the Cauchy problem \eqref{eq:41}-\eqref{eq:42} is well-posed in the space $\mathcal{P}_{m}(T)$ and the solution $u\in\mathcal{P}_{m}(T)$ belongs to the space $\mathcal{C}([0,T];H^{m})$ for any $1\leq m\leq s$.
\end{theo}
\begin{proof}
As consequence of \eqref{eq:445}, $u^{\epsilon}\in\mathcal{P}_{m}(T)\cap\mathcal{C}([0,T];H^{m})$, and since $u^{\epsilon}$ satisfies \eqref{eq:444} with $u^{\epsilon}(x,0)=u_{0}^{\epsilon}$, for every $\epsilon_{1},\epsilon_{2}>0$ we have the estimate,
\begin{equation}
\|u^{\epsilon_{1}}-u^{\epsilon_{2}}\|_{\mathcal{P}_{m}(T)}^{2}\leq\Psi_{0}^{2}(T)\left(\|u_{0}^{\epsilon_{1}}-u_{0}^{\epsilon_{2}}\|_{m}^{2}+\int_{0}^{T}\|\bar{f}^{\epsilon_{1}}-\bar{f}^{\epsilon_{2}}\|_{m-1}^{2}dt\right),\nonumber
\end{equation}
where $\bar{f}^{\epsilon}$ comprises all the terms in the right hand side of \eqref{eq:444}. Since $\bar{f}^{\epsilon_{1}}-\bar{f}^{\epsilon_{2}}\rightarrow 0$ in $L^{2}(0,T;H^{m-1})$ and $u_{0}^{\epsilon_{1}}-u_{0}^{\epsilon_{2}}\rightarrow 0$ in $H^{m}$ if $\epsilon_{1},\epsilon_{2}\rightarrow 0$, we conclude that $\left\lbrace u^{\epsilon}\right\rbrace_{\epsilon>0}$ is a Cauchy sequence in $\mathcal{P}_{m}(T)$, and thus, there exists $\bar{u}\in\mathcal{P}_{m}(T)$ such that 
\begin{equation}
u^{\epsilon}\rightarrow\bar{u}~\mbox{in}~\mathcal{P}_{m}(T).\label{eq:446}
\end{equation}
Now, since $u^{\epsilon}\in\mathcal{C}([0,T];H^{m})$, the convergence in \eqref{eq:446} implies convergence in $\mathcal{C}([0,T];H^{m})$ for all $m\geq 2$ and so we conclude that $u=\bar{u}\in\mathcal{C}([0,T];H^{m})$.
\end{proof}
\subsection{Linear system with hyperbolic-parabolic coupling}
Let us consider three sets of variables, all of them functions of $(x,t)\in Q_{T}$ for $T>0$, $u\in\mathbb{R}^{n}$, $v\in\mathbb{R}^{k}$, $w\in\mathbb{R}^{p}$, with $n+k+p=N$, and the following system of partial differential equations,
\begin{eqnarray}
A_{1}^{0}u_{t}+A^{i}_{11}\partial_{i}u+A_{12}^{i}\partial_{i}v&=&f_{1}(x,t), \label{eq:51}\\
A_{2}^{0}v_{t}+A^{i}_{21}\partial_{i}u+A_{22}^{i}\partial_{i}v+A^{i}_{23}\partial_{i}w-B_{0}^{ij}\partial_{i}\partial_{j}v&=&f_{2}(x,t), \label{eq:52}\\
A_{3}^{0}w_{t}+A^{i}_{32}\partial_{i}v+A_{33}^{i}\partial_{i}w+D_{0}w&=&f_{3}(x,t), \label{eq:53}
\end{eqnarray}
where repeated index notation has been used in the space derivatives $\partial_{i}\cdot$ and $\partial_{i}\partial_{j}\cdot$, and where each capital letter represents a real matrix function of $(x,t)\in\mathbb{R}^{d}\times[0,T]$ such that
\begin{eqnarray}
A_{0}^{1}(x,t)\in\mathbb{M}_{n\times n}&,&A_{11}^{i}(x,t)\in\mathbb{M}_{n\times n}~\forall 1\leq i\leq d,\nonumber\\
A_{12}^{i}(x,t)\in\mathbb{M}_{n\times k}~\forall 1\leq i\leq d&,&A_{21}^{i}(x,t)\in\mathbb{M}_{k\times n}~\forall 1\leq i\leq d,\nonumber\\
A_{22}^{i}(x,t)\in\mathbb{M}_{k\times k}~\forall 1\leq i\leq d&,&A_{23}^{i}(x,t)\in\mathbb{M}_{k\times p}~\forall 1\leq i\leq d,\nonumber\\
A_{32}^{i}(x,t)\in\mathbb{M}_{p\times k}~\forall 1\leq i\leq d&,&A_{33}^{i}(x,t)\in\mathbb{M}_{p\times p}~\forall 1\leq i\leq d,\nonumber\\
A_{2}^{0}(x,t)\in\mathbb{M}_{k\times k}&,&A_{0}^{3}(x,t)\in\mathbb{M}_{p\times p}\nonumber\\
B_{0}^{ij}(x,t)\in\mathbb{M}_{k\times k}&,&D_{0}(x,t)\in\mathbb{M}_{p\times p}.\nonumber
\end{eqnarray}
The mapping $(x,t)\mapsto F(x,t):=\left(f_{1}(x,t), f_{2}(x,t),f_{3}(x,t)\right)^{\top}$ is assumed to be given and for each $(x,t)\in Q_{T}$, $f_{1}(x,t)\in\mathbb{R}^{n}$, $f_{2}(x,t)\in\mathbb{R}^{k}$, $f_{3}(x,t)\in\mathbb{R}^{p}$.\\
We make the following assumptions: 
\begin{itemize}
	\item [\textbf{I}] $A_{11}^{i}$ and $A_{33}^{i}$ are symmetric matrices for all $1\leq i\leq d$.
\end{itemize}
Observe that equations \eqref{eq:51}, \eqref{eq:52}, \eqref{eq:53} can be written in the form 
\begin{equation}
A^{0}U_{t}+A^{i}\partial_{i}U-B^{ij}\partial_{i}\partial_{j}U+Du=F \label{eq:54}
\end{equation}
where 
\begin{equation}
A^{0}=\left(\begin{array}{ccc}
	A_{1}^{0}&0&0\\
	0&A_{2}^{0}&0\\
	0&0&A_{3}^{0}\\
\end{array}\right)\in\mathbb{M}_{N\times N},
\label{eq:55}
\end{equation}
\begin{equation}
A^{i}=\left(\begin{array}{ccc}
	A_{11}^{i}&A_{12}^{i}&0\\
	A_{21}^{i}&A_{22}^{i}&A_{23}^{i}\\
	0&A_{32}^{i}&A_{33}^{i}\\
\end{array}\right)\in\mathbb{M}_{N\times N},
\label{eq:56}
\end{equation}
\begin{equation}
B^{ij}=\left(\begin{array}{ccc}
	0&0&0\\
	0&B_{0}^{ij}&0\\
	0&0&0\\
\end{array}\right)\in\mathbb{M}_{N\times N},
\label{eq:57}
\end{equation}
\begin{equation}
D=\left(\begin{array}{ccc}
	0&0&0\\
	0&0&0\\
	0&0&D_{0}\\
\end{array}\right)\in\mathbb{M}_{N\times N}.
\label{eq:58}
\end{equation}
\begin{itemize}
	\item [\textbf{II}] Every non-zero sub-block of \eqref{eq:55}-\eqref{eq:58} satisfy assumptions \textbf{H1}-\textbf{H6}. In particular, there are two positive constants $a_{0}$ and $a_{1}$ such that for all $v\in\mathbb{R}^{n_{i}}$
	\[a_{0}|v|^{2}\leq\left(A^{0}_{i}(x,t)v,v\right)_{\mathbb{R}^{n_{i}}}\leq a_{1}|v|^{2}~\forall~(x,t)\in Q_{T},\]
and all $n_{i}=n,k,p$. Also, there is a positive constant $\eta>0$ such that,
	\[\left(B^{ij}_{0}(x,t)\xi_{i}\xi_{j}v,v\right)_{\mathbb{R}^{k}}\geq\eta|\xi|^{2}|v|^{2},\]
	for all $\xi=(\xi_{1},..,\xi_{d})\in\mathbb{R}^{d}$, $v\in\mathbb{R}^{k}$ and $(x,t)\in\mathbb{R}^{d}\times[0,T]$.
\end{itemize}
The regularity stated in \textbf{II} is enough to deduce the energy estimates of \eqref{eq:51}-\eqref{eq:53}. Whenever we require to assure the existence of more regular solutions we use:
\begin{itemize}
	\item [\textbf{II}$\mathbf{^{\prime}}$]  Every non-zero sub-block of \eqref{eq:55}-\eqref{eq:58} satisfies assumptions \textup{\textbf{H1}}, $\mathbf{H2^{\prime}}$, \textup{\textbf{H3}}, $\mathbf{H4^{\prime}}$, \textup{\textbf{H5}}, \textup{\textbf{H6}} and \textup{\textbf{H7}}.
\end{itemize}
\begin{rem}
Observe that, contrary to the last section, the system \eqref{eq:51}-\eqref{eq:53} is not fully strongly parabolic. Moreover, assumptions \textbf{I}-\textbf{II} are not enough to assure that the system without diffusion, namely,
\[A^{0}U_{t}+A^{i}\partial_{i}U+DU=0,\]
is hyperbolic \textup{(}see \cite{angeles2021nonhyperbolicity}\textup{)}.
\end{rem}
We consider the Cauchy problem for system \eqref{eq:51}-\eqref{eq:53} with initial condition 
\begin{equation}
\left(u(x,0),v(x,0),w(x,0)\right)^{\top}=\left(u_{0},v_{0},w_{0}\right)^{\top}(x)=U_{0}(x)\quad x\in\mathbb{R}^{d},\label{eq:59}
\end{equation}
by assuming: 
\begin{itemize}
	\item [\textbf{III}] $f_{1},f_{3}\in\mathcal{C}([0,T];H^{m-1})\cap L^{2}(0,T; H^{m})$ and $f_{2}\in\mathcal{C}([0,T];H^{m-1})$.
\end{itemize}
We proceed by a vanishing-viscosity approach. First, for $\delta>0$, we introduce the following parabolic regularization,
\begin{equation}
A^{0}U_{t}^{\delta}+A^{i}\partial_{i}U^{\delta}+DU^{\delta}-B^{ij}\partial_{i}\partial_{j}U^{\delta}=f+\delta\Lambda\Delta U^{\delta},\label{eq:512}
\end{equation} 
where $\Lambda$ is a constant matrix of order $N\times N$ given as 
\begin{equation}
\Lambda=\left(\begin{array}{ccc}
	\mathbb{I}_{n\times n}&0&0\\
	0&0&0\\
	0&0&\mathbb{I}_{p\times p}
\end{array}\right)
\label{eq:513}
\end{equation}
and $\mathbb{I}_{n\times n}$ and $\mathbb{I}_{p\times p}$ denote the identity matrices in $\mathbb{M}_{n\times n}$ and $\mathbb{M}_{p\times p}$ respectively. In terms of its components, system \eqref{eq:512} has the form, 
\begin{eqnarray}
A_{1}^{0}u_{t}^{\delta}+A^{i}_{11}\partial_{i}u^{\delta}+A_{12}^{i}\partial_{i}v^{\delta}-\delta\Delta u^{\delta}&=&f_{1}(x,t), \label{eq:514}\\
A_{2}^{0}v_{t}^{\delta}+A^{i}_{21}\partial_{i}u^{\delta}+A_{22}^{i}\partial_{i}v^{\delta}+A^{i}_{23}\partial_{i}w^{\delta}-B_{0}^{ij}\partial_{i}\partial_{j}v^{\delta}&=&f_{2}(x,t), \label{eq:515}\\
A_{3}^{0}w_{t}^{\delta}+A^{i}_{32}\partial_{i}v^{\delta}+A_{33}^{i}\partial_{i}w^{\delta}+D_{0}w^{\delta}-\delta\Delta w^{\delta}&=&f_{3}(x,t). \label{eq:516}
\end{eqnarray}
Second, by taking into account the observations made in Section \ref{couvsuncou}, we provide a suitable energy estimate, independent of $\delta$, in order to conclude the existence of a solution for \eqref{eq:51}-\eqref{eq:53} as a consequence of a compactness argument. The proof of the following result is left in Appendix B.
\begin{theo}
\label{energydec}
 Let $s\in\mathbb{N}$ such that $s\geq s_{0}+1$, $s_{0}=\left[\frac{d}{2}\right]+1$ and $1\leq m\leq s$. Assume the coefficients \eqref{eq:55}-\eqref{eq:58} satisfy \textup{\textbf{I}}-\textup{\textbf{II}}, $f_{1},f_{3}\in L^{2}(0,T; H^{m})$, $f_{2}\in L^{2}(0,T;H^{m-1})$ and $U_{0}\in H^{m}$. Let $U=(u,v,w)^{\top}$, defined on $Q_{T}$, be such that
\begin{equation}
	\label{eq:lowreg}
	\begin{aligned}
u,w&\in L^{\infty}(0,T; H^{m}),\\
v&\in L^{\infty}(0,T;H^{m})\cap L^{2}(0,T; H^{m+1}),\\
u_{t},v_{t},w_{t}&\in L^{2}(0,T;H^{m-1}).
\end{aligned}
\end{equation}
Set $\mathcal{F}_{m}^{2}(f_{1},f_{2},f_{3}):=\|f_{1}\|_{m}^{2}+\|f_{2}\|_{m-1}^{2}+\|f_{3}\|_{m}^{2}$.
If $U$ satisfies the Cauchy problem for \eqref{eq:51}-\eqref{eq:53} with initial condition \eqref{eq:59} then, there is a positive constant $C_{1}=C_{1}(g,\eta,a_{0})$ such that,
\begin{align}
\|\left(u(t),v(t),w(t)\right)\|_{m}^{2}+&\int_{0}^{t}\|\left(u_{t}(\tau),v_{t}(\tau),w_{t}(\tau)\right)\|_{m-1}^{2}d\tau+\nonumber\\
+&\int_{0}^{t}\left(\|u(\tau)\|_{m}^{2}+\|v(\tau)\|_{m+1}^{2}+\|w(\tau)\|_{m}^{2}\right)d\tau\leq K_{0}^{2}(t)\Phi_{0}^{2}(t),\label{eq:543}
\end{align}
for all $t\in[0,T]$, where 
\begin{equation}
K_{0}^{2}(t):=\|u_{0}\|_{m}^{2}+\|v_{0}\|_{m}^{2}+\|w_{0}\|_{m}^{2}+\int_{0}^{t}\mathcal{F}_{m}^{2}(f_{1}(\tau),f_{2}(\tau),f_{3}(\tau))d\tau
\label{eq:537}
\end{equation}
and
\begin{equation}
\Phi_{0}^{2}(t):= C_{1}e^{C_{1}\int_{0}^{t}(\mu_{0}(\tau)+\mu_{1}(\tau))dt}.
\label{eq:538}
\end{equation}
Moreover, for any $0<\delta<1$, the solution $U^{\delta}\in\mathcal{P}_{m}(T)$ of the Cauchy problem for \eqref{eq:514}-\eqref{eq:516}, with initial condition $U^{\delta}(0)=U_{0}$, satisfies the energy estimate given in \eqref{eq:543}, \eqref{eq:537} and \eqref{eq:538} for all $1\leq m\leq s$.
\end{theo}
\begin{theo}
\label{wplinearcoupling}
Let $s\in\mathbb{N}$ such that $s\geq s_{0}+1$, $s_{0}=\left[\frac{d}{2}\right]+1$ and $1\leq m\leq s$. Let \textup{\textbf{I}}, $\mathbf{II^{\prime}}$ and \textup{\textbf{III}} be satisfied and assume that $U_{0}\in H^{m}$. Then, there is a unique solution $U=(u,v,w)^{\top}$ to the Cauchy problem for the equations \eqref{eq:51}-\eqref{eq:53} with initial condition \eqref{eq:59}, such that 
\begin{eqnarray}
u,v,w&\in&\mathcal{C}([0,T];H^{m}),~1\leq m\leq s,\label{eq:567}\\
u_{t},v_{t},w_{t}&\in& L^{2}(0,T;H^{m-1}),~1\leq m\leq s,\label{eq:568}\\
u_{t},w_{t}&\in&\mathcal{C}([0,T];H^{m-1}),~1\leq m\leq s,\label{eq:569}\\
v_{t}&\in&\mathcal{C}([0,T];H^{m-2}),~2\leq m\leq s,\label{eq:570}\\
v&\in&L^{2}(0,T;H^{m+1}),~1\leq m\leq s.\label{eq:571}
\end{eqnarray}
Moreover, by theorem \ref{energydec}, $U$ satisfies the energy estimate defined in \eqref{eq:543}-\eqref{eq:538}.
\end{theo}
\begin{proof}
By Theorem \ref{parabolicwp}, for each $0<\delta<1$, there is a unique solution $U^{\delta}=(u^{\delta},v^{\delta},w^{\delta})^{\top}\in\mathcal{P}_{m}(T)\cap\mathcal{C}\left([0,T];H^{m}\right)$ of the Cauchy problem \eqref{eq:512}-\eqref{eq:59} and thus,
\begin{equation}
\int_{0}^{T}\langle A^{0}U^{\delta}_{t}+A^{i}\partial_{i}U^{\delta}+DU^{\delta}-B^{ij}\partial_{i}\partial_{j}U^{\delta}-\delta\Lambda\Delta U^{\delta},\phi\rangle dt=\int_{0}^{T}\langle f,\phi\rangle\label{eq:integralsol}
\end{equation}
holds true for all $\phi\in L^{2}(0,T;L^{2})$. Because of Theorem \ref{energydec}, $U^{\delta}$ satisfies \eqref{eq:543}, there exists $U=(u,v,w)^{\top}$ and $U^{1}=(u^{1},v^{1},w^{1})^{\top}$ and a sub-sequence, still denoted as $\{U^{\delta}\}_{0<\delta<1}$, such that $u,w\in L^{\infty}(0,T;H^{m})\cap L^{2}(0,T;H^{m})$, $v\in L^{\infty}(0,T;H^{m})\cap L^{2}(0,T;H^{m+1})$, $U^{1}\in L^{2}(0,T; H^{m-1})$ with the property that,
\begin{equation}
\begin{aligned}
(u^{\delta},w^{\delta})&\overset{\ast}{\rightharpoonup}(u,w)~\mbox{in}~L^{\infty}(0,T;H^{m}),\quad (u^{\delta},w^{\delta})\rightharpoonup (u,w)~\mbox{in}~L^{2}(0,T;H^{m}),\\
v^{\delta}&\overset{\ast}{\rightharpoonup}v~\mbox{in}~L^{\infty}(0,T;H^{m}),\quad v^{\delta}\rightharpoonup v~\mbox{in}~L^{2}(0,T;H^{m+1}),\\
&(u_{t}^{\delta},v_{t}^{\delta},w_{t}^{\delta})\rightharpoonup (u^{1},v^{1},w^{1})~\mbox{in}~L^{2}(0,T;H^{m-1}),
\end{aligned}
\label{eq:562}
\end{equation}
for all $1\leq m\leq s$. After using definition 3.3 together with the weak convergence in \eqref{eq:562} we conclude that, $\partial_{t}u=u^{1}$, $\partial_{t}v=v^{1}$, $\partial_{t}w=w^{1}$. Therefore, by taking the limit in \eqref{eq:integralsol} we deduce that $U$ satisfies equations \eqref{eq:51}-\eqref{eq:53} with the initial condition \eqref{eq:59}. Finally, since $U=(u,v,w)^{\top}$ satisfies \eqref{eq:lowreg} and $\mathbf{II^{\prime}}$ implies \textbf{II}, Theorem \ref{energydec} assures that $U$ satisfies \eqref{eq:543}. To  conclude that $U\in\mathcal{C}([0,T];H^{m})$ we proceed as in Theorem \ref{parabolicwp}.
\end{proof}
\section{Invariant sets under iterations}
Let $U=(u,v,w)^{\top}\in\mathbb{R}^{N}$ be a given function of $(x,t)\in Q_{T}$ and consider the following linear equation for the unknown $V=(\hat{u},\hat{v},\hat{w})^{\top}$,
\begin{align}
A_{1}^{0}(U)\hat{u}_{t}+A^{i}_{11}(U)\partial_{i}\hat{u}+A_{12}^{i}(U)\partial_{i}\hat{v}&=f_{1}(U,D_{x}v), \label{eq:71}\\
A_{2}^{0}(U)\hat{v}_{t}+A^{i}_{21}(U)\partial_{i}\hat{u}+A_{22}^{i}(U)\partial_{i}\hat{v}+A^{i}_{23}(U)\partial_{i}\hat{w}&-B_{0}^{ij}(U)\partial_{i}\partial_{j}\hat{v}\nonumber\\
&=f_{2}(U,D_{x}U), \label{eq:72}\\
A_{3}^{0}(U)\hat{w}_{t}+A^{i}_{32}(U)\partial_{i}\hat{v}+A_{33}^{i}(U)\partial_{i}\hat{w}+D_{0}(U)\hat{w}&=f_{3}(U,D_{x}v), \label{eq:73}
\end{align}
where each coefficient, $A_{j}^{0}$, $A_{jl}^{i}$, $B_{0}^{ij}$ and $D_{0}$ represents a matrix of the same order as in section 6, however, in this case, every one of them is a given function of $U$. We have to make the following assumption for the matrix coefficients:
\begin{itemize}
	\item [\textbf{A}] The functions $A_{j}^{0}(U)$ for $j=1,2,3$; $A_{jl}^{i}(U)$ for $j,l=1,2,3$ and $i=1,..,d$; $B_{0}^{ij}(U)$, $i,j=1,..,d$ and $D_{0}(U)$, are sufficiently smooth functions of its argument $U\in\mathbb{R}^{N}$.
	\item [\textbf{B}] $A_{j}^{0}(U)$ for $j=1,2,3$ are real symmetric and positive definite, uniformly in each compact set with respect to $U\in\mathbb{R}^{N}$.
	\item [\textbf{C}] $A_{11}^{i}(U)\in\mathbb{M}_{n\times n}$ and $A_{33}^{i}(U)\in\mathbb{M}_{p\times p}$ are symmetric for $U\in\mathbb{R}^{N}$.
	\item [\textbf{D}]  The functions $B^{ij}_{0}(U)$ are real symmetric and satisfy $B_{0}^{ij}(U)=B_{0}^{ji}(U)$ for all $U\in\mathbb{R}^{N}$; $\sum_{i,j=1}^{d}B_{0}^{ij}(U)\omega_{i}\omega_{j}$ is real symmetric and positive definite in each compact set with respect to $U\in\mathbb{R}^{N}$ for all $\xi=(\xi_{1},..,\xi_{d})\in\mathbb{S}^{d-1}$.
	\item [\textbf{E}] For every $U\in\mathbb{R}^{N}$, $f_{1}(U,D_{x}v)\in\mathbb{R}^{n},\quad f_{2}(U,D_{x}U)\in\mathbb{R}^{k},\quad f_{3}(U,D_{x}v)\in\mathbb{R}^{p}$. Let $\eta\in\mathbb{R}^{kd}$ and $\xi\in\mathbb{R}^{Nd}$. The functions $f_{1}(U,\eta)$, $f_{2}(U,\xi)$ and $f_{3}(U,\eta)$  are sufficiently smooth in $(U,\eta)\in\mathcal{O}\times\mathbb{R}^{kd}$ and $(U,\xi)\in\mathcal{O}\times\mathbb{R}^{Nd}$, respectively and satisfy that $f_{1}(U,0)=0$, $f_{2}(U,0)=0$, $f_{3}(U,0)=0$, for any $U\in\mathcal{O}\subset\mathbb{R}^{N}$, where $\mathcal{O}\subset\mathbb{R}^{N}$ is an open convex set contained in $\mathbb{R}^{N}$.
\end{itemize}
We provide the system \eqref{eq:71}-\eqref{eq:73} with the initial condition \eqref{eq:59} and assume that
\begin{itemize}
	\item [\textbf{F}] $U_{0}\in H^{s}$ and $(u_{0},v_{0},w_{0})(x)\in\mathcal{O}_{g_{0}}$, where $\mathcal{O}_{g_{0}}$ is a bounded open convex set $\mathcal{O}_{g_{0}}$ in $\mathbb{R}^{N}$ such that $\overline{\mathcal{O}_{g_{0}}}\subset\mathcal{O}$.
\end{itemize}
\begin{defi}
Let $s\in\mathbb{N}$ such that $s\geq s_{0}+1$, $s_{0}=\left[\frac{d}{2}\right]+1$. We denote by $X_{T}^{s}(g_{2},M)$  the set of functions  $(u,v,w)^{\top}(x,t)=:U(x,t)\in\mathcal{O}$ satisfying \eqref{eq:567}-\eqref{eq:571} with $m=s$ and such that: there is a bounded open convex set $\mathcal{O}_{g_{2}}$ in $\mathbb{R}^{N}$, $\overline{\mathcal{O}}_{g_{2}}\subset\mathcal{O}$, 
\begin{equation}
U(x,t)\in\mathcal{O}_{g_{2}}~\forall(x,t)\in Q_{T},
\label{eq:79}
\end{equation}
and there is a positive constant $M$ such that
\begin{equation}
\begin{aligned}
\sup_{0\leq\tau\leq t}\|(u,v,w)(\tau)\|_{s}^{2}&+\int_{0}^{t}\|v(\tau)\|_{s+1}^{2}d\tau+\\
&+\int_{0}^{t}\|\left(u_{t}(\tau),v_{t}(\tau),w_{t}(\tau)\right)\|_{s-1}^{2}d\tau\leq M^{2}
\end{aligned}
\label{eq:710}
\end{equation}
for all $t\in[0,T]$.
\end{defi}
In this section we use the energy estimate \eqref{eq:543} and closely follow the steps lay down by Kawashima in \cite[Section 2.4]{kawa}. Our objective is to determine $\mathcal{O}_{g_{2}}$, $M$, and $T_{0}>0$ ($\leq T$) so that for $(u,v,w)\in X_{T_{0}}^{s}(g_{2},M)$, the Cauchy problem for \eqref{eq:71}-\eqref{eq:73} with initial condition \eqref{eq:59}  has a unique solution $(\hat{u},\hat{v},\hat{w})$ in the same set $X_{T_{0}}^{s}(g_{2},M)$. That is, $X_{T_{0}}^{s}(g_{2},M)$ is invariant under the mapping defined by $(u,v,w)\mapsto(\hat{u},\hat{v},\hat{w})$.\\
In the following, whenever a constant $C$ depends on the values of $U(x,t)\in\mathcal{O}_{g_{2}}$, we will simply write $C=C(g_{2})$ and $C_{F}$ will denote a constant depending on the function $F$.
\begin{lema}
\label{importantlemma}
Let  $U\in X_{T}^{s}(g_{2},M)$. Assume that $h=h(U)$ and $F=F(U,D_{x}U)$ are smooth functions of their arguments and that $F(V,0)=0$ for all $V\in\mathbb{R}^{N}$.
\begin{itemize}
	\item [(i)]  Set $U^{0}=U(t_{0})$ and $U^{1}=U(t_{1})$ for some $t_{0},t_{1}\in[0,T]$. Then, 
	\[\partial_{x}^{\alpha}h(U^{1})-\partial_{x}^{\alpha}h(U^{0})=\int_{0}^{1}\partial_{x}^{\alpha}\left\lbrace Dh(U_{r}))(U^{1}-U^{0})\right\rbrace dr,\]
for all $0\leq|\alpha|\leq s$, where $U_{r}=U^{0}+r(U^{1}-U^{0})$ for $r\in[0,1]$.
\item [(ii)] $h(U)\in \mathcal{C}([0,T];\widehat{H}^{s})$ and there is a positive constant $C_{h}(g_{2},M)$ (indepdendent of $t$) such that 
\begin{align}
\|h(U)\|_{\bar{s}}\leq C_{h}(g_{2},M).
\label{eq:711}
\end{align}
\item [(iii)] The mapping $\mathbb{R}^{d}\times\{t\}\ni (x,t) \mapsto h(U(x,t))$ is uniformly continuous for each $t\in[0,T]$.
\item [(iv)] $\partial_{t}h(U)\in L^{2}(0,T; H^{s-1})$.
\item [(v)] $F\in\mathcal{C}\left([0,T];H^{s-1}\right)$.
\end{itemize}
\end{lema}
\begin{proof}
Observe that 
\[h(U^{1})-h(U^{0})=\int_{0}^{1}Dh(U_{r})(U^{1}-U^{0})dr\]
and let us show that 
\[\partial_{i}h(U^{1})-\partial_{i}h(U^{0})=\int_{0}^{1}\partial_{i}\left\lbrace Dh(U_{r})(U^{1}-U^{0})\right\rbrace dr\]
for every $i=1,...,d$. First, notice that $Dh(U_{r})\in \widehat{H}^{s}$. Indeed, since $\mathcal{O}_{g_{2}}$ is a convex set, $U_{r}\in\mathcal{O}_{g_{2}}$ and the values of $U_{r}=U_{r}(x,t)$ for $(x,t)\in Q_{T}$ are contained in a compact set, implying that $Dh(U_{r})\in L^{\infty}$. By using theorem 3.4 (with $j=s$ and $F=Dh$) we conclude that $D_{x}Dh(U_{r})\in H^{s-1}$. Then, for a test function $\phi\in\mathcal{D}(\mathbb{R}^{d})$, Fubini's theorem states that
\begin{align*}
\int_{\mathbb{R}^{d}}\left(h(U^{1})-h(U^{0})\right)\partial_{i}\phi dx&=\int_{0}^{1}\left(\int_{\mathbb{R}^{d}}Dh(U_{r})(U^{1}-U^{0})\partial_{i}\phi dx\right)dr,\\
&=-\int_{\mathbb{R}^{d}}\int_{0}^{1}\partial_{i}\left(Dh(U_{r})(U^{1}-U^{0})\right)dr\phi dx
\end{align*}
which proves the statement for the case $|\alpha|=1$. The case $2\leq|\alpha|\leq s$ can be done similarly. This concludes $(i)$.\\
By theorem 3.4, $h(U)\in\widehat{H}^{s}$ and there is a constant $C_{h}=C_{h}(g_{2})$ such that
\begin{align*}
\|h(U)\|_{\bar{s}}=\|h(U)\|_{L^{\infty}}+\|D_{x}h(U)\|_{s-1}\leq C_{h}(g_{2})\left(1+\|D_{x}U\|_{s-1}\right).
\end{align*} 
Hence, \eqref{eq:711} follows by \eqref{eq:710}. Let $|\alpha|=0$ in $(i)$, and apply Sobolev's embedding theorem to obtain that,
\begin{equation*}
|h(U^{1})-h(U^{0})|\leq C(g_{2})|U_{1}-U_{0}|\leq C(g_{2})\kappa_{s}\|U^{1}-U^{0}\|_{s},
\end{equation*}
which goes to zero as $t_{0}\rightarrow t_{1}$ because $U\in\mathcal{C}([0,T];H^{s})$. It remains to prove that $D_{x}h(U^{0})\rightarrow D_{x}h(U^{1})$ with respect to the $H^{s-1}$ norm. Let $\alpha\in\mathbb{N}_{0}^{d}$ such that $\alpha=\alpha^{\prime}+\alpha^{\prime\prime}$, $|\alpha^{\prime}|=1$ and $0\leq \alpha^{\prime\prime}\leq s-1$. Apply Jensen's inequality in $(i)$, integrate in $\mathbb{R}^{d}$ and use Fubini's theorem to get
\[\|\partial_{x}^{\alpha^{\prime\prime}}\left(\partial_{x}^{\alpha^{\prime}}h(U^{1})-\partial_{x}^{\alpha^{\prime}}h(U^{0})\right)\|^{2}\leq\int_{0}^{1}\|\partial_{x}^{\alpha}\left\lbrace Dh(U_{r})(U^{1}-U^{0})\right\rbrace\|^{2} dr.\]
Hence,
\[\|\partial_{x}^{\alpha^{\prime}}h(U^{1})-\partial_{x}^{\alpha^{\prime}}h(U^{0})\|^{2}_{s-1}\leq\int_{0}^{1}\|Dh(U_{r}))\|_{\bar{s}}^{2}\|(U^{1}-U^{0})\|_{s}^{2}dr.\]
By Theorem 3.4, there is a positive constant $C_{Dh}=C_{Dh}(g_{2},M)$, independent of $r\in[0,1]$, such that 
\[\|\partial_{x}^{\alpha^{\prime}}h(U^{1})-\partial_{x}^{\alpha^{\prime}}h(U^{0})\|^{2}_{s-1}\leq C_{Dh}(g_{2},M)\|(U^{1}-U^{0})\|_{s}^{2}\quad\forall~|\alpha^{\prime}|=1.\]
Then, $(ii)$ follows.\\
To prove $(iii)$ notice that, by Sobolev's embedding theorem, $U(t)\in H^{s}$ for each $t\in[0,T]$, and since $h=h(U)$ is continuous on the compact set $\overline{\mathcal{O}_{g_{2}}}$, it follows that $h(U(t))$ is uniformly continuous on $\mathbb{R}^{d}$.\\
Given that $U_{t}\in L^{2}(0,T; H^{s-1})$, $(3.2)$ in theorem 3.1 ($r=s-1$) combined with theorem 3.4 yield
\begin{align*}
\int_{0}^{T}\|\partial_{t}h(U(t))\|_{s-1}^{2}dt\leq C_{Dh}(g_{2},M)^{2}\int_{0}^{T}\|U_{t}(t)\|_{s-1}^{2}dt<\infty,
\end{align*}
which proves $(iv)$.\\
Finally, since $F(U,0)=0$ for any $U\in\mathbb{R}^{N}$, we can proceed just as in $(i)$ to conclude that
\begin{eqnarray}
\|F(U,D_{x}U)\|^{2}&\leq&C\|(0,D_{x}U)\|_{s-1}^{2}\int_{0}^{1}\|DF(U,rD_{x}U)\|_{L^{\infty}}^{2}dr.\label{eq:717}\\
&\leq&C_{F}(g_{2},M)M<\infty.\nonumber
\end{eqnarray}
Now, since we can take $\|F(U,D_{x}U)\|_{s-1}=\|F(U,D_{x}U)\|+\|D_{x}F(U,D_{x}U)\|_{s-2}$, and $\|D_{x}F(U,D_{x}U)\|_{s-2}\leq C(g_{2},M)<\infty$, as consequence of theorem 3.4 (by taking $j=s-1$), we conclude that $F\in H^{s-1}$. In order to prove the continuity of $t\mapsto F=F(U(t);D_{x}U(t))$ with respect of the $\|\cdot\|_{s-1}$, we proceed as in $(ii)$
\end{proof}
\begin{lema}
\label{existencelemma}
Assume that $U\in X_{T}^{s}(g_{2},M)$ and that conditions \textup{\textbf{A}}-\textup{\textbf{F}} are satisfied. Then the Cauchy problem for the linear system \eqref{eq:71}-\eqref{eq:73} with initial condition \eqref{eq:59} has a unique solution $V=(\hat{u},\hat{v},\hat{w})^{\top}$ satisfying \eqref{eq:567}-\eqref{eq:571} with $m=s$ and the energy estimate \eqref{eq:543}. Furthermore, there is a positive constant $C_{2}=C_{2}(g,M)$ such that, 
\begin{equation}
\Phi_{0}^{2}(t)\leq C_{1}e^{C_{2}\left(t+Mt^{1/2}\right)}=:\Phi_{1}^{2}(t),
\label{eq:720}
\end{equation}
for all $t\in[0,T]$.
\end{lema}
\begin{proof}
Since $U=U(x,t)$, the coefficients in \eqref{eq:71}-\eqref{eq:73} are also functions of $(x,t)\in Q_{T}$. Then, the result follows from Theorem \ref{wplinearcoupling}, with $m=s$, once assumptions \textup{\textbf{I}}, $\mathbf{II^{\prime}}$ and \textup{\textbf{III}} are verified. Observe that assumptions \textbf{I} is a direct consequence of \textbf{C}. Assumption $\mathbf{II^{\prime}}$ can be decomposed into several conditions for the block matrices. In this case, notice that $\mathbf{H4^{\prime}}$, \textbf{H3}, \textbf{H5}, and \textbf{H7}  are a direct consequence of Lemma \ref{importantlemma}. Meanwhile, \textbf{H1} follows immediately by assumptions \textbf{B} and \textbf{D}. Indeed, since $U\in\mathcal{O}_{g_{2}}$, condition \textbf{B} assures that, for each $j=1,2,3$, there is a positive constant $a_{0}^{j}=a_{0}^{j}(g_{2})$ independent of $(x,t)$ such that
	\begin{equation}
	a_{0}^{j}|y^{j}|^{2}\leq(A_{j}^{0}(U)y^{j},y^{j})_{\mathbb{R}^{N(j)}}~\forall y^{j}\in\mathbb{R}^{N(j)},\nonumber
	\end{equation}
 where $N(1)=n$, $N(2)=k$ and $N(3)=p$. Define, $\sup_{Q_{T}}|A_{j}^{0}(U(x,t))|=:a_{1}^{j}(g_{2})<\infty$ for each $j=1,2,3$ and thus,
\begin{equation}
	(A_{j}^{0}(U)y^{j},y^{j})_{\mathbb{R}^{N(j)}}\leq a_{1}^{j}|y^{j}|^{2}~\forall y^{j}\in\mathbb{R}^{N(j)}.\nonumber
	\end{equation} 
	By taking $a_{0}(g_{2})=\min_{j}\{ a_{0}^{j}\}$ and $a_{1}(g_{2})=\max_{j}\{ a_{0}^{j}\}$ we obtain that 
	\begin{equation}
	a_{0}|y^{j}|^{2}\leq(A_{j}^{0}(U)y^{j},y^{j})_{\mathbb{R}^{N(j)}}\leq a_{1}|y^{j}|^{2}~\forall y^{j}\in\mathbb{R}^{N(j)},\label{eq:714}
	\end{equation}
	for all $j=1,2,3$. A similar argument leads us to conclude that the symbol $B^{ij}_{0}(U(x,t))\omega_{i}\omega_{j}$ is positive definite for all $\omega=(\omega_{1},..,\omega_{d})\in\mathbb{S}^{d-1}$, with constant $\eta=\eta(g_{2})>0$ independently of $(x,t)\in Q_{T}$.\\
Statement $\mathbf{H2^{\prime}}$ is a consequence of \textbf{A} and \eqref{eq:714}. In particular, according to \eqref{eq:711}, we take $g=g(g_{2},M)$ such that 
\begin{align}
C_{A^{0}_{j}}(g_{2},M),C_{(A^{0}_{j})^{-1}}(g_{2},M),\sum_{i,j=1}^{d}C_{B^{ij}}(g_{2},M)\leq g.
\label{eq:715}
\end{align}
This concludes all the statements comprising $\mathbf{II^{\prime}}$. We are left with verifying \textbf{III}. By \textbf{E}, we can take $F(U,D_{x}U)=\left(f_{1}(U, D_{x}v),f_{2}(U, D_{x}U),f_{3}(U, D_{x}v)\right)^{\top}$ in Lemma \ref{importantlemma} and thus, $f_{j}\in\mathcal{C}([0,T];H^{s-1})$ ($j=1,2,3$). To show that $f_{l}(U,D_{x}v)\in L^{2}(0,T;H^{s})$ for $l=1,3$, apply \eqref{eq:717}, so that, 
\[\int_{0}^{T}\|f_{l}(U(t),D_{x}v(t))\|^{2}\leq C_{F}(g_{2},M)\int_{0}^{T}\|v(t)\|_{s}^{2}dt,\]
and use theorem 3.4 with $j=s$ to obtain,
\[\int_{0}^{T}\|D_{x}f_{l}(U(t),D_{x}v(t))\|_{s-1}^{2}\leq C(g_{2},M)\int_{0}^{T}\|U(t)\|_{s}^{2}+\|v(t)\|_{s+1}^{2}dt<\infty.\]
This proves \textbf{III}. Finally, by Lemma \ref{importantlemma}, there are two positive constants $K_{1}=K_{1}(g_{2},M)$ and $K_{2}=K_{2}(g_{2},M)$ such that 
\begin{align*}
\int_{0}^{t}\mu_{0}(\tau)d\tau\leq K_{1}t,\quad\int_{0}^{t}\mu_{1}(\tau)d\tau\leq K_{2}t^{1/2}\left(\int_{0}^{t}\|U_{t}(\tau)\|_{s-1}^{2}d\tau\right)^{1/2}.
\end{align*}
Then, \eqref{eq:720} follows by \eqref{eq:538} with $C_{2}:=C_{1}\max\left(K_{1},K_{2}\right)$.
\end{proof}
Now, fix a constant $g_{2}>0$ so that $0<g_{2}<g_{1}:=d(\mathcal{O}_{g_{0}},\partial\mathcal{O})$ and take
\begin{eqnarray}
\mathcal{O}_{g_{2}}&=&g_{2}-\mbox{neighborhood of}~\mathcal{O}_{g_{0}},\label{eq:722}\\
M&=&2\sqrt{C_{1}}\|(u_{0},v_{0},w_{0})\|_{s}.\label{eq:723}
\end{eqnarray}
We are ready for the main result of this section.
\begin{theo}
\label{invariant}
There is a positive constant $T_{0}$ that depends on $g_{0}$, $g_{2}$ and $\|U_{0}\|_{s}$ such that, if $U=(u,v,w)^{\top}\in X_{T_{0}}^{s}(g_{2},M,)$ with $g_{2}$ and $M$ defined by \eqref{eq:722}-\eqref{eq:723}, the Cauchy problem for \eqref{eq:71}-\eqref{eq:73} with initial condition \eqref{eq:59} has a unique solution $V=(\hat{u},\hat{v},\hat{w})^{\top}$ in the same set $X_{T_{0}}^{s}(g_{2},M)$.
\end{theo}
\begin{proof}
By  Lemma \ref{existencelemma}, there is a unique solution $V=(\hat{u},\hat{v},\hat{w})^{\top}$ to \eqref{eq:71}-\eqref{eq:73} with initial condition \eqref{eq:59} such that, the estimate
\begin{align}
&\|V(t)\|_{s}^{2}+\int_{0}^{t}\|\hat{v}(\tau)\|_{s+1}^{2}d\tau+\int_{0}^{t}\|V_{t}(\tau)\|_{s-1}^{2}d\tau\nonumber\\
&\leq C_{1}e^{C_{2}\left(t+Mt^{1/2}\right)}\left\lbrace\|U_{0}\|_{s}^{2}+\int_{0}^{t}\mathcal{F}^{s}(f_{1},f_{2},f_{3})(\tau)d\tau\right\rbrace\label{eq:724}
\end{align}
holds for all $t\in[0,T]$. Since $\mathcal{F}^{s}(f_{1},f_{2},f_{3})\in L^{1}(0,T)$, we can choose $0<T_{0}<T$ such that 
\begin{equation}
\int_{0}^{T_{0}}\mathcal{F}^{s}(f_{1},f_{2},f_{3})(t)dt\leq \|U_{0}\|_{s}^{2}\quad\mbox{and}\quad e^{C_{2}\left(T_{0}+MT^{1/2}_{0}\right)}\leq 2,
\label{eq:725}
\end{equation}
happen simultaneously. Then, from \eqref{eq:724},
\[\|\widehat{U}(t)\|_{s}^{2}+\int_{0}^{t}\|\hat{v}(\tau)\|_{s+1}^{2}d\tau+\int_{0}^{t}\|\widehat{U}_{t}(\tau)\|_{s-1}^{2}d\tau\leq 4C_{1}\|(u_{0},v_{0},w_{0})\|_{s}^{2}=M^{2}.\]
The last estimate implies that,
\begin{align}
|(\hat{u},\hat{v},\hat{w})(x,t)-(u_{0},v_{0},w_{0})(x)|&\leq\int_{0}^{t}\|\partial_{t}(\hat{u},\hat{v},\hat{w})(\tau)\|_{L^{\infty}}d\tau\nonumber\\
&\leq\kappa_{s-1}\int_{0}^{t}\|\widehat{U}_{t}(\tau)\|_{s-1}d\tau\nonumber\\
&\leq\kappa_{s-1}t^{1/2}M\leq\kappa_{s-1}T_{0}^{1/2}M\label{eq:727}.
\end{align}
If $T_{0}$ is such that 
\begin{equation}
\kappa_{s-1}T_{0}^{1/2}M\leq g_{2},
\label{eq:728}
\end{equation}
then, from \eqref{eq:728}, \eqref{eq:79} is satisfied. Summarizing, we have to choose a positive $T_{0}$ that satisfies \eqref{eq:725} and \eqref{eq:728} to assure that $V\in X_{T_{0}}^{s}(g_{2},M)$ whenever $U\in X_{T_{0}}^{s}(g_{2},M)$. This completes the proof of Theorem \ref{invariant}.
\end{proof}
\section{Boundedness in the high norm and contraction in the low norm}
In this section we find $X$ and $Y$ such that the fixed point Theorem \ref{fixedpo} can be applied. Consider the initial value problem for \eqref{eq:71}-\eqref{eq:73}, written as
\begin{equation}
\begin{aligned}
A^{0}(U)V_{t}+A^{i}(U)\partial_{i}V-B^{ij}(U)\partial_{i}\partial_{j}V+D(U)V&=F(U;D_{x}U),\\
\left.V\right\rvert_{t=0}&=U_{0}
\end{aligned}
\label{eq:81}
\end{equation}
where for each $U=(u,v,w)\in X_{T_{0}}^{s}(g_{2},M)$ the matrix coefficients are given as in \eqref{eq:55}-\eqref{eq:58}. According with Theorem \ref{invariant}, \eqref{eq:81} has a unique solution $V=(\hat{u},\hat{v},\hat{w})^{\top}\in X_{T_{0}}^{s}(g_{2},M)$. This means that the operator
\begin{equation}
\begin{aligned}
\mathcal{T}:X_{T_{0}}^{s}(g_{2},M)&\rightarrow X_{T_{0}}^{s}(g_{2},M)\\
U&\mapsto V
\end{aligned}
\label{eq:82}
\end{equation}
is well defined. Thus, we can take
\begin{equation}
X:=X_{T_{0}}^{s}(g_{2},M).
\label{eq:83}
\end{equation}
Now, we have to find a Banach space $Y$ such that $X\subset Y$ and statements $(i)$ and $(ii)$ of Theorem \ref{fixedpo} are satisfied. Consider then, $U^{p}\in X$ and set $\mathcal{T}(U^{p})=:V^{p}\in X$ and $\mathcal{T}(V^{p})=:W^{p}\in X$ for $p=1,2$. Set $V^{p}=(\hat{u}^{p},\hat{v}^{p},\hat{w}^{p})$, $W^{p}=(\bar{u}^{p},\bar{v}^{p},\bar{w}^{p})$ for $p=1,2$. Then, by the definition of $W^{p}$, the difference $W^{2}-W^{1}=W$ satisfies,
\begin{equation}
\begin{aligned}
A^{0}(V^{2})W_{t}+A^{i}(V^{2})\partial_{i}W-&B^{ij}(V^{2})\partial_{i}\partial_{j}W+D(V^{2})W=\widehat{R
},\\
\left.W\right\rvert_{t=0}&=0,
\end{aligned}
\label{eq:86}
\end{equation}
where $\widehat{R}:=(\hat{f}_{1}, \hat{f}_{2}, \hat{f}_{3})^{\top}$ satisfies the following estimates.
\begin{lema}
\label{nonlinearestimates}
Let $V^{p},W^{p}\in X_{T_{0}}^{s}(g_{2},M)$ for $p=1,2$. Then, 
\begin{align*}
\|\hat{f}_{l}\|_{s-1}&\leq C\left(\|\hat{u}^{2}-\hat{u}^{1}\|_{s-1}+\|\hat{v}^{2}-\hat{v}^{1}\|_{s}+\|\hat{w}^{2}-\hat{w}^{1}\|_{s-1}\right),\quad l=1,3,\\
\|\hat{f}_{2}\|_{s-2}&\leq C\|(\hat{u}^{2}-\hat{u}^{1},\hat{v}^{2}-\hat{v}^{1},\hat{w}^{2}-\hat{w}^{1})\|_{s-1},
\end{align*}
for some positive constant $C=C(g_{2},M)$.
\end{lema}
\begin{proof}
We only prove the result for $\hat{f}_{1}$ and proceed in the same manner for $\hat{f}_{2}$ and $\hat{f}_{3}$. Observe that
\begin{align*}
\hat{f}_{1}=&A^{0}_{1}(V^{2})\left\lbrace A_{1}^{0}(V^{2})^{-1}f_{1}(V^{2};D_{x}\hat{v}^{2})-A_{1}^{0}(V^{1})^{-1}f_{1}(V^{1};D_{x}\hat{v}^{1})\right\rbrace\nonumber\\
+&A_{1}^{0}(V^{2})\left\lbrace A_{1}^{0}(V^{1})^{-1}A_{11}^{i}(V^{1})-A_{0}^{1}(V^{2})^{-1}A^{i}_{11}(V^{2}
)\right\rbrace\partial_{i}\bar{u}^{1}\nonumber\\
+&A_{1}^{0}(V^{2})\left\lbrace A_{1}^{0}(V^{1})^{-1}A_{12}^{i}(V^{1})-A_{0}^{1}(V^{2})^{-1}A^{i}_{12}(V^{2}
)\right\rbrace\partial_{i}\bar{v}^{1}.
\end{align*}
Define $H(V,D_{x}\hat{v}):=(A_{0}^{1})^{-1}(V)f_{1}(V,D_{x}\hat{v})$ and apply a similar argument to the one that led us to $(i)$ in Lemma \ref{importantlemma} to conclude that 
\begin{align*}
\|H(V^{2},D_{x}\hat{v}^{2})-H(V^{1},D_{x}\hat{v}^{1})\|_{s-1}\leq C^{\prime}\left(\|(\hat{u}^{2}-\hat{u}^{1},\hat{w}^{2}-\hat{w}^{1})\|_{s-1}+\|\hat{v}^{2}-\hat{v}^{1}\|_{s}\right),
\end{align*}
where $C^{\prime}$ is a positive constant depending on $g_{2}$ and $M$. Given that, $V^{2},W^{1}\in X_{T_{0}}^{s}(g_{2},M)$, we can apply theorem 3.1 to obtain the result for the first term in $\hat{f}_{1}$. The rest of the terms can be treated similarly.
\end{proof}
\begin{defi}
Define $Y$ as the vector space of functions $Z=(z_{1},z_{2},z_{3})^{\top}\in\mathbb{R}^{n}\times\mathbb{R}^{k}\times\mathbb{R}^{p}=\mathbb{R}^{N}$ such that 
\begin{align*}
Z\in\mathcal{C}([0,T_{0}];H^{s-1}),\quad z_{2}\in L^{2}(0,T_{0};H^{s}),\quad Z_{t}\in L^{2}(0,T_{0};H^{s-2}),
\end{align*}
with norm defined by
\[\|Z\|_{Y}^{2}:=\sup_{0\leq t\leq T_{0}}\|Z(t)\|_{s-1}^{2}+\int_{0}^{T_{0}}\|z_{2}(t)\|_{s}^{2}dt+\int_{0}^{T_{0}}\|Z_{t}(t)\|_{s-2}^{2}dt.\]
\end{defi}
\begin{rem}
Observe that $\left(Y,\|\cdot\|_{Y}\right)$ is a Banach space and $X\subset Y$. Nonetheless, $X$ is not a closed subset of $Y$.
\end{rem}
\begin{theo}
\label{lipschitz}
There is a constant $C=C(g_{2},M)>0$ such that 
\begin{align}
\|\mathcal{T}(V^{2})-\mathcal{T}(V^{1})\|_{Y}^{2}&\leq C\max(1,T_{0})\left(\|V^{2}-V^{1}\|_{Y}^{2}\right),\quad \mbox{for all}~ V^{1},V^{2}\in X.\label{eq:818}
\end{align}
That is, $\mathcal{T}:X\rightarrow X$ is Lipschitz continuous on its domain. Furthermore, 
\begin{equation}
\|\mathcal{T}^{2}(U^{2})-\mathcal{T}^{2}(U^{1})\|_{Y}\leq CT_{0}\left(\|\mathcal{T}(U^{2})-\mathcal{T}(U^{1})\|_{Y}+\|U^{2}-U^{1}\|_{Y}\right),
\label{eq:823}
\end{equation}
for all $U^{1},U^{2}\in X$.
\end{theo}
\begin{proof}
Let $V^{p}=(\hat{u}^{p},\hat{v}^{p},\hat{w}^{p})\in X$ and set $\mathcal{T}(V^{p}):=(\bar{u}^{p},\bar{v}^{p},\bar{w}^{p})\in X$ for $p=1,2$. Then, by Theorem \ref{energydec}, $\mathcal{T}(V^{2})-\mathcal{T}(V^{1})$ is the unique solution of \eqref{eq:86}. Consequently, estimate \eqref{eq:543} holds true with $m=s-1$ for all $t\in[0,T_{0}]$ and by taking into account \eqref{eq:720}, \eqref{eq:725} and Lemma \ref{nonlinearestimates}, we have that
\begin{align}
&\sup_{0\leq t\leq T_{0}}\|\left(\mathcal{T}(V^{2})-\mathcal{T}(V^{1})\right)(t)\|_{s-1}^{2}+\int_{0}^{T_{0}}\|(\bar{v}^{2}-\bar{v}^{1})(\tau)\|_{s}^{2}d\tau\nonumber\\
&+\int_{0}^{t}\|\partial_{t}\mathcal{T}(V^{2})-\partial_{t}\mathcal{T}(V^{1})\|_{s-2}^{2}d\tau\nonumber\\
&\leq C_{1}e^{C_{2}\left(T_{0}+M T^{1/2}_{0}\right)}\left\lbrace\int_{0}^{T_{0}}\mathcal{F}^{s-1}(\hat{f}_{1}(t),\hat{f}_{2}(t),\hat{f}_{3}(t))dt\right\rbrace\nonumber\\
&\leq 2C_{1}\int_{0}^{T_{0}}\|\hat{f}_{1}(\tau)\|_{s-1}^{2}+\|\hat{f}_{2}(\tau)\|_{s-2}^{2}+\|\hat{f}_{3}(\tau)\|_{s-1}^{2}d\tau\nonumber\\
&\leq C\left(\int_{0}^{T_{0}}\|(V^{2}-V^{1})(\tau)\|_{s-1}^{2}d\tau+\int_{0}^{T_{0}}\|(\hat{v}^{2}-\hat{v^{1}})(\tau)\|_{s}^{2}d\tau\right)\nonumber\\
&\leq C\left(T_{0}\sup_{0\leq t\leq T_{0}}\|(V^{2}-V^{1})(t)\|_{s-1}^{2}+\int_{0}^{T_{0}}\|(\hat{v}^{2}-\hat{v}^{1})(\tau)\|_{s}^{2}d\tau\right).\label{eq:816}
\end{align}
After recalling the definition of $\|\cdot\|_{Y}$,  we can assert that \eqref{eq:818} follows by \eqref{eq:816}.
Now let $U^{1},U^{2}\in X$ and set $V^{p}:=\mathcal{T}(U^{p})$ for $p=1,2$. Then, by \eqref{eq:816}, 
\begin{align}
&\|\mathcal{T}^{2}(U^{2})-\mathcal{T}^{2}(U^{1})\|_{s-1}^{2}=\|\mathcal{T}(V^{2})-\mathcal{T}(V^{1})\|_{s-1}^{2}\nonumber\\
&\leq C\left(T_{0}\sup_{0\leq t\leq T_{0}}\|\mathcal{T}(U^{2})-\mathcal{T}(U^{1})\|_{s-1}^{2}+\int_{0}^{T_{0}}\|(\hat{v}^{2}-\hat{v}^{1})(\tau)\|_{s}^{2}d\tau\right).\label{eq:prefp}
\end{align}
Since $V=V^{2}-V^{1}$ is the unique solution of \eqref{eq:86} with $U^{2}$ instead of $V^{2}$, we can assure that $\hat{v}:=\hat{v}^{2}-\hat{v}^{1}$ satisfies the following parabolic equation
\begin{equation}
A_{2}^{0}(U^{2})\hat{v}_{t}+A_{22}^{i}(U^{2})\partial_{i}\hat{v}-B_{0}^{ij}(U^{2})\partial_{i}\partial_{j}\hat{v}=\hat{g},
\label{eq:819}
\end{equation}
where $\hat{g}:=f_{2}^{\prime}-A^{i}_{21}(U^{2})\partial_{i}\hat{u}-A^{i}_{23}(U^{2})\partial_{i}\hat{w}$, $\hat{u}=\hat{u}^{2}-\hat{u}^{1}$, $\hat{w}=\hat{w}^{2}-\hat{w}^{1}$ and 
\begin{align}
f_{2}^{\prime}&=A_{2}^{0}(U^{2})\left\lbrace A_{2}^{0}(U^{2})^{-1}f_{2}(U^{2},D_{x}U^{2})-A_{2}^{0}(U^{1})^{-1}f_{2}(U^{1},D_{x}U^{1})\right\rbrace\nonumber\\
&-A_{2}^{0}(U^{2})\left\lbrace A_{2}^{0}(U^{2})^{-1}A_{21}^{i}(U^{2})-A_{2}^{0}(U^{1})^{-1}A_{21}^{i}(U^{1})\right\rbrace\partial_{i}\hat{u}^{1}\nonumber\\
&-A_{2}^{0}(U^{2})\left\lbrace A_{2}^{0}(U^{2})^{-1}A_{22}^{i}(U^{2})-A_{2}^{0}(U^{1})^{-1}A_{22}^{i}(U^{1})\right\rbrace\partial_{i}\hat{v}^{1}\nonumber\\
&-A_{2}^{0}(U^{2})\left\lbrace A_{2}^{0}(U^{2})^{-1}A_{23}^{i}(U^{2})-A_{2}^{0}(U^{1})^{-1}A_{23}^{i}(U^{1})\right\rbrace\partial_{i}\hat{w}^{1}\nonumber\\
&+A_{2}^{0}(U^{2})\left\lbrace A_{2}^{0}(U^{2})^{-1}B_{0}^{ij}(U^{2})-A_{2}^{0}(U^{1})^{-1}B_{0}^{ij}(U^{1})\right\rbrace\partial_{i}\partial_{j}\hat{v}^{1}.\label{eq:820}
\end{align}
Hence, $\hat{g}\in\mathcal{C}\left([0,T];H^{s-2}\right)$. Then, $\hat{v}$ satisfies the estimate \eqref{eq:441} with $m=s-1$ for all $t\in[0,T_{0}]$ and thus, 
\begin{equation*}
\int_{0}^{T_{0}}\|\hat{v}(\tau)\|_{s}^{2}d\tau\leq \Psi_{0}^{2}(T_{0})\int_{0}^{T_{0}}\|\hat{g}(\tau)\|_{s-2}^{2}d\tau.
\end{equation*}
Observe that, by Lemma \ref{existencelemma} and \eqref{eq:725}, $\Psi_{0}^{2}(T_{0})\leq \Phi_{1}^{2}(T_{0})\leq 2$, and by Lemma \ref{nonlinearestimates} applied to \eqref{eq:820} we have that,
\begin{equation*}
\|\hat{g}(t)\|_{s-2}^{2}\leq C\sup_{0\leq t\leq T_{0}}\left(\|(V^{2}-V^{1})(t)\|_{s-1}^{2}+\|(U^{2}-U^{1})(t)\|_{s-1}^{2}\right)~\forall t\in[0,T_{0}]
\end{equation*}
for some positive constant $C=C(g_{2},M)$. Therefore, 
\begin{equation*}
\int_{0}^{T_{0}}\|\hat{v}(\tau)\|_{s}^{2}d\tau\leq CT_{0}\left(\sup_{0\leq t\leq T_{0}}\|(V^{2}-V^{1})(t)\|_{s-1}^{2}+\sup_{0\leq t\leq T_{0}}\|(U^{2}-U^{1})(t)\|_{s-1}^{2}\right).
\end{equation*}
If we apply this inequality in \eqref{eq:prefp} we conclude that $\mathcal{T}$ satisfies \eqref{eq:823} for all $U^{1},U^{2}\in X$.
\end{proof}
\begin{rem}
Observe that the Lipschitz constant in \eqref{eq:818} is always greater or equal than one, no matter how small $T_{0}$ is taken. On the other hand, if additionally to \eqref{eq:725} and \eqref{eq:728} we take $T_{0}$ small enough so that 
\begin{align}
0<\alpha_{0}:=CT_{0}<\tfrac{1}{2}\label{eq:fibonacci},
\end{align}
we can assure that $\mathcal{T}$ satisfies \eqref{eq:25} in Theorem \ref{fixedpo}.
\end{rem}
\section{Existence of the extension}
\label{Extensionsection}
In this section we show the existence of the extension $\widehat{\mathcal{T}}$ of $\mathcal{T}$ as described in Theorem \ref{fixedpo}. To begin with, notice that, given the Banach space structure of $Y$, if $\left\lbrace U^{k}\right\rbrace\subset X$ is a Cauchy sequence in $Y$ then, there is $U=(u,v,w)\in\mathcal{C}([0,T_{0}];H^{s-1})$, $v_{0}\in L^{2}(0,T_{0};H^{s})$ and $(u_{1},v_{1},w_{1})\in L^{2}(0,T_{0},H^{s-2})$ such that 
\begin{align}
(u^{k},v^{k},w^{k})&\rightarrow (u,v,w)~\mbox{in}~\mathcal{C}([0,T_{0}];H^{s-1}),\label{eq:91}\\
v^{k}&\rightarrow v_{0}~\mbox{in}~L^{2}(0,T_{0};H^{s}),\label{eq:92}\\
(u_{t}^{k},v_{t}^{k},w_{t}^{k})&\rightarrow (u_{1},v_{1},w_{1})~\mbox{in}~L^{2}(0,T_{0};H^{s-2}),\label{eq:93}
\end{align}
where $v_{0}=v$ and $(u^{1},v^{1},w^{1})=(u_{t},v_{t},w_{t})$. Consider the following result.
\begin{lema}
\label{X0}
Let $s\in\mathbb{N}$ such that $s\geq s_{0}+1$, $s_{0}=\left[\frac{d}{2}\right]+1$. Define $X_{0}$ as the set of all functions $V=(v_{1},v_{2},v_{3})^{\top}$ of $(x,t)\in Q_{T_{0}}$ \textup{(}$v_{1}\in\mathbb{R}^{n}$, $v_{2}\in\mathbb{R}^{k}$ and $v_{3}\in\mathbb{R}^{p}$\textup{)} such that
\begin{equation}
\label{eq:97}
\begin{aligned}
V\in L^{\infty}(0,T_{0};H^{s}),&\quad v_{2}\in L^{2}(0,T_{0};H^{s+1}),\\
V_{t}\in L^{2}(0,T_{0};H^{s-1}),&\quad V(x,t)\in\mathcal{O}_{g_{2}}\quad\forall(x,t)\in Q_{T_{0}},
\end{aligned}
\end{equation}
and the estimate 
\begin{equation}
\sup_{0\leq\tau\leq t}\|V(\tau)\|_{s}^{2}+\int_{0}^{t}\|v_{2}(\tau)\|_{s+1}^{2}d\tau+
+\int_{0}^{t}\|V_{t}(\tau)\|_{s-1}^{2}d\tau\leq M^{2}
\label{eq:98}
\end{equation}
are satisfied. Then, $\overline{X}\subset X_{0}$.
\end{lema}
\begin{proof}
Let $\left\lbrace U^{k}\right\rbrace\subset X$ such that $U^{k}\rightarrow U$ in $Y$. From the definition of $X$ it follows that
\begin{equation}
\sup_{0\leq\tau\leq t}\|U^{k}(\tau)\|_{s}^{2}+\int_{0}^{t}\|v^{k}(\tau)\|_{s+1}^{2}d\tau+
+\int_{0}^{t}\|U_{t}^{k}(\tau)\|_{s-1}^{2}d\tau\leq M^{2}\quad\forall k\in\mathbb{N}.
\label{eq:99}
\end{equation}
Then, in particular, for each $t\in[0,T_{0}]$, the sequence $\left\lbrace U^{k}(t)\right\rbrace$ is bounded in $H^{s}$, implying the existence of $U^{t}\in H^{s}$ and weak convergent subsequence $\left\lbrace U^{j}(t)\right\rbrace$ (where $j$ depends on $t$) such that
\begin{equation}
U^{j}(t)\rightharpoonup U^{t}\quad\mbox{in}\quad H^{s}.
\label{eq:910}
\end{equation}
Let us show that $U^{t}=U(t)$ for all $t\in[0,T_{0}]$. For $h\in L^{2}$ the map $g\mapsto\lambda_{h}(g):=\int_{\mathbb{R}^{d}}(h,g)_{\mathbb{R}^{N}}dx$, is linear and continuous on $H^{s}$, hence, by the Riesz representation Lemma, for each $h\in L^{2}$, there is a unique $\psi=\psi(h)\in H^{s}$ such that, 
\begin{align*}
\int_{\mathbb{R}^{d}}(h,U^{t})_{\mathbb{R}^{N}}dx&=\lambda_{h}(U^{t})=\langle\psi(h),U^{t}\rangle_{s}=\lim_{j\rightarrow\infty}\langle\psi(h),U^{j}(t)\rangle_{s}\\
&=\lim_{j\rightarrow\infty}\lambda_{h}(U^{j}(t))=\lim_{j\rightarrow\infty}\int_{\mathbb{R}^{d}}(h,U^{j}(t))_{\mathbb{R}^{N}}dx=\int_{\mathbb{R}^{d}}(h,U(t))_{\mathbb{R}^{N}}dx,
\end{align*}
where the last equality follows from \eqref{eq:91}. Thus, $U^{t}=U(t)$ for all $t\in[0,T_{0}]$. Now, since $\left\lbrace U^{k}\right\rbrace\subset X$ there is a subsequence $\left\lbrace U^{j}\right\rbrace \subseteq \left\lbrace U^{k}\right\rbrace$, $U^{0}=(u^{0},v^{0},v^{0})\in L^{\infty}(0,T_{0};H^{s})$ and $U_{1}\in L^{2}(0,T_{0};H^{s-1})$ such that $v^{0}\in L^{2}(0,T_{0};H^{s+1})$ and 
\begin{equation}
\label{eq:911}
\begin{aligned}
U^{j}\overset{\ast}{\rightharpoonup} U^{0}~\mbox{in}~ L^{\infty}(0,T_{0};H^{s}),&\quad v^{j}\rightharpoonup v^{0}~\mbox{in}~ L^{2}(0,T_{0};H^{s+1}),\\
U_{t}^{j}\rightharpoonup U_{1}~\mbox{in}~&L^{2}(0,T_{0};H^{s-1}).
\end{aligned}
\end{equation}
By \eqref{eq:91}-\eqref{eq:93} and the uniqueness of the weak limit, it follows that, $U^{0}=U$ and $U_{1}=U_{t}^{0}$. As consequence of \eqref{eq:99} and the weak convergence stated in \eqref{eq:910} and \eqref{eq:911}, for each $t\in[0,T_{0}]$, there is a subsequence $\left\lbrace U^{j}\right\rbrace \subseteq \left\lbrace U^{k}\right\rbrace$ such that
\begin{align*}
&\|U(t)\|_{s}^{2}+\int_{0}^{T_{0}}\|v(\tau)\|_{s+1}^{2}d\tau+
+\int_{0}^{T_{0}}\|U_{t}(\tau)\|_{s-1}^{2}d\tau\leq\\
&\liminf_{k\rightarrow\infty}\left(\|U^{k}(t)\|_{s}^{2}+\int_{0}^{T_{0}}\|v^{k}(\tau)\|_{s+1}^{2}d\tau+\int_{0}^{T_{0}}\|U_{t}^{k}(\tau)\|_{s-1}^{2}d\tau\right)\leq M^{2}.
\end{align*}
Lastly, since $T_{0}$ satisfies \eqref{eq:728}, it follows by \eqref{eq:727} that $V(x,t)\in\mathcal{O}_{g_{2}}$ for all $(x,t)\in Q_{T_{0}}$. Thus, $\overline{X}\subseteq X_{0}$.
\end{proof}
\begin{theo}
\label{extension}
Let $X=X_{T_{0}}^{s}(g_{2},M)$, $\mathcal{T}: X\rightarrow X$ the operator defined in \eqref{eq:81}-\eqref{eq:82} and let $X_{\infty}$ together with $\widehat{\mathcal{T}}$ be defined as in $(i)$ of Theorem \ref{fixedpo}. Then, the map $\widehat{\mathcal{T}}:X_{\infty}\rightarrow X_{\infty}$ is a well-defined extension of $\mathcal{T}$.
\end{theo}
\begin{proof}
Let $U\in X_{\infty}$. Then, there is $\left\lbrace U^{k}\right\rbrace\subset X$ and $V\in Y$ such that $U^{k}\rightarrow U$ and $\mathcal{T}(U^{k})\rightarrow V$ in $Y$. According to the definition of $\mathcal{T}$, if we set $V^{k}:=\mathcal{T}(U^{k})$, we have that
\begin{align*}
A^{0}(U^{k})V^{k}_{t}+A^{i}(U^{k})\partial_{i}V^{k}-&B^{ij}(U^{k})\partial_{i}\partial_{j}V^{k}+D(U^{k})V^{k}=F(U^{k};D_{x}U^{k}),\\
\left.U^{k}\right\rvert_{t=0}=&U_{0}
\end{align*}
for all $k\in\mathbb{N}$. By taking the limit as $k\rightarrow\infty$ and using \eqref{eq:91}-\eqref{eq:93} we find that $V$ satisfies the Cauchy problem \eqref{eq:81} (where the equation is satisfied in $\mathcal{C}([0,T_{0}];H^{s-2})$). If there is another sequence $\{\widehat{U}^{k}\}$ in X such that $\widehat{U}^{k}\rightarrow U$ and $\widehat{V}^{k}:=\mathcal{T}(\widehat{U}^{k})\rightarrow\widehat{V}$ in $Y$ then, as in the previous lines, $\widehat{V}$ satisfies \eqref{eq:81}. But since $U^{k},V^{k},\widehat{V}^{k}\in X$, we have that, $U,V,\widehat{V}\in\overline{X}$ and by Lemma \ref{X0}, it follows that $U,V,\widehat{V}\in X_{0}$. Therefore, $V-\widehat{V}$ satisfies \eqref{eq:lowreg} with $m=s$ and $T=T_{0}$ and thus, the well-posedness of Theorem \ref{energydec} implies that $V=\widehat{V}$. Hence, $\widehat{\mathcal{T}}(U)$ is independent of the Cauchy sequence chosen to converge to $U$.
\end{proof}
\begin{rem}
\label{similarlinear}
As consequence of the Lipschitz continuity of $\mathcal{T}$, it follows that $X_{\infty}=\overline{X}$. Indeed, is enough to show that $\overline{X}\subseteq X_{\infty}$. If $V\in\overline{X}$, there is a sequence $\left\lbrace V^{k}\right\rbrace\subset X$ such that, $V^{k}\rightarrow V$ in $Y$. Then, according with \eqref{eq:818}, $\left\lbrace T(V^{k})\right\rbrace$ is a Cauchy sequence in $Y$ and so, there exists $W\in Y$ such that $T(V^{k})\rightarrow W$ in $Y$. Hence $X_{\infty}=\overline{X}$.
\end{rem}
\begin{coro}
There is a unique $U_{\infty}\in X_{\infty}$ such that $\widehat{\mathcal{T}}(U_{\infty})=U_{\infty}$.
\end{coro}
\begin{proof}
By \eqref{eq:823} and \eqref{eq:fibonacci}, the operator $\mathcal{T}:X\rightarrow X$ satisfies \eqref{eq:25} and by Theorem \ref{extension}, $\widehat{\mathcal{T}}: X_{\infty}\rightarrow X_{\infty}$ is well defined. Therefore, the hypothesis of Theorem \ref{fixedpo} are satisfied and the result follows.
\end{proof}
\begin{theo}
\label{surjective}
Let $X=X_{T_{0}}^{s}(g_{2},M)$, $\widehat{\mathcal{T}}$ the operator defined in Theorem 9.1 and $V_{\infty}\in X_{\infty}$. If there is $W_{\infty}\in X_{\infty}$ such that $\widehat{\mathcal{T}}(W_{\infty})=V_{\infty}$ then, $V_{\infty}\in X$.
\end{theo}
\begin{proof}
Observe that, $V_{\infty},W_{\infty}\in X_{0}\cap\mathcal{C}([0,T_{0}];H^{s-1})$, as consequence of Lemma \ref{X0} and \eqref{eq:91}. It remains to show that $V_{\infty}\in\mathcal{C}([0,T_{0}]; H^{s})$. If $V_{\infty}^{\epsilon}$ is the mollification of $V_{\infty}$ then, for each $\epsilon>0$, $V_{\infty}^{\epsilon}\in\mathcal{C}([0,T_{0}];H^{s})$. Consider $V_{\infty}^{\epsilon}$ and $V_{\infty}^{\bar{\epsilon}}$ as well as their respective Cauchy problems (see \eqref{eq:444}). Since $W_{\infty}\in X_{0}$, the assumptions of Theorem \ref{energydec} are satisfied (with $m=s$, $T=T_{0}$) and according with \eqref{eq:543}, \eqref{eq:537} and \eqref{eq:538}, the estimate, $\|V_{\infty}^{\epsilon}(t)-V_{\infty}^{\bar{\epsilon}}(t)\|_{s}\leq K_{0,\epsilon,\bar{\epsilon}}(T_{0})\Phi_{0,\epsilon,\bar{\epsilon}}(T_{0})$, holds true for all $t\in[0,T_{0}]$ and some $K_{0,\epsilon,\bar{\epsilon}}(T_{0})$ and $\Phi_{0,\epsilon,\bar{\epsilon}}(T_{0})$ with the property that $K_{0,\epsilon,\bar{\epsilon}}(T_{0}),\Phi_{0,\epsilon,\bar{\epsilon}}(T_{0})\rightarrow 0$ as $\epsilon,\bar{\epsilon}\rightarrow 0$. Hence, $\left\lbrace V_{\infty}^{\epsilon}\right\rbrace_{\epsilon\rightarrow 0}$ is a Cauchy sequence in $\mathcal{C}([0,T_{0}];H^{s})$ and $V_{\infty}\in\mathcal{C}([0,T_{0}];H^{s})$.
\end{proof}
\begin{coro}
Let $U_{\infty}\in X_{\infty}$ be the unique fixed point of $\widehat{\mathcal{T}}$. Then, $U_{\infty}\in X$.\qed
\end{coro}
Finally, as consequence of Corollaries 2 and 3 we conclude the following Theorem.
\begin{theo}[Local existence and uniqueness]
\label{local}
Let $s\in\mathbb{N}$ such that $s\geq s_{0}+1$, $s_{0}=\left[\frac{d}{2}\right]+1$. Assume that conditions \textup{\textbf{A}} through \textup{\textbf{F}} are satisfied. Then there exists a positive constant $T_{0}$($\leq T$), depending only on $\mathcal{O}_{g_{0}}$, $\mathcal{O}_{g_{2}}$ and $\|(u_{0},v_{0},w_{0})\|_{s}$, such that the initial value problem 
\begin{equation}
\label{eq:quasilinear}
\begin{aligned}
A_{1}^{0}(U)u_{t}+A^{i}_{11}(U)\partial_{i}u+A_{12}^{i}(U)\partial_{i}v&=f_{1}(U,D_{x}v),\\
A_{2}^{0}(U)v_{t}+A^{i}_{21}(U)\partial_{i}u+A_{22}^{i}(U)\partial_{i}v+A^{i}_{23}(U)\partial_{i}w&-B_{0}^{ij}(U)\partial_{i}\partial_{j}v\\
&=f_{2}(U,D_{x}U),\\
A_{3}^{0}(U)w_{t}+A^{i}_{32}(U)\partial_{i}v+A_{33}^{i}(U)\partial_{i}w+D_{0}(U)w&=f_{3}(U,D_{x}v),\\
U(x,0)&=U_{0}(x),
\end{aligned}
\end{equation}
has a unique solution $U\in X_{T_{0}}^{s}(g_{2},M)$, where $g_{2}$ and $M$ are determined by \eqref{eq:722} and \eqref{eq:723} respectively.\qed
\end{theo}
\section{The viscous Cattaneo-Christov system for compressible fluid flow in several space dimensions}
\label{CattaneoC}
Consider the basic equations for a compressible, viscous, heat conducting fluid in space ($x\in\mathbb{R}^{d}$ with $d=1,~2~\mbox{or}~3$)
\begin{eqnarray}
\rho_{t}+\nabla\cdot(\rho v)&=&0,\label{eq:111}\\
(\rho v)_{t}+\nabla\cdot(\rho v\otimes v)&=&\nabla\cdot\mathbb{T},\label{eq:112}\\
(\rho E)_{t}+\nabla\cdot(\rho E v)&=&\nabla\cdot(\mathbb{T}v)-\nabla\cdot q,\label{eq:113}
\end{eqnarray}
where $\rho=\rho(x,t)$ is the mass density, $v=(v_{1},..,v_{d})(x,t)\in\mathbb{R}^{d}$ is the velocity field, $E:=\rho\left(\frac{1}{2}|v|^{2}+e\right)$ is the total energy, $e=e(x,t)\in\mathbb{R}$ is the internal energy field, $\mathbb{T}$ denotes the Newtonian stress tensor given as $\mathbb{T}=2\mu\mathbb{D}(v)+\lambda\nabla\cdot v\mathbb{I}-p\mathbb{I}$, with $\mathbb{I}$ the identity matrix of order $d\times d$ and $q=(q_{1},..,q_{d})(x,t)\in\mathbb{R}^{d}$ is the heat flux vector. Here, $\mathbb{D}(u)=\frac{1}{2}(\nabla u+\nabla u ^{T})$ denotes the deformation tensor, $p=p(x,t)\in\mathbb{R}$ is the pressure field and $\lambda$ and $\mu$ are the Newtonian viscosity coefficients.\\
In order to close the system \eqref{eq:111}-\eqref{eq:113} a constitutive equation for the heat flux is needed (see \cite[Chapter 6]{tadmor}). In this section, we consider equations \eqref{eq:111}-\eqref{eq:113} coupled together with the frame invariant formulation of the Maxwell-Cattaneo law proposed by Christov namely, equation \eqref{eq:17}. We make following thermodynamical assumptions:
\begin{itemize}
	\item [\textbf{T1}] The independent thermodynamical variables are the mass density $\rho>0$ and the absolute temperature $\theta>0$. They vary within the domain
	\[\mathcal{D}:=\left\lbrace(\rho,\theta)\in\mathbb{R}^{2}:\rho>0,\theta>0\right\rbrace.\]
Moreover, $p,e,\lambda,\mu,\kappa\in\mathcal{C}^{\infty}(\mathcal{D})$.
	\item [\textbf{T2}] The viscosity coefficients and the heat conductivity satisfy the inequalities
	\[\frac{2}{3}\mu(\rho,\theta)+\lambda(\rho,\theta),~\mu(\rho,\theta),~\kappa(\rho,\theta)>0\quad\mbox{for all}~(\rho,\theta)\in\mathcal{D}\] 
	\item [\textbf{T3}] The fluid satisfies that, $p,p_{\rho},p_{\theta},e_{\theta}>0$, for all $(\rho,\theta)\in\mathcal{D}$.
\end{itemize}
\begin{rem}
Assumption \textup{\textbf{T3}} refers to compressible fluids satisfying the standard assumptions of Weyl \cite{weyl}. In particular, with them, one can show that the Euler system of equations is of hyperbolic type \textup{(}\cite{daf}, \cite{serre}\textup{)}. Meanwhile, assumption \textup{\textbf{T2}} gives a viscous and heat conductive character to the model. For these reason, we refer to equations \eqref{eq:111}-\eqref{eq:113} together with \eqref{eq:17} as the viscous Cattaneo-Christov system. In a limit case in which $\mu=\lambda=0$ on $\mathcal{D}$ we obtain the inviscid version of the system.
\end{rem}
\subsection{Quasilinear Cattaneo-Christov systems}
The Cattaneo-Christov system, (equations \eqref{eq:111}-\eqref{eq:113} and \eqref{eq:17}), can be written in the quasilinear form, 
\begin{equation}
\widetilde{A}^{0}(U)U_{t}+\widetilde{A}^{i}(U)\partial_{i}U+\widetilde{D}(U)U-\widetilde{B}^{ij}(U)\partial_{i}\partial_{j}U=\widetilde{F}(U,D_{x}U),
\label{eq:116}
\end{equation}
where $U=(\rho,v,\theta,q)^{\top}\in\mathcal{O}\subset\mathbb{R}^{N}$ is the vector of state variables, defined on the convex open set 
\[\mathcal{O}:=\left\lbrace(\rho,v,\theta,q)\in\mathbb{R}^{N}:\rho>0,\theta>0\right\rbrace,\]
$N=2d+2$ and $d$ is the spatial dimension. We show the local existence and uniqueness of the Cauchy problem for \eqref{eq:116} when $d=3$. To obtain \eqref{eq:116}, we provide the symbols
\begin{align}
\widetilde{A}(\xi;U):=\sum_{i=1}^{3}\widetilde{A}^{i}(U)\xi_{i}\quad\mbox{and}\quad \widetilde{B}(\xi;U):=\sum_{i,j=1}^{3}\widetilde{B}^{ij}(U)\xi_{i}\xi_{j},\label{eq:115}
\end{align}
for each $\xi=(\xi_{1},\xi_{2},\xi_{3})\in\mathbb{S}^{2}$, as well as $\widetilde{A}^{0}(U)$, $\widetilde{D}(U)$ and the nonlinear terms $\widetilde{F}(U,D_{x}U)$. Indeed, $\widetilde{A}^{0}(U)$ and $\widetilde{D}(U)$ are diagonal matrices given as
\begin{align*}
\widetilde{A}^{0}(U)=\left(\begin{array}{cccc}
	1&&&\\
	&\rho\mathbb{I}_{3}&&\\
	&&\rho e_{\theta}&\\
	&&&\tau\mathbb{I}_{3}
\end{array}\right),\quad \widetilde{D}(U)=\left(\begin{array}{cc}
	\mathbb{O}_{5\times 5}&\\
	&\mathbb{I}_{3}
\end{array}\right),
\end{align*}
where $\mathbb{I}_{3}$ denotes the identity matrix of order $3\times 3$, $\mathbb{O}_{l_{1}\times l_{2}}$ is the zero matrix of order $l_{1}\times l_{2}$ and all the empty spaces refer to zero block matrices of the appropriate sizes. 
\small
\begin{equation*}
\widetilde{A}(\xi;U)=\left(\begin{array}{cccccccc}
	\xi\cdot v&\xi_{1}\rho&\xi_{2}\rho&\xi_{3}\rho&0&0&0&0\\
	\xi_{1}p_{\rho}&\rho\xi\cdot v&0&0&\xi_{1}p_{\theta}&0&0&0\\
	\xi_{2}p_{\rho}&0&\rho\xi\cdot v&0&\xi_{2}p_{\theta}&0&0&0\\
	\xi_{3}p_{\rho}&0&0&\rho\xi\cdot v&\xi_{3}p_{\theta}&0&0&0\\
	0&\xi_{1}\theta p_{\theta}&\xi_{2}\theta p_{\theta}&\xi_{3}\theta p_{\theta}&\rho e_{\theta}\xi\cdot v&\xi_{1}&\xi_{2}&\xi_{3}\\
	0&&&&\xi_{1}\kappa&\tau\xi\cdot v&0&0\\
	0&&\tau\mathcal{Q}(\xi;q)&&\xi_{2}\kappa&0&\tau\xi\cdot v&0\\
	0&&&&\xi_{3}\kappa&0&0&\tau\xi\cdot v\\
\end{array}\right),
\end{equation*}
\normalsize
where, for each $\xi\in\mathbb{S}^{2}$, the matrix $\widetilde{A}(\xi,U)$ has a sub-block matrix, $\mathcal{Q}(q;\xi)$, of order $3\times 3$, that is being multiplied by $\tau\neq 0$, given as
\begin{equation*}
\mathcal{Q}(q;\xi)=\left(\begin{array}{ccc}
	-\xi_{2} q_{2}-\xi_{3} q_{3}&\xi_{2}q_{1}&\xi_{3}q_{1}\\
	\xi_{1}q_{2}&-\xi_{1}q_{1}-\xi_{3}q_{3}&\xi_{3}q_{2}\\
	\xi_{1}q_{3}&\xi_{2}q_{3}&-\xi_{1}q_{1}-\xi_{2}q_{2}
\end{array}\right).
\end{equation*}
There is a matrix of order $3\times 3$, given as $B_{0}(\xi;U):=\mu\mathbb{I}_{3\times 3}+(\lambda+\mu)\xi\otimes\xi$ such that,
\[\widetilde{B}(\xi;U)=\left(\begin{array}{cccccc}
	&\mathbb{O}_{1\times 4}&&&\\
	&&&&\mathbb{O}_{4\times 4}\\
	\mathbb{O}_{3\times 1}&B_{0}(\xi;U)&&\\
	&\mathbb{O}_{4\times 4}&&&\mathbb{O}_{4\times 4}
\end{array}\right).\] 
Since $\lambda$ and $\mu$ are functions of $\rho$ and $\theta$, the non-linear terms are 
\[\widetilde{F}(U,D_{x}U)=\left(\begin{array}{c}
	0\\
	(\nabla\cdot v)\partial_{x_{1}}\lambda+D_{x} v_{1}\cdot D_{x}\mu+\partial_{x_{1}}v\cdot\nabla\mu\\
	(\nabla\cdot v)\partial_{x_{2}}\lambda+D_{x} v_{2}\cdot D_{x}\mu+\partial_{x_{2}}v\cdot\nabla\mu\\
	(\nabla\cdot v)\partial_{x_{3}}\lambda+D_{x} v_{3}\cdot D_{x}\mu+\partial_{x_{3}}v\cdot\nabla\mu\\
	\lambda(\nabla\cdot v)^{2}+\frac{\mu}{2}(\partial_{j}v_{i}+\partial_{i}v_{j})^{2}\\
	\mathbb{O}_{3\times 1}
\end{array}\right),\]
where the summation convention is being used for $(\partial_{j}v_{i}+\partial_{i}v_{j})^{2}$. When $\lambda=\mu=0$ for all $(\rho,\theta)\in\mathcal{D}$ we obtain a first order system  that is not hyperbolic \cite{angeles2021nonhyperbolicity}, even when its characteristic speeds are real. Consequently, \eqref{eq:116} is not Friedrichs symmetrizable. Nonetheless, the system has a \emph{partial symmetrizer} as stated in the following result.
\begin{lema}
\label{partialsym}
Let \textup{\textbf{T1}}-\textup{\textbf{T3}} be satisfied and consider the Cattaneo-Christov system in the form \eqref{eq:116} for $d=3$. There exists a partition of $U$ into $U=(u,v,w)$ where $u\in\mathbb{R}^{n}$, $v\in\mathbb{R}^{k}$, $w\in\mathbb{R}^{p}$, $n+k+p=8$, and $S=S(U)$, a matrix value function of $U\in\mathcal{O}$ of order $8\times 8$ such that:
\begin{itemize}
	\item [(a)] $S=S(U)$ is smooth and positive definite uniformly on each compact set contained in $\mathcal{O}$;
	\item [(b)] the products $A^{0}(U):=S(U)\widetilde{A}^{0}(U)$, $A^{i}(U):=S(U)\widetilde{A}^{i}(U)$, $B^{ij}(U):=S(U)\widetilde{B}^{ij}(U)$ and $D(U):=S(U)\widetilde{D}(U)$ all have the block matrix decomposition described in equations \eqref{eq:55}-\eqref{eq:58}.
\end{itemize}
\end{lema}
\begin{proof}
First, notice that, in the Cattaneo-Christov systems, the variable $\rho$ is decoupled from the variable $(\theta,q)\in\mathbb{R}^{4}$, that is, there are no spatial derivatives of $(\theta,q)$ in the evolution equation for $\rho$ (the conservation of mass) and vice-versa. Therefore, we take $u:=\rho$, $v$ as the velocity field, and $w:=(\theta,q)$, implying that $u\in\mathbb{R}$, $v\in\mathbb{R}^{3}$ and $w\in\mathbb{R}^{4}$. Consider the diagonal matrix function
\begin{equation}
S(U)=\left(\begin{array}{cccc}
\frac{p_{\rho}}{\rho}&&&\\
&\mathbb{I}_{3\times 3}&&\\
&&\frac{1}{\theta}&\\
&&&\frac{1}{\kappa\theta}\mathbb{I}_{3\times 3}\\
\end{array}\right).
\label{eq:117}
\end{equation}
By \textbf{T1}-\textbf{T3}, $S(U)$ is smooth and satisfies $(a)$. Observe that, $S(U)\widetilde{A}(\xi;U)$ is given as
\small
\begin{equation*}
\left(\begin{array}{cccccccc}
	\frac{p_{\rho}}{\rho}\xi\cdot v&p_{\rho}\xi_{1}&p_{\rho}\xi_{2}&p_{\rho}\xi_{3}&0&0&0&0\\
	\xi_{1}p_{\rho}&\rho\xi\cdot v&0&0&\xi_{1}p_{\theta}&0&0&0\\
	\xi_{2}p_{\rho}&0&\rho\xi\cdot v&0&\xi_{2}p_{\theta}&0&0&0\\
	\xi_{3}p_{\rho}&0&0&\rho\xi\cdot v&\xi_{3}p_{\theta}&0&0&0\\
	0&\xi_{1}p_{\theta}&\xi_{2}p_{\theta}&\xi_{3}p_{\theta}&\frac{\rho e_{\theta}}{\theta}\xi\cdot v&\frac{\xi_{1}}{\theta}&\frac{\xi_{2}}{\theta}&\frac{\xi_{3}}{\theta}\\
	0&&&&\frac{\xi_{1}}{\theta}&\frac{\tau}{\kappa\theta}\xi\cdot v&0&0\\
	0&&\frac{\tau}{\kappa\theta}\mathcal{Q}(\xi;q)&&\frac{\xi_{2}}{\theta}&0&\frac{\tau}{\kappa\theta}\xi\cdot v&0\\
	0&&&&\frac{\xi_{3}}{\theta}&0&0&\frac{\tau}{\kappa\theta}\xi\cdot v\\
\end{array}\right),
\end{equation*}
\normalsize
In this case, by using the symbols for the block matrices, we recognize that, for each $\xi\in\mathbb{S}^{2}$,  $A^{i}_{11}(\xi;U):=\frac{p_{\rho}}{\rho}v\cdot\xi$ is a matrix of order $1\times 1$ and thus symmetric, and
\[A^{i}_{33}(U)=\left(\begin{array}{cccc}
	\frac{\rho e_{\theta}}{\theta}v\cdot\xi&\frac{\xi_{1}}{\theta}&\frac{\xi_{2}}{\theta}&\frac{\xi_{3}}{\theta}\\
	\frac{\xi_{1}}{\theta}&\frac{\tau\xi\cdot v}{\kappa\theta}&0&0\\
	\frac{\xi_{2}}{\theta}&0&\frac{\tau\xi\cdot v}{\kappa\theta}&0\\
	\frac{\xi_{3}}{\theta}&0&0&\frac{\tau\xi\cdot v}{\kappa\theta}
\end{array}\right)\]
is a symmetric matrix of order $4\times 4$. Hence, assumption \textbf{I} is satisifed. Once this matrices are recognized, the rest of the block structure for $S(U)\widetilde{A}(\xi;U)$ follows. Notice that $S(U)\widetilde{B}(\xi;U)=\widetilde{B}(\xi;U)$ and 
\begin{equation*}
A^{0}(U)=S(U)\widetilde{A}^{0}(U)=\left(\begin{array}{cccc}
	\frac{p_{\rho}}{\rho}&&&\\
	&\rho\mathbb{I}_{3\times 3}&&\\
  &&\frac{\rho e_{\theta}}{\theta}&\\
	&&&\frac{\tau}{\kappa\theta}\mathbb{I}_{3\times 3}\\
\end{array}\right).
\end{equation*}
Therefore,  
\begin{align}
\label{eq:A0s}
A^{0}_{1}(U)=\rho,\quad
A^{0}_{3}(U)=\left(\begin{array}{cc}
	\frac{\rho e_{\theta}}{\theta}&\\
	&\frac{\tau}{\kappa\theta}\mathbb{I}_{3\times 3}
\end{array}\right)\quad\mbox{and}\quad A^{0}_{2}(U)=\rho\mathbb{I}_{3\times 3}.
\end{align}
Finally,  
 \begin{equation*}
S(U)\widetilde{D}(U)=\left(\begin{array}{cc}
	\mathbb{O}_{5\times 5}&\\
  &\frac{\tau}{\kappa\theta}\mathbb{I}_{3\times 3}
\end{array}\right).
\end{equation*}
\end{proof}
Observe that, as the Cattaneo-Christov system shows, the existence of a partial symmetrizer for a system of the form \eqref{eq:116} doesn't imply the hyperbolicity of the first order part of the system. Nonetheless, it is enough to conclude the following result.
\begin{theo}
Let $s\in\mathbb{N}$ such that $s\geq s_{0}+1$, $s_{0}=\left[\frac{d}{2}\right]+1$. Let $T>0$ and $U_{0}\in\mathcal{O}$ be given and suppose that the initial data satisfies \textup{\textbf{F}}. Under assumptions \textup{\textbf{T1}}-\textup{\textbf{T3}}, there are constants $0<T_{0}<T$, $g_{2}>0$ and $M$ such that, there is a unique solution $U=(\rho,v,\theta,q)^{\top}\in X_{T_{0}}^{s}(g_{2},M)$ for the Cauchy problem of the viscous Cattaneo-Christov system.
\end{theo}
\begin{proof}
First, multiply \eqref{eq:116} by the partial symmetrizer given in \eqref{eq:117}. Then, according with Lemma \ref{partialsym}, this yields a quasilinear equation of the form \eqref{eq:quasilinear}. By \eqref{eq:A0s}, assumption \textbf{B} is satisfied. It is easy to show that, $(B_{0}(\xi;U)y,y)_{\mathbb{R}^{3}}\geq\min\left\lbrace 2\mu+\lambda,\mu\right\rbrace|y|^{2}$, for all $y\in\mathbb{R}^{3}$ and all $\xi\in\mathbb{S}^{2}$. Hence, by \textbf{T2}, assumption \textbf{D} holds true. Since the nonlinear terms are
\[f_{1}(U;D_{x}v):=0,\quad f_{3}(U,D_{x}v)=\frac{1}{\theta}\left(\lambda(\nabla\cdot v)^{2}+\frac{\mu}{2}(\partial_{j}v_{i}+\partial_{i}v_{j})^{2},0,0,0\right)^{\top}\] 
and
\[f_{2}(U,D_{x}U)=\left(\begin{array}{c}
	(\nabla\cdot v)\partial_{x_{1}}\lambda+D_{x} v_{1}\cdot D_{x}\mu+\partial_{x_{1}}v\cdot\nabla\mu\\
	(\nabla\cdot v)\partial_{x_{2}}\lambda+D_{x} v_{2}\cdot D_{x}\mu+\partial_{x_{2}}v\cdot\nabla\mu\\
	(\nabla\cdot v)\partial_{x_{3}}\lambda+D_{x} v_{3}\cdot D_{x}\mu+\partial_{x_{3}}v\cdot\nabla\mu
\end{array}\right)\]
we deduce that, $f_{1}(U;0)=0$, $f_{2}(U;0)=0$ and $f_{3}(U;0)=0$ for all $U\in\mathcal{O}$, thus complying with assumption \textbf{E}. Take $O_{g_{2}}$ and $M$ as it was described in \eqref{eq:722}-\eqref{eq:723} and the result follows as an application of Theorem \ref{local}.
\end{proof}
\subsection{An impossible viscous approximation}
\label{impvis}
For each $0<\epsilon<1$ consider the Cauchy problem for the linear equation in \eqref{eq:54} with $\epsilon B^{ij}$ instead of $B^{ij}$, that is,
\begin{equation}
\begin{aligned}
A^{0}U_{t}+A^{i}\partial_{i}U-\epsilon B^{ij}\partial_{i}\partial_{j}U+DU&=F,\\
\left.V\right\rvert_{t=0}&=U_{0},
\end{aligned}
\label{eq:121}
\end{equation}
where the equation is assumed to have a partially symmetrized structure. For each $0<\epsilon<1$, there is a unique solution $U^{\epsilon}=U^{\epsilon}(x,t)$, defined on $Q_{T}$, to the initial value problem \eqref{eq:121} that satisfies \eqref{eq:567}-\eqref{eq:571} and the energy estimate \eqref{eq:543}. If we want to use the vanishing viscosity method to conclude the $L^{2}$-well-posedness of the first order Cauchy problem
\begin{equation}
\begin{aligned}
A^{0}U_{t}+A^{i}\partial_{i}U-DU&=F,\\
\left.V\right\rvert_{t=0}&=U_{0},
\end{aligned}
\label{eq:122}
\end{equation}
we need the energy estimate for the linear equation \eqref{eq:122}. However, as the Cattaneo-Christov system shows, the equation in \eqref{eq:122} doesn't have to be hyperbolic. In this case, a vanishing viscosity approach for equation \eqref{eq:122} cannot take place because the hyperbolicity is a necessary condition for the $L^{2}$-well-posedness of \eqref{eq:122}. In fact, when the matrix coefficients of the equation in \eqref{eq:122} are constants (with respect to $(x,t)\in Q_{T}$) the hyperbolicity is equivalent to the $L^{2}$-well posedness of its Cauchy problem (see, \cite{benzoni}, \cite{otto} and \cite{serre}). Moreover, when the matrix coefficients are given functions of $(x,t)\in Q_{T}$, and the algebraic multiplicities of the characteristic eigenvalues are independent of $(x,t,\xi)$, Kano \cite{kano} has shown that the hyperbolicity is necessary for the $L^{2}$-well posedness. Also, Metivier \cite{guy} has shown that, having real eigenvalues is necessary for the well-posedness of the nonlinear Cauchy problem.\\
Consider the inviscid Cattaneo-Christov system and let $U\in\mathcal{O}$. Then, every linear equation of the form 
\begin{align}
A^{0}(U)V_{t}+A^{i}(U)\partial_{i}V+D(U)V=0\label{eq:124}
\end{align} 
was shown to have real eigenvalues of constant algebraic multiplicites with respect to $(x,t,\xi)$ for every fixed $U\in\mathcal{O}$ (see \cite{angeles2021nonhyperbolicity}). Furthermore, whenever the fixed state $U\in\mathcal{O}$ is of the form $U=(\rho, v,\theta,q)$ with $q\neq 0$, it was shown that \eqref{eq:124} is non-hyperbolic. Then, according with \cite{kano}, for any of these states ($q\neq 0$) no $L^{2}$ energy estimate can exist. Consequently, unless $U=(\rho, v, \theta, 0)$, no linearization (eq. \eqref{eq:124}) of the Cattaneo-Christov system can be understood as the vanishing viscosity limit of the extended equation \eqref{eq:121}.

\section{Final comments and conclusions}
In this work we have provided an alternative fixed point argument to deal with the local existence and uniqueness of the Cauchy problem for a system of quasilinear equations. We have highlighted the fact that the \emph{partial symmetrizability} of \eqref{eq:11} is enough to conclude. Hence, the hyperbolicity of the system without diffusion (i.e. formally setting $B^{ij}=0$ for all $i,j=1,..d$ in equation \eqref{eq:11}) is not a required property for the local well-posedness.\\
Notice that, by assuming a linearization of the form \eqref{eq:13} as in \cite{kawa}, it is suggested that only partial symmetrizability is required to conclude the local existence and uniqueness of the Cauchy problem for \eqref{eq:11}. However, in compressible fluid dynamics (where \cite{kawa} has a wide range of applicability), such requirement is trivially satisfied because the existence of a conservative structure and a convex entropy, for the well known models, assure the existence of a Friedrichs symmetrizer (as the Hessian of the entropy) (cf. \cite{daf}, \cite{serre}). In this case, the partial symmetrizability and the hyperbolicity of the inviscid models are always given. For example, when the quasilinear system \eqref{eq:11} can be derived from a set of viscous conservation laws and is a symmetrizable system, Kawashima and Shizuta \cite{ka} were able to formulate a simple sufficient condition (the so called \emph{N condition}) under which \eqref{eq:11} can be put in the following normal form
\begin{equation}
\label{eq:kawade}
\begin{aligned}
A_{1}^{0}(u,v)\partial_{t}u&+A^{i}_{11}(u,v)\partial_{i}u=f_{1}(U,D_{x}v),\\
A_{2}^{0}(u,v)\partial_{t}v&-B_{22}^{ij}(u,v)\partial_{i}\partial_{j}v=f_{2}(U,D_{x}U).
\end{aligned}
\end{equation}
Thus, the local existence and uniqueness for the Cauchy problem is reduced to that of a system of the form \eqref{eq:kawade}.\\
According with \cite{kawa}, the viscous Cattaneo-Christov system may be classified as \emph{hyperbolic-parabolic}. Of course, in such system there is an interplay between hyperbolic and parabolic variables but, as it was argued in section \ref{impvis}, although the Cauchy problem for the viscous Cattaneo-Christov system (both linear and quasilinear cases) is well-posed in $L^{2}$, this property has no influence in the $L^{2}$-well-posedness for the inviscid version of the system. For this reasons, we argue that a more precise classification for equations such as the viscous Cattaneo-Christov system (i.e. equations whose first order associated system \eqref{eq:124} is not hyperbolic), would be \emph{partially parabolic} systems.\\
Regarding the fixed point result presented in Theorem \ref{fixedpo}, we have shown that, it is not necessary for $\mathcal{T}$ or even $\mathcal{T}^{2}$ to be contraction maps. Is enough to verify if $\mathcal{T}^{2}$ satisfies  \eqref{eq:25}. Furthermore, the particular formulation of Theorem \ref{fixedpo}, was developed with the only purpose of compensate the lack of continuity of $\mathcal{T}:X\rightarrow X$ produced by the discrepancy between the \emph{low} norm that defines $Y$ and the \emph{high} norm that defines $X$. For this reason, \eqref{eq:24} was motivated from the notion of \emph{closable operator} in linear functional analysis (see \cite[Section 5.1]{weidmann}, for example). Although $\widehat{\mathcal{T}}$ doesn't have to be a linear operator, it is true that its \emph{graph} is a closed subset of $Y\times Y$. Moreover, it is easy to show that $\widehat{\mathcal{T}}$ is the smallest closed extension of $\mathcal{T}$. Another similarity with closed linear operators was described in Remark \ref{similarlinear} because it was shown that $X_{\infty}$ is a closed subset of $Y$ as consequence of the Lipschitz continuity of $\mathcal{T}$ (cf. \cite[Theorem 5.2]{weidmann}). However, notice that the Lipschitz continuity (or continuity to that matter) of $\mathcal{T}$ is not part of the hypothesis of Theorem \ref{fixedpo}, this was a particular property of the operator $\mathcal{T}$ defined in \eqref{eq:81}-\eqref{eq:83}.\\
The presented formulation of the fixed point theorem also seems to be applicable in the case of a possible fixed point theorem with an iteration inequality of the form 
\[a_{k}\leq\alpha_{0}(a_{k-1}+a_{k-2}+...+a_{k-j})\]
for all $k\geq j$ and some $0<\alpha_{0}<1$ small enough.\\
Through the notion of the extension $\widehat{\mathcal{T}}$ introduced in the formulation of Theorem \ref{fixedpo}, we understood that the best conclusion drawn from the interplay between the \emph{boundedness in the high norm} and \emph{contraction in the low norm} processes,  is  the existence of  the unique fixed point $U_{\infty}$ in $X_{\infty}$, for the extension $\widehat{\mathcal{T}}$. Then, the energy estimate combined with the surjectivity (i.e $\widehat{\mathcal{T}}(U_{\infty})=U_{\infty}$) imply that $U_{\infty}\in X$.
\section{Appendix A: Parabolic energy estimate}
In this section, an a priori estimate of the solution $u$ of the Cauchy problem \eqref{eq:41}-\eqref{eq:42} in the space $\mathcal{P}_{m}(T)$ is obtained.
First, we establish \eqref{eq:441} with the following extra regularity assumption, 
\begin{itemize}
	\item [\textbf{ER}] $u\in\mathcal{P}_{m+1}(T)$, $u_{0}\in H^{m+1}$ and $f\in L^{2}(0,T;H^{m})$.
\end{itemize}
Then, in order to conclude the estimate for $u\in\mathcal{P}_{m}(T)$ and $f\in L^{2}(0,T;H^{m-1})$, a standard mollification argument can be used.\\
By \textbf{H6}, $B^{ij}(t)\in\widehat{H}^{s}$ for a.a. $t\in[0,T]$. On the other hand, \textbf{ER} assures that $\partial_{i}\partial_{j}u(t)\in H^{m}$ for a.a. $t\in[0,T]$. In particular, it holds that $s>\frac{d}{2}$ and thus, theorem 3.1 is applicable. Hence  $B^{ij}(t)\partial_{i}\partial_{j}u(t)\in H^{m}$ for a.a $t\in[0,T]$. The same reasoning leads us to conclude that $A^{0}(t)\partial_{t}u(t)\in H^{m}$  for a.a. $t\in[0,T]$. Therefore, the following procedure is justified.\\
Given that \textbf{H1} implies the existence of $(A^{0})^{-1}(t)$ for all $t\in[0,T]$, we can multiply equation \eqref{eq:41} by $(A^{0})^{-1}$ then, apply the operator $\partial_{x}^{\alpha}\cdot$ to this equation, use Leibniz's rule together with the definition of commutator and multiply the result by $A^{0}$ to get
\begin{align}
A^{0}(\partial_{x}^{\alpha}u)_{t}-B^{ij}\partial_{i}\partial_{j}(\partial_{x}^{\alpha}u)&=A^{0}\partial_{x}^{\alpha}\left((A^{0})^{-1}[f-A^{i}\partial_{i}u-Du]\right)\nonumber\\
&+A^{0}G_{\alpha}((A^{0})^{-1}B^{ij},\partial_{i}\partial_{j}u).\label{eq:47}
\end{align}
Observe that \eqref{eq:47} coincides with \eqref{eq:41} when $\alpha=0$. Since both $B^{ij}$ and $(A^{0})^{-1}$ belong to the space $\widehat{H}^{s}$ for a.a. $t\in[0,T]$, theorem 3.1 assures that $(A^{0})^{-1}B^{ij}\in\widehat{H}^{s}$. Thus, we can take $\xi=(A^{0})^{-1}B^{ij}$ and $w:=\partial_{i}\partial_{j}u\in H^{m}$ in theorem 3.3 to obtain that $G_{\alpha}\in H^{m-|\alpha|}$ for every $1\leq m\leq s$ and $|\alpha|\leq m$ as well as the estimate,
\begin{equation}
\|G_{\alpha}\left((A^{0})^{-1}B^{ij},\partial_{i}\partial_{j}u\right)\|_{m-|\alpha|}\leq C\|\nabla\left((A^{0})^{-1}B^{ij}\right)\|_{s-1}\|\partial_{i}\partial_{j}u\|_{m-1}.\label{eq:ce1}
\end{equation}
Then, by theorem 3.1, $A^{0}G_{\alpha}\in H^{m-|\alpha|}$ and
\begin{align}
\|A^{0}G_{\alpha}\left((A^{0})^{-1}B^{ij},\partial_{i}\partial_{j}u\right)\|_{m-|\alpha|}\leq C\|A^{0}\|_{\bar{s}}\|G_{\alpha}\left((A^{0})^{-1}B^{ij},\partial_{i}\partial_{j}u\right)\|_{m-|\alpha|}.\label{eq:ce2}
\end{align}
The estimates for $A^{0}\partial_{x}^{\alpha}\left((A^{0})^{-1}[f-A^{i}\partial_{i}u-Du]\right)$, are obtained in a similar manner. Therefore, we can take the inner product of \eqref{eq:47} with $2\partial_{x}^{\alpha}u$ in $L^{2}$ for every $|\alpha|\leq m$, to get
\begin{align}
\langle A^{0}\partial_{x}^{\alpha}u_{t},2\partial_{x}^{\alpha}u\rangle&-\langle B^{ij}\partial_{i}\partial_{j}\partial_{x}^{\alpha}u,2\partial_{x}^{\alpha}u\rangle=\langle A^{0}\partial_{x}^{\alpha}\left((A^{0})^{-1}f\right),2\partial_{x}^{\alpha}u\rangle\nonumber\\
&-\langle A^{0}\partial_{x}^{\alpha}\left((A^{0})^{-1}A^{i}\partial_{i}u\right),2\partial_{x}^{\alpha}u\rangle-\langle A^{0}\partial_{x}^{\alpha}\left((A^{0})^{-1}Du\right),2\partial_{x}^{\alpha}u\rangle\nonumber\\
&+\langle A^{0}G_{\alpha}\left((A^{0})^{-1}B^{ij},\partial_{i}\partial_{j}u\right),2\partial_{x}^{\alpha}u\rangle\label{eq:48}
\end{align}
where 
\begin{equation}
2\langle A^{0}\partial_{x}^{\alpha}u_{t},\partial_{x}^{\alpha}u\rangle=\frac{d}{dt}\langle A^{0}\partial_{x}^{\alpha}u,\partial_{x}^{\alpha}u\rangle-\langle \partial_{x}^{\alpha}u,(\partial_{t}A^{0})\partial_{x}^{\alpha}u\rangle,\label{eq:49}
\end{equation}
holds true since each inner product appearing is a $L^{1}(0,T)$ function by assumptions \textbf{H2} and \textbf{H3}. The commutator term in \eqref{eq:48} can be estimated by means of \eqref{eq:ce1}, \eqref{eq:ce2} and theorem 3.1. Similarly, the rest of the terms in the right hand side of \eqref{eq:48} and the last term in the right hand side of \eqref{eq:49} can be estimated by means of integration by parts,  theorems 3.1 and 3.3 together with assumptions \textbf{H2}-\textbf{H6}, yielding
\small
\begin{align}
\langle A^{0}\partial_{x}^{\alpha}\left((A^{0})^{-1}A^{i}\partial_{i}u\right),2\partial_{x}^{\alpha}u\rangle\leq& C\left\lbrace\sum_{i=1}^{d}\|A^{i}\|_{\bar{s}}\|u\|_{m}\|\nabla u\|_{m}+\|\partial_{i}A^{i}\|_{s-1}\|u\|_{m}^{2}\right.\nonumber\\
+&\left.\|A^{0}\|_{\bar{s}}\|(A^{0})^{-1}\|_{\bar{s}}\sum_{i=1}^{d}\|A^{i}\|_{\bar{s}}\|u\|_{m}^{2}\right\rbrace,\label{eq:inest1}\\
\langle A^{0}\partial_{x}^{\alpha}\left((A^{0})^{-1}Du\right),2\partial_{x}^{\alpha}u\rangle\leq& C\left\lbrace\|D\|_{\bar{s}}+\|A^{0}\|_{\bar{s}}\|(A^{0})^{-1}\|_{\bar{s}}\|D\|_{\bar{s}}\right\rbrace\|u\|_{m}^{2},\label{eq:inest2}\\
\langle\partial_{x}^{\alpha}u,(\partial_{t}A^{0})\partial_{x}^{\alpha}u\rangle\leq&\|\partial_{t}A^{0}\|_{s-1}\|u\|_{m}^{2},\label{eq:inest3}\\
\langle A^{0}G_{\alpha}\left((A^{0})^{-1}B^{ij},\partial_{i}\partial_{j}u\right),2\partial_{x}^{\alpha}u\rangle\leq& C\|A^{0}\|_{\bar{s}}\|(A^{0})^{-1}\|_{\bar{s}}\sum_{i,j=1}^{d}\|B^{ij}\|_{\bar{s}}\|\nabla u\|_{m}\|u\|_{m},\label{eq:inest4}\\
\langle A^{0}\partial_{x}^{\alpha}\left((A^{0})^{-1}f\right),\partial_{x}^{\alpha}u\rangle\leq& C\left\lbrace\|\nabla u\|_{m}+\|A^{0}\|_{\bar{s}}\|(A^{0})^{-1}\|_{\bar{s}}\|u\|_{m}\right\rbrace\|f\|_{m-1},\label{eq:inest5}
\end{align}
\normalsize
valid for all $0\leq|\alpha|\leq m$. Where the constant $C$ is independent of the matrix coefficients, $t\in[0,T]$, $u$ and $f$. By Garding's inequality (see, \cite{agmon}, \cite{mclean} and \cite{mizo}), there are two constants $G_{0}=G_{0}(\eta, g)>0$ and $\gamma_{0}=\gamma_{0}(\eta)\geq 0$ such that
\begin{eqnarray}
-\sum_{|\alpha|\leq m}\langle B^{ij}\partial_{i}\partial_{j}\partial_{x}^{\alpha}u,\partial_{x}^{\alpha}u\rangle&\geq&\sum_{|\alpha|\leq m}G_{0}\|\partial_{x}^{\alpha}u\|_{1}^{2}-\gamma_{0}\|\partial_{x}^{\alpha}u\|^{2}\nonumber\\
&=&G_{0}\left(\|u\|_{m}^{2}+\|\nabla u\|_{m}^{2}\right)-\gamma_{0}\|u\|_{m}^{2}.\label{eq:garding}
\end{eqnarray}
By using estimates \eqref{eq:inest1}- \eqref{eq:garding} into \eqref{eq:48} and \eqref{eq:49} yields
\begin{align}
 \frac{d}{dt}\sum_{|\alpha|\leq m}\langle A^{0}\partial_{x}^{\alpha}u,\partial_{x}^{\alpha}u\rangle+&G_{0}\left(\|u\|_{m}^{2}+\|\nabla u\|_{m}^{2}\right)\leq C\left\lbrace\|\nabla u\|_{m}\|f\|_{m-1}+\sum_{i=1}^{d}\|A^{i}\|_{\bar{s}}\|u\|_{m}^{2}\right.\nonumber\\
	&+\sum_{i=1}^{d}\|A^{i}\|_{\bar{s}}\|u\|_{m}\|\nabla u\|_{m}+\|D\|_{\bar{s}}\|u\|_{m}^{2}+\|\partial_{t}A^{0}\|_{s-1}\|u\|_{m}^{2}\nonumber\\
	&+\left.\sum_{i,j=1}^{d}\|B^{ij}\|_{\bar{s}}\|u\|_{m}\|\nabla u\|_{m}+\gamma_{0}\|u\|_{m}^{2}\right\rbrace,\label{eq:412}
\end{align}
for some constant $C=C(g)>0$ independent of $u$ and $t\in[0,T]$. Now, we use Cauchy's weighted inequality to \emph{isolate} the term $\|\nabla u\|_{m}^{2}$, that is,
\begin{eqnarray*}
\|\nabla u\|_{m}\|f\|_{m-1}&\leq& \epsilon_{1}\|\nabla u\|_{m}^{2}+\frac{\|f\|_{m-1}^{2}}{4\epsilon_{1}},\\
\sum_{i=1}^{d}\|A^{i}\|_{\bar{s}}\|u\|_{m}\|\nabla u\|_{m}&\leq&\epsilon_{2}\|\nabla u\|_{m}^{2}+\frac{2^{d}\sum_{i=1}^{d}\|A^{i}\|_{\bar{s}}^{2}\|u\|_{m}^{2}}{4\epsilon_{2}},\\
\sum_{i,j=1}^{d}\|B^{ij}\|_{\bar{s}}\|u\|_{m}\|\nabla u\|_{m}&\leq&\epsilon_{3}\|\nabla u\|_{m}^{2}+\frac{2^{d^{2}}\sum_{i,j=1}^{d}\|B^{ij}\|_{\bar{s}}^{2}\|u\|_{m}^{2}}{4\epsilon_{3}},
\end{eqnarray*}
where $\epsilon_{1},\epsilon_{2},\epsilon_{3}>0$ are chosen so that $C(\sum_{i=1}^{3}\epsilon_{i})\leq\frac{1}{2}G_{0}$. Therefore,
 \begin{equation}
 \frac{d}{dt}\sum_{|\alpha|\leq m}\langle A^{0}\partial_{x}^{\alpha}u,\partial_{x}^{\alpha}u\rangle+\frac{G_{0}}{2}\left(\|u\|_{m}^{2}+\|\nabla u\|_{m}^{2}\right)\leq C\left\lbrace\left(\mu_{0}(t)+\mu_{1}(t)\right)\|u\|_{m}^{2}+\|f\|_{m-1}^{2}\right\rbrace,\label{eq:415}
\end{equation}
where $\mu_{0},\mu_{1}\in L^{1}(0,T)$, as consequence of \textbf{H2}-\textbf{H6}. Integrate with respect to $t\in[0,T]$, use \textbf{H1} and apply Gronwall's inequality to obtain,
\begin{align}
\|u(t)\|_{m}^{2}+&\int_{0}^{t}\left(\|u(\tau)\|_{m}^{2}+\|\nabla u(\tau)\|_{m}^{2}\right)d\tau \leq\nonumber\\
&C_{0} e^{C_{0}\int_{0}^{t}(\mu_{0}(\tau)+\mu_{1}(\tau))d\tau}\left\lbrace\|u_{0}\|_{m}^{2}+\int_{0}^{t}\|f(\tau)\|_{m-1}^{2}d\tau\right\rbrace,\label{eq:417}
\end{align}
for all $t\in[0,T]$, where $C_{0}=C_{0}(\eta,g,a_{0})$ is such that 
\begin{equation}
C_{0}\geq\frac{\max\left\lbrace a_{1},C\right\rbrace}{\min\left\lbrace a_{0},\frac{G_{0}}{2}\right\rbrace}\label{eq:416}.
\end{equation}
\subsection{Estimate for time derivative}\label{time}
Next, we take the inner product in $L^{2}$ of \eqref{eq:47} with $2\partial_{x}^{\alpha}u_{t}$, but in this case for $|\alpha|\leq m-1$,
\begin{align}
&\langle A^{0}\partial_{x}^{\alpha}u_{t},2\partial_{x}^{\alpha}u_{t}\rangle-\langle B^{ij}\partial_{i}\partial_{j}(\partial_{x}^{\alpha}u),2\partial_{x}^{\alpha}u_{t}\rangle=\langle A^{0}\partial_{x}^{\alpha}\left((A^{0})^{-1}f\right),2\partial_{x}^{\alpha}u_{t}\rangle\nonumber\\
-&\langle A^{0}\partial_{x}^{\alpha}\left((A^{0})^{-1}A^{i}\partial_{i}u\right),2\partial_{x}^{\alpha}u_{t}\rangle-\langle A^{0}\partial_{x}^{\alpha}\left((A^{0})^{-1}Du\right),2\partial_{x}^{\alpha}u_{t}\rangle\nonumber\\
+&\langle A^{0}G_{\alpha}\left((A^{0})^{-1}B^{ij},\partial_{i}\partial_{j}u\right),2\partial_{x}^{\alpha}u_{t}\rangle.\label{eq:418}
\end{align}
In order to properly estimate \eqref{eq:418} we need to work out the term $\langle B^{ij}\partial_{i}\partial_{j}(\partial_{x}^{\alpha}u),\partial_{x}^{\alpha}u_{t}\rangle$. Let us begin with the identity
\begin{align}
\langle B^{ij}\partial_{i}\partial_{j}(\partial_{x}^{\alpha}u),\partial_{x}^{\alpha}u_{t}\rangle =-\langle B^{ij}\partial_{j}\partial_{x}^{\alpha}u,\partial_{i}\partial_{x}^{\alpha}u_{t}\rangle -\langle\left(\partial_{i}B^{ij}\right)\partial_{j}\partial_{x}^{\alpha}u,\partial_{x}^{\alpha}u_{t}\rangle,\label{eq:id1}
\end{align}
which is justified by the \textbf{ER} assumption since $\partial_{x}^{\alpha}u_{t}\in H^{1}$. Given that $B^{ij}\in \widehat{H}^{s}$ and $s-1>\frac{d}{2}$, theorem 3.1 and corollary 1 assure that, $ B^{ij}\partial_{j}\partial_{x}^{\alpha}u\in L^{2}$ and $\left(\partial_{i}B^{ij}\right)\partial_{j}\partial_{x}^{\alpha}u\in L^{2}$, respectively. Using a product rule for the time derivative yields
\begin{align}
\langle B^{ij}\partial_{i}\partial_{j}(\partial_{x}^{\alpha}u),\partial_{x}^{\alpha}u_{t}\rangle&=-\frac{d}{dt}\langle B^{ij}\partial_{j}\partial_{x}^{\alpha}u,\partial_{i}\partial_{x}^{\alpha}u\rangle\label{eq:419}\\
&+\langle\partial_{t}\left(B^{ij}\partial_{j}\partial_{x}^{\alpha}u\right),\partial_{i}\partial_{x}^{\alpha}u\rangle\label{eq:420}\\
&-\langle\left(\partial_{i}B^{ij}\right)\partial_{j}\partial_{x}^{\alpha}u,\partial_{x}^{\alpha}u_{t}\rangle.\nonumber
\end{align}
For \eqref{eq:419} we have the following computations, for each $|\alpha|\leq m-1$,
\begin{align}
-&\frac{d}{dt}\langle B^{ij}\partial_{j}\partial_{x}^{\alpha}u,\partial_{i}\partial_{x}^{\alpha}u\rangle=\frac{d}{dt}\langle\partial_{i}\left(B^{ij}\partial_{j}\partial_{x}^{\alpha}u\right),\partial_{x}^{\alpha}u\rangle\nonumber\\
&=\frac{d}{dt}\langle B^{ij}\partial_{i}\partial_{j}\partial_{x}^{\alpha}u,\partial_{x}^{\alpha}u\rangle\nonumber+\frac{d}{dt}\langle\left(\partial_{i}B^{ij}\right)\partial_{j}\partial_{x}^{\alpha}u,\partial_{x}^{\alpha}u\rangle\nonumber\\
&=\frac{d}{dt}\langle B^{ij}\partial_{i}\partial_{j}\partial_{x}^{\alpha}u,\partial_{x}^{\alpha}u\rangle\nonumber+\langle\partial_{j}\partial_{x}^{\alpha}u_{t},\left(\partial_{i}B^{ij}\right)\partial_{x}^{\alpha}u\rangle+\langle\partial_{j}\partial_{x}^{\alpha}u,\partial_{t}\left(\partial_{i}B^{ij}\partial_{x}^{\alpha}u\right)\rangle\nonumber\\
&=\frac{d}{dt}\langle B^{ij}\partial_{i}\partial_{j}\partial_{x}^{\alpha}u,\partial_{x}^{\alpha}u\rangle\nonumber\\
&-\langle\partial_{x}^{\alpha}u_{t},\left(\partial_{j}\partial_{i}B^{ij}\right)\partial_{x}^{\alpha}u\rangle\label{eq:421}\\
&+\langle\partial_{j}\partial_{x}^{\alpha}u,\left(\partial_{t}\partial_{i}B^{ij}\right)\partial_{x}^{\alpha}u\rangle\label{eq:423}\\
&+\langle\partial_{j}\partial_{x}^{\alpha}u,\left(\partial_{i}B^{ij}\right)\partial_{x}^{\alpha}u_{t}\rangle-\langle\partial_{x}^{\alpha}u_{t},\left(\partial_{i}B^{ij}\right)\partial_{j}\partial_{x}^{\alpha}u\rangle\label{eq:424}.
\end{align}
Where we need to justify the integration by parts carried on terms \eqref{eq:421}-\eqref{eq:424}. First note that each term in \eqref{eq:424} can be justified as in \eqref{eq:id1}. Moreover, due to the symmetry of each matrix $B^{ij}$ the line \eqref{eq:424} equals to zero. In the case of \eqref{eq:421}, assumption \textbf{H5} yields that $\partial_{i}\partial_{j}B^{ij}\in H^{s-2}$ and since $\partial_{x}^{\alpha}u\in H^{1}$ we apply theorem 3.2 directly (with $m=s-2$, $n=1$, $k=0$ and $m+n-k=s-1>\frac{d}{2}$) to get that $\partial_{i}\partial_{j}B^{ij}\partial_{x}^{\alpha}u\in L^{2}$. The same argument applies to \eqref{eq:423} since \textbf{H5} states that $\partial_{t}\partial_{i}B^{ij}\in H^{s-2}$.\\
For the term in \eqref{eq:420}, the \textbf{ER} assumption and the symmetry of each $B^{ij}$ lead us to,
\begin{align*}
\langle\partial_{t}\left(B^{ij}\partial_{j}\partial_{x}^{\alpha}u\right),\partial_{i}\partial_{x}^{\alpha}u\rangle&=\langle\left(\partial_{t}B^{ij}\right)\partial_{j}\partial_{x}^{\alpha}u,\partial_{i}\partial_{x}^{\alpha}u\rangle\nonumber\\
&-\langle\partial_{x}^{\alpha}u_{t},\left(\partial_{j}B^{ij}\right)\partial_{i}\partial_{x}^{\alpha}u\rangle\nonumber\\
&-\langle\partial_{x}^{\alpha}u_{t}, B^{ij}\partial_{i}\partial_{j}\partial_{x}^{\alpha}u\rangle.
\end{align*}
By using the last two identities into \eqref{eq:419} and \eqref{eq:420} yields,
\begin{align}
2\langle B^{ij}\partial_{i}\partial_{j}\partial_{x}^{\alpha}u,\partial_{x}^{\alpha}u_{t}\rangle=&\frac{d}{dt}\langle B^{ij}\partial_{i}\partial_{j}\partial_{x}^{\alpha}u,\partial_{x}^{\alpha}u\rangle-2\langle\left(\partial_{j}B^{ij}\right)\partial_{i}\partial_{x}^{\alpha}u,\partial_{x}^{\alpha}u_{t}\rangle\nonumber\\
-&\langle\partial_{x}^{\alpha}u_{t},\left(\partial_{j}\partial_{i}B^{ij}\right)\partial_{x}^{\alpha}u\rangle+\langle\partial_{j}\partial_{x}^{\alpha}u,\left(\partial_{t}\partial_{i}B^{ij}\right)\partial_{x}^{\alpha}u\rangle\nonumber\\
+&\langle\left(\partial_{t}B^{ij}\right)\partial_{j}\partial_{x}^{\alpha}u,\partial_{i}\partial_{x}^{\alpha}u\rangle.\label{eq:425}
\end{align}
If we use \eqref{eq:425} into \eqref{eq:418} we are left with the identity,
\begin{equation}
\label{eq:utest}
\begin{aligned}
\langle A^{0}\partial_{x}^{\alpha}u,\partial_{x}^{\alpha}u\rangle&-\frac{d}{dt}\langle B^{ij}\partial_{i}\partial_{j}\partial_{x}^{\alpha}u,\partial_{x}^{\alpha}u\rangle=\langle A^{0}\partial_{x}^{\alpha}\left((A^{0})^{-1}f\right),2\partial_{x}^{\alpha}u_{t}\rangle\\
&-\langle A^{0}\partial_{x}^{\alpha}\left((A^{0})^{-1}A^{i}\partial_{i}u\right),2\partial_{x}^{\alpha}u_{t}\rangle+\langle A^{0}\partial_{x}^{\alpha}\left((A^{0})^{-1}Du\right),2\partial_{x}^{\alpha}u_{t}\rangle\\
&+\langle A^{0}G_{\alpha}\left((A^{0})^{-1}B^{ij},\partial_{i}\partial_{j}u\right),2\partial_{x}^{\alpha}u_{t}\rangle-2\langle\left(\partial_{j}B^{ij}\right)\partial_{i}\partial_{x}^{\alpha}u,\partial_{x}^{\alpha}u_{t}\rangle\\
&-\langle\partial_{x}^{\alpha}u_{t},\left(\partial_{j}\partial_{i}B^{ij}\right)\partial_{x}^{\alpha}u\rangle+\langle\partial_{j}\partial_{x}^{\alpha}u,\left(\partial_{t}\partial_{i}B^{ij}\right)\partial_{x}^{\alpha}u\rangle\\
&+\langle\left(\partial_{t}B^{ij}\right)\partial_{j}\partial_{x}^{\alpha}u,\partial_{i}\partial_{x}^{\alpha}u\rangle,\quad\mbox{for all}~|\alpha|\leq m-1.
\end{aligned}
\end{equation}
As it was previously done for \eqref{eq:48}, we estimate the right hand side of \eqref{eq:utest}. By taking into account that $|\alpha|\leq m-1$ and by \textbf{H2}, there is a positive constant $C=C(g)$ such that
\begin{align}
\langle A^{0}\partial_{x}^{\alpha}u,\partial_{x}^{\alpha}u\rangle&-\frac{d}{dt}\langle B^{ij}\partial_{i}\partial_{j}\partial_{x}^{\alpha}u,\partial_{x}^{\alpha}u\rangle\leq C\left\lbrace\sum_{i=1}^{d}\|A^{i}\|_{\bar{s}}\|\nabla u\|_{m-1}\|u_{t}\|_{m-1}\right.\nonumber\\
&+\|f\|_{m-1}\|u_{t}\|_{m-1}+\sum_{i,j=1}^{d}\|B^{ij}\|_{\bar{s}}\|\nabla u\|_{m-1}\|u_{t}\|_{m-1}\nonumber\\
&+\sum_{i,j=1}^{d}\|\partial_{t}B^{ij}\|_{s-1}\|\nabla u\|_{m-1}^{2}+\sum_{i,j=1}^{d}\|\partial_{t}B^{ij}\|_{s-1}\|\nabla u\|_{m-1}\|u\|_{m-1}\nonumber\\
&+\left. +\|D\|_{\bar{s}}\|u\|_{m-1}\|u_{t}\|_{m-1}+\sum_{i,j=1}^{d}\|B^{ij}\|_{\bar{s}}\|u\|_{m-1}\|u_{t}\|_{m-1}\right\rbrace.\label{eq:435}
\end{align}
As we did for the estimate in \eqref{eq:412}, we apply Cauchy's weighted inequality in order to \emph{isolate} the term $\|u_{t}\|_{m-1}^{2}$. Then, we add all the terms in \eqref{eq:435} with respect to $|\alpha|\leq m-1$, use \textbf{H1} and integrate from $0$ to $t\in[0,T]$ to obtain
\begin{eqnarray}
\int_{0}^{t}\|u_{t}(\tau)\|_{m-1}^{2}d\tau&-&\sum_{|\alpha|\leq m-1}\langle B^{ij}(t)\partial_{i}\partial_{j}\partial_{x}^{\alpha}u(t),\partial_{x}^{\alpha}u(t)\rangle\leq\nonumber\\
&\leq&-\sum_{|\alpha|\leq m-1}\langle B^{ij}(0)\partial_{i}\partial_{j}\partial_{x}^{\alpha}u_{0},\partial_{x}^{\alpha}u_{0}\rangle+C\int_{0}^{t}\|f(\tau)\|_{m-1}^{2}d\tau\nonumber\\
&+&C\int_{0}^{t}\left(\mu_{0}(\tau)+\mu_{1}(\tau)\right)\|u(\tau)\|_{m}^{2}d\tau,\label{eq:437}
\end{eqnarray}
for some constant $C=C(g,a_{0})$. Notice that the \textbf{ER} assumption has been used to justify that term $\langle B^{ij}(0)\partial_{i}\partial_{j}\partial_{x}^{\alpha}u_{0},\partial_{x}^{\alpha}u_{0}\rangle$ is finite for each $|\alpha|\leq m-1$. As consequence of \textbf{H6} and integration by parts, there is some positive constant $C=C(g)$ such that,
\begin{align*}
-\sum_{|\alpha|\leq m-1}\langle B^{ij}(0)\partial_{i}\partial_{j}\partial_{x}^{\alpha}u_{0},\partial_{x}^{\alpha}u_{0}\rangle\leq C\|u_{0}\|_{m}^{2}.
\end{align*}
This estimate, together with Garding's inequality (see equation \eqref{eq:garding}), applied into \eqref{eq:437} yields
\begin{eqnarray}
\int_{0}^{t}\|u_{t}(\tau)\|_{m-1}^{2}d\tau&+&G_{0}\left(\|u(t)\|_{m-1}^{2}+\|\nabla u(t)\|_{m-1}^{2}\right)\leq \gamma_{0}\|u(t)\|_{m-1}^{2}\nonumber\\
&+&C\left\lbrace\|u_{0}\|_{m}^{2}+\int_{0}^{t}\|f(\tau)\|_{m-1}^{2}d\tau\right.\nonumber\\
&+&\left.\int_{0}^{t}(\mu_{0}(\tau)+\mu_{1}(\tau))\|u(\tau)\|_{m}^{2}d\tau\right\rbrace\label{eq:439}
\end{eqnarray}
for all $t\in[0,T]$. The term, $\gamma_{0}\|u(t)\|_{m-1}^{2}$, can be controlled by integrating the inequality,
\begin{align*}
\frac{d}{dt}\|u\|_{m-1}^{2}\leq C\left(\epsilon\|u_{t}\|_{m-1}^{2}+\frac{\|u\|_{m-1}^{2}}{4\epsilon}\right),\nonumber
\end{align*}
for some $\epsilon>0$ small enough. Therefore, by Gronwall's inequality, 
\begin{align}
\int_{0}^{t}\|u_{t}(\tau)\|_{m-1}^{2}d\tau&+\|u(t)\|_{m}^{2}\leq\nonumber\\
& Ce^{C\int_{0}^{t}\left(\mu_{0}(\tau)+\mu_{1}(\tau)\right)}\left\lbrace\|u_{0}\|_{m}^{2}+\int_{0}^{t}\|f(\tau)\|_{m-1}^{2}d\tau\right\rbrace\label{eq:440}
\end{align}
for all $t\in[0,T]$. Finally, by adding up \eqref{eq:417} and \eqref{eq:440}, we obtain the a priori estimate \eqref{eq:441}.\\
Now, the mollification of $u$ satisfies the following equation
\begin{equation}
A^{0}u_{t}^{\epsilon}+A^{i}\partial_{i}u^{\epsilon}+Du^{\epsilon}-B^{ij}\partial_{i}\partial_{j}u^{\epsilon}=f^{\epsilon}+F^{\epsilon},\label{eq:444}
\end{equation}
where $F^{\epsilon}\rightarrow 0$ in $H^{m}$ for a.a. $t\in[0,T]$ as $\epsilon\rightarrow 0$ (see \cite[Theorem 1.5.1 and Theorem 1.5.6]{cherrier}, for example). We apply estimate \eqref{eq:441} to equation \eqref{eq:444} with source term $f^{\epsilon}+F^{\epsilon}$ and initial condition $u^{\epsilon}(x,0)=u_{0}^{\epsilon}(x)$. Then, by letting $\epsilon\rightarrow 0$, we drop out the \textbf{ER} assumption to conclude.
\section{Appendix B: Energy estimate of linear coupled equation}
In this section we prove Theorem \ref{energydec}. We obtain the energy estimate for equations \eqref{eq:514}-\eqref{eq:516} independent of $0\leq\delta<1$. As before, we assume that $U^{\delta}$ is a solution of \eqref{eq:514}-\eqref{eq:516} satisfying the \textbf{ER} assumption. For simplicity we write $U=(u,v,w)^{\top}$ instead of $U^{\delta}=(u^{\delta},v^{\delta},w^{\delta})^{\top}$. We begin by estimating \eqref{eq:514}. Then, consider the following identity
\begin{align}
\frac{d}{dt}\langle A_{1}^{0}\partial_{x}^{\alpha}u,\partial_{x}^{\alpha}u\rangle&-\delta\langle\Delta\partial_{x}^{\alpha}u,2\partial_{x}^{\alpha}u\rangle=\langle(A_{1}^{0}\partial_{x}^{\alpha}\left((A_{1}^{0})^{-1}f_{1}\right),2\partial_{x}^{\alpha}u\rangle\nonumber\\
&-\langle A_{1}^{0}\partial_{x}^{\alpha}\left((A_{1}^{0})^{-1}A_{11}^{i}\partial_{i}u\right),2\partial_{x}^{\alpha}u\rangle-\langle A_{1}^{0}\partial_{x}^{\alpha}\left((A_{1}^{0})^{-1}A_{12}^{i}\partial_{i}v\right),2\partial_{x}^{\alpha}u\rangle\nonumber\\
&+\delta\langle A_{1}^{0}G_{\alpha}\left((A_{1}^{0})^{-1},\Delta u\right),2\partial_{x}^{\alpha}u\rangle+\langle \partial_{x}^{\alpha}u,(\partial_{t}A_{1}^{0})\partial_{x}^{\alpha}u\rangle.\label{eq:520}
\end{align}
Given that, there is no diffusion term for $u$, we avoid of any term involving $\|\nabla u\|_{m}$ when estimating \eqref{eq:520}. For example, the second term in the right hand side of \eqref{eq:520} can be written as 
\begin{eqnarray}
\langle A_{1}^{0}\partial_{x}^{\alpha}\left((A_{1}^{0})^{-1}A_{11}^{i}\partial_{i}u\right),2\partial_{x}^{\alpha}u\rangle&=&\langle A_{11}^{i}\partial_{x}^{\alpha}\partial_{i}u,2\partial_{x}^{\alpha}u\rangle\nonumber\\
&+&\langle A_{1}^{0}G_{\alpha}\left((A_{1}^{0})^{-1}A_{11}^{i},\partial_{i}u\right), 2\partial_{x}^{\alpha}u\rangle,\nonumber
\end{eqnarray}
where we have to integrate by parts the first term in the right hand side of this identity and use assumption \textbf{I}, so that,
\begin{eqnarray}
\langle A_{1}^{0}\partial_{x}^{\alpha}\left((A_{1}^{0})^{-1}A_{11}^{i}\partial_{i}u\right),2\partial_{x}^{\alpha}u\rangle&=&-\langle(\partial_{i}A_{11}^{i})\partial_{x}^{\alpha}u,\partial_{x}^{\alpha}u\rangle\nonumber\\
&+&\langle A_{1}^{0}G_{\alpha}\left((A_{1}^{0})^{-1}A_{11}^{i},\partial_{i}u\right), 2\partial_{x}^{\alpha}u\rangle.\nonumber
\end{eqnarray}
The Cauchy-Schwarz inequality, together with theorems 3.1 and 3.3 and assumption \textbf{II} yield 
\begin{eqnarray}
\sum_{|\alpha|=0}^{m}\langle A_{1}^{0}\partial_{x}^{\alpha}\left((A_{1}^{0})^{-1}A_{11}^{i}\partial_{i}u\right),2\partial_{x}^{\alpha}u\rangle&\leq&C\left(\sum_{i=1}^{d}\|A^{i}_{11}\|_{\bar{s}}\right)\|u\|^{2}_{m},\label{eq:521}
\end{eqnarray}
for some positive constant $C=C(g)$. For the term involving $f_{1}$ in \eqref{eq:520} we have that,
\begin{equation}
\sum_{|\alpha|=0}^{m}\langle(A_{1}^{0}\partial_{x}^{\alpha}\left((A_{1}^{0})^{-1}f_{1}\right),2\partial_{x}^{\alpha}u\rangle\leq C\|f_{1}\|_{m}\|u\|_{m},
\label{eq:522}
\end{equation}
which is different from the parabolic case (see \eqref{eq:inest5}) because we are not integrating by parts to relief $f_{1}$ from one degree of differentiation. The rest of the terms in \eqref{eq:520} can be dealt similarly as in the a priori estimate for the parabolic case in Appendix A. Therefore,
\begin{eqnarray}
\frac{d}{dt}\sum_{|\alpha|=0}^{m}\langle A_{1}^{0}\partial_{x}^{\alpha}u,\partial_{x}^{\alpha}u\rangle+2\delta\|\nabla u\|_{m}^{2}&\leq& C\left\lbrace\|f_{1}\|_{m}\|u\|_{m}+\left(\sum_{i=1}^{d}\|A^{i}_{11}\|_{\bar{s}}\right)\|u\|_{m}^{2}\right.\nonumber\\
&+&+\|\partial_{t}A_{1}^{0}\|_{s-1}\|u\|_{m}^{2}+\delta\|u\|_{m}\|\nabla u\|_{m}\nonumber\\
&+&\left.\left(\sum_{i=1}^{d}\|A^{i}_{12}\|_{\bar{s}}\right)\|u\|_{m}\left(\|v\|_{m}+\|\nabla v\|_{m}\right)\right\rbrace,\nonumber\\
\label{eq:523}
\end{eqnarray}
for some positive constant $C=C(g)$. Since, 
\begin{align*}
\delta\|u\|_{m}\|\nabla u\|_{m}\leq\delta\left(\frac{\|u\|_{m}^{2}}{2\epsilon}+\frac{\epsilon\|\nabla u\|_{m}^{2}}{2}\right),\quad\mbox{for all}~\epsilon>0,
\end{align*}
if we take $\epsilon=\frac{1}{C}$, it follows that,
\begin{align}
\frac{d}{dt}\sum_{|\alpha|=0}^{m}\langle A_{1}^{0}\partial_{x}^{\alpha}u,\partial_{x}^{\alpha}u\rangle&\leq C\left\lbrace\|f_{1}\|_{m}^{2}+\|u\|_{m}^{2}+\left(\sum_{i=1}^{d}\|A_{11}^{i}\|_{\bar{s}}\right)\|u\|_{m}^{2}+\|\partial_{t}A_{1}^{0}\|_{s-1}\|u\|_{m}^{2}\right.\nonumber\\
&+\left. \left(\sum_{i=1}^{d}\|A_{12}^{i}\|_{\bar{s}}\right)\|u\|_{m}(\|v\|_{m}+\|\nabla v\|_{m})\right\rbrace.\label{eq:526}
\end{align}
In the same manner, we obtain the energy estimate for $w$,
\begin{align}
\frac{d}{dt}\sum_{|\alpha|=0}^{m}\langle A_{3}^{0}\partial_{x}^{\alpha}u,\partial_{x}^{\alpha}u\rangle&\leq C\left\lbrace\|f_{3}\|_{m}^{2}+\|w\|_{m}^{2}+\left(\sum_{i=1}^{d}\|A_{33}^{i}\|_{\bar{s}}\right)\|w\|_{m}^{2}+\|\partial_{t}A_{3}^{0}\|_{s-1}\|w\|_{m}^{2}\right.\nonumber\\
&+\left. \left(\sum_{i=1}^{d}\|A_{32}^{i}\|_{\bar{s}}\right)\|w\|_{m}(\|v\|_{m}+\|\nabla v\|_{m})+\|D_{0}\|_{\bar{s}}\|w\|_{m}^{2}\right\rbrace.\label{eq:529}
\end{align}
To derive the energy estimate for $v$ observe that, since $u,w\in\mathcal{C}([0,T];H^{m})$, we can take $f=f_{2}-A^{i}_{21}\partial_{i}u-A^{i}_{23}\partial_{i}w$ an treat the estimate as in Appendix A. In particular, we follow the steps that led us to \eqref{eq:412}. Therefore, 
\begin{align}
\frac{d}{dt}\sum_{|\alpha|=0}^{m}\langle A_{1}^{0}\partial_{x}^{\alpha}u,\partial_{x}^{\alpha}u\rangle\leq\frac{C\epsilon}{2}\|\nabla v\|_{m}^{2}+C\left\lbrace\|f_{1}\|_{m}^{2}+(\mu_{0}(t)+\mu_{1}(t))\left(\|u\|_{m}^{2}+\|v\|_{m}^{2}\right)\right\rbrace,\label{eq:531}
\end{align}
\begin{align}
\frac{d}{dt}\sum_{|\alpha|=0}^{m}\langle A_{2}^{0}\partial_{x}^{\alpha}u,\partial_{x}^{\alpha}u\rangle&+G_{0}\left(\|v\|_{m}^{2}+\|\nabla v\|_{m}^{2}\right)\leq\frac{4C\epsilon}{2}\|\nabla v\|_{m}^{2}+C\left\lbrace\|f_{2}\|_{m-1}^{2}\right.\nonumber\\
&+C\left.(\mu_{0}(t)+\mu_{1}(t))\left(\|u\|_{m}^{2}+\|v\|_{m}^{2}+\|w\|_{m}^{2}\right)\right\rbrace,
\label{eq:532}
\end{align}
\begin{align}
\frac{d}{dt}\sum_{|\alpha|=0}^{m}\langle A_{3}^{0}\partial_{x}^{\alpha}u,\partial_{x}^{\alpha}u\rangle\leq\frac{C\epsilon}{2}\|\nabla v\|_{m}^{2}+C\left\lbrace\|f_{3}\|_{m}^{2}+(\mu_{0}(t)+\mu_{1}(t))\left(\|w\|_{m}^{2}+\|v\|_{m}^{2}\right)\right\rbrace.\label{eq:533}
\end{align}
We add up \eqref{eq:531}, \eqref{eq:532}, \eqref{eq:533} and take $\epsilon>0$ such that $G_{0}-3C\epsilon=\frac{1}{2}$ to obtain
\begin{align}
\frac{d}{dt}\mathcal{E}_{m}^{2}(u,v,w)+\|v\|_{m}^{2}+\|\nabla v\|_{m}^{2}&\leq C_{1}\left\lbrace\mathcal{F}_{m}^{2}(f_{1},f_{2},f_{3})\right.\nonumber\\
&+\left.(\mu_{0}(t)+\mu_{1}(t))\mathcal{E}_{m}^{2}(u,v,w)\right\rbrace,\label{eq:534}
\end{align}
for some positive constant $C_{1}=C_{1}(g,\eta,a_{0})$, where
\begin{align*}
\mathcal{E}_{m}^{2}(u,v,w):=\sum_{|\alpha|=0}^{m}\langle A_{1}^{0}\partial_{x}^{\alpha}u,\partial_{x}^{\alpha}u\rangle+\sum_{|\alpha|=0}^{m}\langle A_{2}^{0}\partial_{x}^{\alpha}u,\partial_{x}^{\alpha}u\rangle+\sum_{|\alpha|=0}^{m}\langle A_{3}^{0}\partial_{x}^{\alpha}u,\partial_{x}^{\alpha}u\rangle.
\end{align*} 
After integrating we use \textbf{III} and Gronwall's inequality to obtain, 
\begin{equation*}
\|u(t)\|_{m}^{2}+\|v(t)\|_{m}^{2}+\|w(t)\|_{m}^{2}+\int_{0}^{t}\|v(\tau)\|_{m}^{2}+\|\nabla v(\tau)\|_{m}^{2}d\tau\leq K_{0}^{2}(t)\Phi_{0}^{2}(t),
\end{equation*}
for all $t\in[0,T]$. We proceed as in section \ref{time} to obtain the estimates for $U_{t}$. 
Hence, we have proven that, for every $0\leq\delta<1$, the solution, $(u^{\delta},v^{\delta},w^{\delta})^{\top}(t)$, of the strongly parabolic extended system \eqref{eq:512}, satisfies \eqref{eq:543} for all $t\in[0,T]$. This concludes the result.
\section*{Acnowledgments}
Thanks to professor Ram\'on G. Plaza for his comments and guidance. This work was supported by CONACyT (M\'exico), through a scholarship for doctoral studies, grant no. 465484, by DGAPA-UNAM program PAPIIT, grant IN100318 and by CONACyT through the postdoctoral research project  A1-S-10457.

\bibliographystyle{plain} 
\bibliography{feVSkawabiblio}

\begin{thebibliography}{10}

\bibitem{agmon}
S.~Agmon.
\newblock {\em Lectures on elliptic boundary value problems}.
\newblock AMS Chelsea Publishing, 1965.

\bibitem{angeles2021nonhyperbolicity}
F.~Angeles.
\newblock Non-hyperbolicity of the inviscid {C}attaneo-{C}hristov system for
  compressible fluid flow in several space dimensions.
\newblock {\em Q. J. Mech. Appl. Math.}, 2022.
\newblock hbac005.

\bibitem{amp}
F.~Angeles, C.~M\'alaga, and R.~Plaza.
\newblock Strict dissipativity of {C}attaneo-{C}hristov systems for
  compressible fluid flow.
\newblock {\em J. Phys. A: Math. Theor.}, 53, 2020.

\bibitem{benzoni}
S.~Benzoni-Gavage and D.~Serre.
\newblock {\em Multidimensional {H}yperbolic {P}artial {D}ifferential
  {E}quations}.
\newblock Oxford {U}niversity {P}ress, 2007.

\bibitem{cat}
C.~Cattaneo.
\newblock Sulla conduzione de calore.
\newblock {\em Atti Semin. Mat. Fis. della Universit\ `{a} di Modena},
  3:83--101, 1948.

\bibitem{chandrasek}
D.~S. Chandrasekharaiah.
\newblock Thermoelasticity with second sound: {A} review.
\newblock {\em Appl. Mech. Rev.}, 39(3):355--376, 1986.

\bibitem{christov}
C.~I. Christov.
\newblock On frame indifferent formulation of the{M}axwell-{C}attaneo model of
  finite-speed heat conduction.
\newblock {\em Mech. Research Comm}, 36:481--486, 2009.

\bibitem{chjo}
C.~I. Christov and P.~M. Jordan.
\newblock Heat conduction paradox involving second-sound propagation in moving
  media.
\newblock {\em Phys. Rev. Lett}, 94:154301, 2005.

\bibitem{daf}
C.~Dafermos.
\newblock {\em Hyperbolic Conservation Laws in Continuum Physics}.
\newblock Springer, 2016.

\bibitem{evans}
L.~C. Evans.
\newblock {\em Partial Differential Equations}.
\newblock American Mathematical Society, 2010.

\bibitem{gonzalez}
O.~Gonz\'alez and A.~Stuart.
\newblock {\em A First Course in Continuum Mechanics}.
\newblock Cambridge University Press, 2010.

\bibitem{zhan}
S.~Han, L.~Zheng, C.~Li, and X.~Zhang.
\newblock Coupled flow and heat transfer in viscoelastic fluid with
  cattaneo-christov heat flux model.
\newblock {\em Appl. Math. Lett}, 38:87--93, 2014.

\bibitem{tuomas}
T.~Hyt{\"o}nen, J.~van Neerven, M.~Veraar, and L.~Weis.
\newblock {\em Analysis in {B}anach spaces}.
\newblock Springer, 2016.

\bibitem{itaya1}
N.~Itaya.
\newblock The {E}xistence and {U}niqueness of the {S}olution of the {E}quations
  {D}escribing {C}ompressible {V}iscous {F}luid {F}low.
\newblock {\em Proc. Japan Acad.}, 46:379--382, 1970.

\bibitem{itaya2}
N.~Itaya.
\newblock On the {C}auchy {P}roblem for the {S}ystem of {F}undamental
  {E}quations {D}escribing the {M}ovement of {C}ompressible {V}iscous {F}luid.
\newblock {\em Kodai Math. Sem. Rep.}, 23:60--120, 1971.

\bibitem{jor}
P.M. Jordan.
\newblock Second-sound phenomena in inviscid, thermally relaxing gases.
\newblock {\em Discrete Contin. Dyn. Syst. Ser. B}, 19:2189--2205, 2014.

\bibitem{kano}
T.~Kano.
\newblock A {N}ecessary {C}ondition for the {W}ell-posedness of the {C}auchy
  {P}roblem for the {F}irst {O}rder {H}yperbolic {S}ystem with {M}ultiple
  {C}haracteristics.
\newblock {\em Publ. RIMS, Kyoto Univ.}, 5:149--164, (1969).

\bibitem{kato1}
T.~Kato.
\newblock Linear evolution equations of hyperbolic type ii.
\newblock {\em J. Math. Soc. Japan}, 25, No. 4:648--666, 1973.

\bibitem{katosym}
T.~Kato.
\newblock The {C}auchy {P}roblem for {Q}uasi-linear {S}ymmetric {H}yperbolic
  {S}ystems.
\newblock {\em Arch. Rational Mech. Anal}, 58:181--205, 1975.

\bibitem{kawa}
S.~Kawashima.
\newblock {\em Systems of a hyperbolic parabolic type with applications to the
  equations of magnetohydrodynamics}.
\newblock PhD thesis, Kyoto University, 1983.

\bibitem{ka}
S.~Kawashima and Y.~Shizuta.
\newblock On the normal form of the symmetric hyperbolic-parabolic systems
  associated with the conservation laws.
\newblock {\em Tohoku Math. J.}, 40:449--464, 1987.

\bibitem{otto}
H.-O. Kreiss and J.~Lorenz.
\newblock {\em Initial-{B}oundary {V}alue {P}roblems and the {N}avier-{S}tokes
  {E}quations}.
\newblock SIAM, 2004.

\bibitem{kawa2}
S.~Kwashima and Y.~Ueda.
\newblock Mathematical entropy and {E}uler-{C}attaneo-{M}axwell system.
\newblock {\em Analysis and Applications}, 14, No. 1:101--143, 2016.

\bibitem{lax}
P.~D. Lax.
\newblock {\em Hyperbolic {S}ystems of {C}onservation {L}aws and the
  {M}athematical {T}heory of {S}hock {W}aves}.
\newblock SIAM, 1973.

\bibitem{lionscom}
P.-L. Lions.
\newblock {\em Mathematical {T}opics in {F}luid {M}echanics {V}olume 2
  {C}ompressible {M}odels}.
\newblock Oxford University Press, 1998.

\bibitem{maj}
A.~Majda.
\newblock {\em Compressible {F}luid {F}low and {S}ystems of {C}onservation
  {L}aws in {S}everal {S}pace {V}ariables}.
\newblock Springer-Verlag, 1984.

\bibitem{matsu}
A.~Matsumura.
\newblock {\em Initial value problems for some quasilinear partial differential
  equations in mathematical physics}.
\newblock PhD thesis, Kyoto University, 1980.

\bibitem{matsuni}
A.~Matsumura and T.~Nishida.
\newblock The initial value problem for the equations of motion of viscous and
  heat-conductive gases.
\newblock {\em J. Math. Kyoto Univ.}, 20-1:67--104, 1980.

\bibitem{mclean}
W.~McLean.
\newblock {\em Strongly elliptic systems and boundary integral equations}.
\newblock American Mathematical Society, 2013.

\bibitem{guy}
G.~M\'{e}tivier.
\newblock Remarks on the well-posedness of the non-linear cauchy problem.
\newblock {\em Contemporary Mathematics}, 368:337--356, 2005.

\bibitem{cherrier}
A.~Milani and P.~Cherrier.
\newblock {\em Linear and quasilinear evolution equations in Hilbert spaces}.
\newblock Graduate studies in mathematics, 2012.

\bibitem{mizo}
S.~Mizohata.
\newblock {\em Theory of partial differential equations}.
\newblock Cambridge University Press, 1973.

\bibitem{amorro}
A.~Morro.
\newblock A {T}hermodynamic {A}pproach to {R}ate {E}quations in {C}ontinuum
  {P}hysics.
\newblock {\em J. Phys. Sci. Appl.}, 7:15--23, (2017).

\bibitem{novo}
A.~Novotn\'y and I.~Stra\v skraba.
\newblock {\em Introduction to the {M}athematical {T}heory of {C}ompressible
  {F}luid {F}low}.
\newblock Oxford University Press, 2004.

\bibitem{qin}
Y.~Qin, Z.~Ma, and X.~Yang.
\newblock Exponential stability for nonlinear thermoelastic equations with
  second sound.
\newblock {\em Quart. Appl. Math.}, 11:2502--2513, 2010.

\bibitem{rack}
R.~Racke.
\newblock {\em Lectures on nonlinear evolution equations}.
\newblock Birkhauser, 2015.

\bibitem{garding}
L.~G\aa rding.
\newblock Probl\`{e}me de {C}auchy pour les syst\`{e}mes quasi-lin\'{e}aires
  d\'{} ordre un strictement hyperboliques.
\newblock {\em Les \'{E}quations aux D\`{e}riv\'{e}es Partielles. \'{E}ditions
  du Centre National de la Recherche Scientifique}, pages 33--40, 1963.

\bibitem{serre}
D.~Serre.
\newblock {\em Conservation Laws 1: Hyperbolicity, Entropies, Shock Waves}.
\newblock Cambridge University Press, 2003.

\bibitem{serre2}
D.~Serre.
\newblock Local existence for viscous system of conservation laws:{$H^{s}$-data
  with $s>1+d/2$}.
\newblock {\em Nonlinear partial differential equations and hyperbolic wave
  phenomena Contemph. Math. Amer. Math. Soc.}, 526:339--358, 2010.

\bibitem{serrevis}
D.~Serre.
\newblock The structure of dissipative viscous systems of conservation laws.
\newblock {\em Physica D}, 239:1381--1386, 2010.

\bibitem{ka1}
Y.~Shizuta and S.~Kawashima.
\newblock Systems of equations of hyperbolic-parabolic type with applications
  to the discrete {B}oltzmann equation.
\newblock {\em Hokkaido Math. J.}, 14:249--275, 1985.

\bibitem{stra}
B.~Straughan.
\newblock Acoustic waves in a {C}attaneo-{C}hristov gas.
\newblock {\em Phys. Lett. A}, 374:2667--2669, 2010.

\bibitem{stra2}
B.~Straughan.
\newblock Thermal convection with the cattaneo-christov model.
\newblock {\em Int. J. Heat Mass Transfer}, 53:95--98, 2010.

\bibitem{straughan2}
B.~Straughan.
\newblock {\em Heat {W}aves}.
\newblock Springer, 2011.

\bibitem{stra4}
B.~Straughan.
\newblock Tipping points in {C}attaneo-{C}hristov thermohaline convection.
\newblock {\em Proc. R. Soc. A}, 467:7--18, 2011.

\bibitem{tadmor}
E.~Tadmor, R.~Miller, and R.~Elliot.
\newblock {\em Continuum {M}echanics and {T}hermodynamics}.
\newblock Cambridge University Press, 2012.

\bibitem{tara}
M.~A. Tarabek.
\newblock On the existence of smooth solutions in one-dimensional
  thermoelasticity with second sound.
\newblock {\em Quart. Appl. Math.}, 50:727--742, 1992.

\bibitem{hudja}
A.~I. Vol'pert and S.~I. Hudjaev.
\newblock On the cauchy problem for composite systems of nonlinear differential
  equations.
\newblock {\em Math.USSR Sbornik}, 16:517--544, 1972.

\bibitem{weidmann}
J.~Weidmann.
\newblock {\em Linear {O}perators in {H}ilbert {S}paces}.
\newblock Springer-Verlag, 1976.

\bibitem{weyl}
H.~Weyl.
\newblock Shock waves in arbitrary fluids.
\newblock {\em Communications on Pure and Applied Mathematics}, vol.
  2:103--122, 1949.

\bibitem{yosida}
K.~Yosida.
\newblock {\em Functional Analysis}.
\newblock Springer-Verlag, 1980.

\end{thebibliography}

\end{document}